\documentclass[11pt, reqno]{amsart}
\usepackage{amssymb}
\usepackage{eucal}
\usepackage{amsmath}
\usepackage{amscd}
\usepackage[dvips]{color}
\usepackage{multicol}
\usepackage[all]{xy}           %xypic macro for latex2.09
\usepackage{graphicx}
\usepackage{color}
\usepackage{colordvi}
\usepackage{xspace}
\usepackage{tikz, tikz-cd}
\usepackage{ifpdf}
\ifpdf
\usepackage[colorlinks,final,backref=page,hyperindex]{hyperref}
\else
\usepackage[colorlinks,final,backref=page,hyperindex,hypertex]{hyperref}
\fi
\numberwithin{equation}{subsection}%numerote les equations selon  les numeros de section%
\begin{document}
	\newcommand {\emptycomment}[1]{} %to remove paragraphs
	
	\baselineskip=14pt
	\newcommand{\nc}{\newcommand}
	\newcommand{\delete}[1]{}
	\nc{\mfootnote}[1]{\footnote{#1}} % Use this to show footnotes
	\nc{\todo}[1]{\tred{To do:} #1}
	
	\delete{
		\nc{\mlabel}[1]{\label{#1}}  % Use this to suppress names
		\nc{\mcite}[1]{\cite{#1}}  % Use this to suppress names
		\nc{\mref}[1]{\ref{#1}}  % Use this to suppress names
		\nc{\meqref}[1]{\ref{#1}} % Use this to suppress names
		\nc{\mbibitem}[1]{\bibitem{#1}} % Use this to show number
	}
	
	%\delete{
	\nc{\mlabel}[1]{\label{#1}  % Use the next two lines to show names
		{\hfill \hspace{1cm}{\bf{{\ }\hfill(#1)}}}}
	\nc{\mcite}[1]{\cite{#1}{{\bf{{\ }(#1)}}}}  % Use this lines to show names
	\nc{\mref}[1]{\ref{#1}{{\bf{{\ }(#1)}}}}  % Use this lines to show names
	\nc{\meqref}[1]{\eqref{#1}{{\bf{{\ }(#1)}}}} % Use this lines to show names
	\nc{\mbibitem}[1]{\bibitem[\bf #1]{#1}} % Use this to show name
	%}
	
	%%%%%%%%%%%%%%%%%%%%%%%% Statements
	\newtheorem{thm}{Theorem}[section]
	\newtheorem{lem}[thm]{Lemma}
	\newtheorem{cor}[thm]{Corollary}
	\newtheorem{pro}[thm]{Proposition}
	\theoremstyle{definition}
	\newtheorem{defi}[thm]{Definition}
	\newtheorem{ex}[thm]{Example}
	\newtheorem{rmk}[thm]{Remark}
	\newtheorem{pdef}[thm]{Proposition-Definition}
	\newtheorem{condition}[thm]{Condition}
	
	\renewcommand{\labelenumi}{{\rm(\alph{enumi})}}
	\renewcommand{\theenumi}{\alph{enumi}}
	
	\nc{\tred}[1]{\textcolor{red}{#1}}
	\nc{\tblue}[1]{\textcolor{blue}{#1}}
	\nc{\tgreen}[1]{\textcolor{green}{#1}}
	\nc{\tpurple}[1]{\textcolor{purple}{#1}}
	\nc{\btred}[1]{\textcolor{red}{\bf #1}}
	\nc{\btblue}[1]{\textcolor{blue}{\bf #1}}
	\nc{\btgreen}[1]{\textcolor{green}{\bf #1}}
	\nc{\btpurple}[1]{\textcolor{purple}{\bf #1}}
	
	\nc{\ld}[1]{\textcolor{blue}{Landry:#1}}
	\nc{\cm}[1]{\textcolor{red}{Chengming:#1}}
	\nc{\li}[1]{\textcolor{yellow}{#1}}
	\nc{\lir}[1]{\textcolor{blue}{Li:#1}}
	
	%%%%%%%%%%%%%% Matrix symbols.
	
	\nc{\twovec}[2]{\left(\begin{array}{c} #1 \\ #2\end{array} \right )}
	\nc{\threevec}[3]{\left(\begin{array}{c} #1 \\ #2 \\ #3 \end{array}\right )}
	\nc{\twomatrix}[4]{\left(\begin{array}{cc} #1 & #2\\ #3 & #4 \end{array} \right)}
	\nc{\threematrix}[9]{{\left(\begin{matrix} #1 & #2 & #3\\ #4 & #5 & #6 \\ #7 & #8 & #9 \end{matrix} \right)}}
	\nc{\twodet}[4]{\left|\begin{array}{cc} #1 & #2\\ #3 & #4 \end{array} \right|}
	
	\nc{\rk}{\mathrm{r}}
	\newcommand{\g}{\mathfrak g}
	\newcommand{\h}{\mathfrak h}
	\newcommand{\pf}{\noindent{$Proof$.}\ }
	\newcommand{\frkg}{\mathfrak g}
	\newcommand{\frkh}{\mathfrak h}
	\newcommand{\Id}{\rm{Id}}
	\newcommand{\gl}{\mathfrak {gl}}
	\newcommand{\ad}{\mathrm{ad}}
	\newcommand{\add}{\frka\frkd}
	\newcommand{\frka}{\mathfrak a}
	\newcommand{\frkb}{\mathfrak b}
	\newcommand{\frkc}{\mathfrak c}
	\newcommand{\frkd}{\mathfrak d}
	\newcommand {\comment}[1]{{\marginpar{*}\scriptsize\textbf{Comments:} #1}}
	%%%%%%%%%%%%%%%%%%%%%%% symbols
	
	\nc{\tforall}{\text{ for all }}
	
	\nc{\svec}[2]{{\tiny\left(\begin{matrix}#1\\
				#2\end{matrix}\right)\,}}  % column vector
	\nc{\ssvec}[2]{{\tiny\left(\begin{matrix}#1\\
				#2\end{matrix}\right)\,}} % subscript column vector
	
	\nc{\typeI}{local cocycle $3$-Lie bialgebra\xspace}
	\nc{\typeIs}{local cocycle $3$-Lie bialgebras\xspace}
	\nc{\typeII}{double construction $3$-Lie bialgebra\xspace}
	\nc{\typeIIs}{double construction $3$-Lie bialgebras\xspace}
	
	\nc{\bia}{{$\mathcal{P}$-bimodule ${\bf k}$-algebra}\xspace}
	\nc{\bias}{{$\mathcal{P}$-bimodule ${\bf k}$-algebras}\xspace}
	
	\nc{\rmi}{{\mathrm{I}}}
	\nc{\rmii}{{\mathrm{II}}}
	\nc{\rmiii}{{\mathrm{III}}}
	\nc{\pr}{{\mathrm{pr}}}
	\newcommand{\huaA}{\mathcal{A}}

	\nc{\OT}{constant $\theta$-}
	\nc{\T}{$\theta$-}
	\nc{\IT}{inverse $\theta$-}

	\nc{\pll}{\beta}
	\nc{\plc}{\epsilon}
	
	\nc{\ass}{{\mathit{Ass}}}
	\nc{\lie}{{\mathit{Lie}}}
	\nc{\comm}{{\mathit{Comm}}}
	\nc{\dend}{{\mathit{Dend}}}
	\nc{\zinb}{{\mathit{Zinb}}}
	\nc{\tdend}{{\mathit{TDend}}}
	\nc{\prelie}{{\mathit{preLie}}}
	\nc{\postlie}{{\mathit{PostLie}}}
	\nc{\quado}{{\mathit{Quad}}}
	\nc{\octo}{{\mathit{Octo}}}
	\nc{\ldend}{{\mathit{ldend}}}
	\nc{\lquad}{{\mathit{LQuad}}}
	
	\nc{\adec}{\check{;}} \nc{\aop}{\alpha}
	\nc{\dftimes}{\widetilde{\otimes}} \nc{\dfl}{\succ} \nc{\dfr}{\prec}
	\nc{\dfc}{\circ} \nc{\dfb}{\bullet} \nc{\dft}{\star}
	\nc{\dfcf}{{\mathbf k}} \nc{\apr}{\ast} \nc{\spr}{\cdot}
	\nc{\twopr}{\circ} \nc{\tspr}{\star} \nc{\sempr}{\ast}
	\nc{\disp}[1]{\displaystyle{#1}}
	\nc{\bin}[2]{ (_{\stackrel{\scs{#1}}{\scs{#2}}})}  %binomial coeff
	\nc{\binc}[2]{ \left (\!\! \begin{array}{c} \scs{#1}\\
			\scs{#2} \end{array}\!\! \right )}  %binomial coeff
	\nc{\bincc}[2]{  \left ( {\scs{#1} \atop
			\vspace{-.5cm}\scs{#2}} \right )}  %binomial coeff
	\nc{\sarray}[2]{\begin{array}{c}#1 \vspace{.1cm}\\ \hline
			\vspace{-.35cm} \\ #2 \end{array}}
	\nc{\bs}{\bar{S}} \nc{\dcup}{\stackrel{\bullet}{\cup}}
	\nc{\dbigcup}{\stackrel{\bullet}{\bigcup}} \nc{\etree}{\big |}
	\nc{\la}{\longrightarrow} \nc{\fe}{\'{e}} \nc{\rar}{\rightarrow}
	\nc{\dar}{\downarrow} \nc{\dap}[1]{\downarrow
		\rlap{$\scriptstyle{#1}$}} \nc{\uap}[1]{\uparrow
		\rlap{$\scriptstyle{#1}$}} \nc{\defeq}{\stackrel{\rm def}{=}}
	\nc{\dis}[1]{\displaystyle{#1}} \nc{\dotcup}{\,
		\displaystyle{\bigcup^\bullet}\ } \nc{\sdotcup}{\tiny{
			\displaystyle{\bigcup^\bullet}\ }} \nc{\hcm}{\ \hat{,}\ }
	\nc{\hcirc}{\hat{\circ}} \nc{\hts}{\hat{\shpr}}
	\nc{\lts}{\stackrel{\leftarrow}{\shpr}}
	\nc{\rts}{\stackrel{\rightarrow}{\shpr}} \nc{\lleft}{[}
	\nc{\lright}{]} \nc{\uni}[1]{\tilde{#1}} \nc{\wor}[1]{\check{#1}}
	\nc{\free}[1]{\bar{#1}} \nc{\den}[1]{\check{#1}} \nc{\lrpa}{\wr}
	\nc{\curlyl}{\left \{ \begin{array}{c} {} \\ {} \end{array}
		\right .  \!\!\!\!\!\!\!}
	\nc{\curlyr}{ \!\!\!\!\!\!\!
		\left . \begin{array}{c} {} \\ {} \end{array}
		\right \} }
	\nc{\leaf}{\ell}       % number of leafs
	\nc{\longmid}{\left | \begin{array}{c} {} \\ {} \end{array}
		\right . \!\!\!\!\!\!\!}
	\nc{\ot}{\otimes} \nc{\sot}{{\scriptstyle{\ot}}}
	\nc{\otm}{\overline{\ot}}
	\nc{\ora}[1]{\stackrel{#1}{\rar}}
	\nc{\ola}[1]{\stackrel{#1}{\la}}%${\Bbb Z}$
	\nc{\pltree}{\calt^\pl}
	\nc{\epltree}{\calt^{\pl,\NC}}
	\nc{\rbpltree}{\calt^r}
	\nc{\scs}[1]{\scriptstyle{#1}} \nc{\mrm}[1]{{\rm #1}}
	\nc{\dirlim}{\displaystyle{\lim_{\longrightarrow}}\,}
	\nc{\invlim}{\displaystyle{\lim_{\longleftarrow}}\,}
	\nc{\mvp}{\vspace{0.5cm}} \nc{\svp}{\vspace{2cm}}
	\nc{\vp}{\vspace{8cm}} \nc{\proofbegin}{\noindent{\bf Proof: }}
	%\nc{\proofbegin}{\begin{proof}} % AMS command
	\nc{\proofend}{$\blacksquare$ \vspace{0.5cm}}
	%\nc{\proofend}{\end{proof}} %AMS command
	\nc{\freerbpl}{{F^{\mathrm RBPL}}}
	\nc{\sha}{{\mbox{\cyr X}}}  %used to be \cyr
	\nc{\ncsha}{{\mbox{\cyr X}^{\mathrm NC}}} \nc{\ncshao}{{\mbox{\cyr
				X}^{\mathrm NC,\,0}}}
	\nc{\shpr}{\diamond}    %Shuffle product
	\nc{\shprm}{\overline{\diamond}}    %Shuffle product
	\nc{\shpro}{\diamond^0}    %Shuffle product
	\nc{\shprr}{\diamond^r}     %product on controlled trees
	\nc{\shpra}{\overline{\diamond}^r}
	\nc{\shpru}{\check{\diamond}} \nc{\catpr}{\diamond_l}
	\nc{\rcatpr}{\diamond_r} \nc{\lapr}{\diamond_a}
	\nc{\sqcupm}{\ot}
	\nc{\lepr}{\diamond_e} \nc{\vep}{\varepsilon} \nc{\labs}{\mid\!}
	\nc{\rabs}{\!\mid} \nc{\hsha}{\widehat{\sha}}
	\nc{\lsha}{\stackrel{\leftarrow}{\sha}}
	\nc{\rsha}{\stackrel{\rightarrow}{\sha}} \nc{\lc}{\lfloor}
	\nc{\rc}{\rfloor}
	\nc{\tpr}{\sqcup}
	\nc{\nctpr}{\vee}
	\nc{\plpr}{\star}
	\nc{\rbplpr}{\bar{\plpr}}
	\nc{\sqmon}[1]{\langle #1\rangle}
	\nc{\forest}{\calf}
	\nc{\altx}{\Lambda_X} \nc{\vecT}{\vec{T}} \nc{\onetree}{\bullet}
	\nc{\Ao}{\check{A}}
	\nc{\seta}{\underline{\Ao}}
	\nc{\deltaa}{\overline{\delta}}
	\nc{\trho}{\tilde{\rho}}
	
	\nc{\rpr}{\circ}
	%\nc{\apr}{\cdot}
	\nc{\dpr}{{\tiny\diamond}}
	\nc{\rprpm}{{\rpr}}
	
	%%%%%%%%%%%%%%%%%%%%% roman fonts, in alphabetic order
	\nc{\mmbox}[1]{\mbox{\ #1\ }} \nc{\ann}{\mrm{ann}}
	\nc{\Aut}{\mrm{Aut}} \nc{\can}{\mrm{can}}
	\nc{\twoalg}{{two-sided algebra}\xspace}
	\nc{\colim}{\mrm{colim}}
	\nc{\Cont}{\mrm{Cont}} \nc{\rchar}{\mrm{char}}
	\nc{\cok}{\mrm{coker}} \nc{\dtf}{{R-{\rm tf}}} \nc{\dtor}{{R-{\rm
				tor}}}
	\renewcommand{\det}{\mrm{det}}
	\nc{\depth}{{\mrm d}}
	\nc{\Div}{{\mrm Div}} \nc{\End}{\mrm{End}} \nc{\Ext}{\mrm{Ext}}
	\nc{\Fil}{\mrm{Fil}} \nc{\Frob}{\mrm{Frob}} \nc{\Gal}{\mrm{Gal}}
	\nc{\GL}{\mrm{GL}} \nc{\Hom}{\mrm{Hom}} \nc{\hsr}{\mrm{H}}
	\nc{\hpol}{\mrm{HP}} \nc{\id}{\mrm{id}} \nc{\im}{\mrm{im}}
	\nc{\incl}{\mrm{incl}} \nc{\length}{\mrm{length}}
	\nc{\LR}{\mrm{LR}} \nc{\mchar}{\rm char} \nc{\NC}{\mrm{NC}}
	\nc{\mpart}{\mrm{part}} \nc{\pl}{\mrm{PL}}
	\nc{\ql}{{\QQ_\ell}} \nc{\qp}{{\QQ_p}}
	\nc{\rank}{\mrm{rank}} \nc{\rba}{\rm{RBA }} \nc{\rbas}{\rm{RBAs }}
	\nc{\rbpl}{\mrm{RBPL}}
	\nc{\rbw}{\rm{RBW }} \nc{\rbws}{\rm{RBWs }} \nc{\rcot}{\mrm{cot}}
	\nc{\rest}{\rm{controlled}\xspace}
	\nc{\rdef}{\mrm{def}} \nc{\rdiv}{{\rm div}} \nc{\rtf}{{\rm tf}}
	\nc{\rtor}{{\rm tor}} \nc{\res}{\mrm{res}} \nc{\SL}{\mrm{SL}}
	\nc{\Spec}{\mrm{Spec}} \nc{\tor}{\mrm{tor}} \nc{\Tr}{\mrm{Tr}}
	\nc{\mtr}{\mrm{sk}}
	
	%%%%%%%%%%%%%%%%%% bold face
	\nc{\ab}{\mathbf{Ab}} \nc{\Alg}{\mathbf{Alg}}
	\nc{\Algo}{\mathbf{Alg}^0} \nc{\Bax}{\mathbf{Bax}}
	\nc{\Baxo}{\mathbf{Bax}^0} \nc{\RB}{\mathbf{RB}}
	\nc{\RBo}{\mathbf{RB}^0} \nc{\BRB}{\mathbf{RB}}
	\nc{\Dend}{\mathbf{DD}} \nc{\bfk}{{\bf k}} \nc{\bfone}{{\bf 1}}
	\nc{\base}[1]{{a_{#1}}} \nc{\detail}{\marginpar{\bf More detail}
		\noindent{\bf Need more detail!}
		\svp}
	\nc{\Diff}{\mathbf{Diff}} \nc{\gap}{\marginpar{\bf
			Incomplete}\noindent{\bf Incomplete!!}
		\svp}
	\nc{\FMod}{\mathbf{FMod}} \nc{\mset}{\mathbf{MSet}}
	\nc{\rb}{\mathrm{RB}} \nc{\Int}{\mathbf{Int}}
	\nc{\Mon}{\mathbf{Mon}}
	%\nc{\remark}{\noindent{\bf Remark: }}
	\nc{\remarks}{\noindent{\bf Remarks: }}
	\nc{\OS}{\mathbf{OS}} %free operated semigroup
	\nc{\Rep}{\mathbf{Rep}}
	\nc{\Rings}{\mathbf{Rings}} \nc{\Sets}{\mathbf{Sets}}
	\nc{\DT}{\mathbf{DT}}
	
	%%%%%%%%%%%%%%%%%%%Bbb fonts
	\nc{\BA}{{\mathbb A}} \nc{\CC}{{\mathbb C}} \nc{\DD}{{\mathbb D}}
	\nc{\EE}{{\mathbb E}} \nc{\FF}{{\mathbb F}} \nc{\GG}{{\mathbb G}}
	\nc{\HH}{{\mathbb H}} \nc{\LL}{{\mathbb L}} \nc{\NN}{{\mathbb N}}
	\nc{\QQ}{{\mathbb Q}} \nc{\RR}{{\mathbb R}} \nc{\BS}{{\mathbb{S}}} \nc{\TT}{{\mathbb T}}
	\nc{\VV}{{\mathbb V}} \nc{\ZZ}{{\mathbb Z}}

	%%%%%%%%%%%%%%%%%%% cal fonts
	
	\nc{\calao}{{\mathcal A}} \nc{\cala}{{\mathcal A}}
	\nc{\calc}{{\mathcal C}} \nc{\cald}{{\mathcal D}}
	\nc{\cale}{{\mathcal E}} \nc{\calf}{{\mathcal F}}
	\nc{\calfr}{{{\mathcal F}^{\,r}}} \nc{\calfo}{{\mathcal F}^0}
	\nc{\calfro}{{\mathcal F}^{\,r,0}} \nc{\oF}{\overline{F}}
	\nc{\calg}{{\mathcal G}} \nc{\calh}{{\mathcal H}}
	\nc{\cali}{{\mathcal I}} \nc{\calj}{{\mathcal J}}
	\nc{\call}{{\mathcal L}} \nc{\calm}{{\mathcal M}}
	\nc{\caln}{{\mathcal N}} \nc{\calo}{{\mathcal O}}
	\nc{\calp}{{\mathcal P}} \nc{\calq}{{\mathcal Q}} \nc{\calr}{{\mathcal R}}
	\nc{\calt}{{\mathcal T}} \nc{\caltr}{{\mathcal T}^{\,r}}
	\nc{\calu}{{\mathcal U}} \nc{\calv}{{\mathcal V}}
	\nc{\calw}{{\mathcal W}} \nc{\calx}{{\mathcal X}}
	\nc{\CA}{\mathcal{A}}

	%%%%%%%%%%%%%%%%%%  frak fonts
	\nc{\fraka}{{\mathfrak a}} \nc{\frakB}{{\mathfrak B}}
	\nc{\frakb}{{\mathfrak b}} \nc{\frakd}{{\mathfrak d}}
	\nc{\oD}{\overline{D}}
	\nc{\frakF}{{\mathfrak F}} \nc{\frakg}{{\mathfrak g}}
	\nc{\frakm}{{\mathfrak m}} \nc{\frakM}{{\mathfrak M}}
	\nc{\frakMo}{{\mathfrak M}^0} \nc{\frakp}{{\mathfrak p}}
	\nc{\frakS}{{\mathfrak S}} \nc{\frakSo}{{\mathfrak S}^0}
	\nc{\fraks}{{\mathfrak s}} \nc{\os}{\overline{\fraks}}
	\nc{\frakT}{{\mathfrak T}}
	\nc{\oT}{\overline{T}}
	%\nc{\frakx}{{\mathfrak x}}
	\nc{\frakX}{{\mathfrak X}} \nc{\frakXo}{{\mathfrak X}^0}
	\nc{\frakx}{{\mathbf x}}
	%\nc{\frakTxo}{{\frakTx}^0}
	\nc{\frakTx}{\frakT}      %All rooted trees, correspond to \ncsha(X)
	\nc{\frakTa}{\frakT^a}        % rooted trees for \ncsha(A)
	\nc{\frakTxo}{\frakTx^0}   % rooted trees for \ncshao(X)
	\nc{\caltao}{\calt^{a,0}}   % rooted trees for \ncshao(A)
	\nc{\ox}{\overline{\frakx}} \nc{\fraky}{{\mathfrak y}}
	\nc{\frakz}{{\mathfrak z}} \nc{\oX}{\overline{X}}
	
	\font\cyr=wncyr10
	
	\nc{\al}{\alpha}
	\nc{\lam}{\lambda}
	\nc{\lr}{\longrightarrow}
	\newcommand{\K}{\mathbb {K}}
	\newcommand{\A}{\rm A}
	
	%%%%%%%%%%%%%%%%%%%%%%%%%%%%%%%%%%%%%%%%%%%%%%%%%%%%%%%%%%%%%%%%%%
\title[Pre-anti-flexible bialgebrask]{Pre-anti-flexible bialgebras}
\author[Mafoya Landry Dassoundo]{Mafoya Landry Dassoundo}
\address[]{Chern Institute of Mathematics \& LPMC, 
Nankai University, Tianjin 300071, China} \email{dassoundo@yahoo.com}
	
	\begin{abstract}
In this paper, we derive  
pre-anti-flexible algebras structures in term of 
zero weight's Rota-Baxter operators defined on anti-flexible algebras, 
view pre-anti-flexible algebras  as a splitting 
of anti-flexible algebras, introduce the notion of 
pre-anti-flexible bialgebras and establish equivalences among 
matched pair of  anti-flexible algebras, matched pair of pre-anti-flexible algebras 
and pre-anti-flexible bialgebras. 
Investigation on special class of pre-anti-flexible bialgebras  
leads to the establishment  of 
the pre-anti-flexible Yang-Baxter equation.
Both dual bimodules of pre-anti-flexible algebras and 
dendriform algebras have the same shape and this induces
that both  pre-anti-flexible Yang-Baxter equation and 
$\mathcal{D}$-equation are identical. 
Symmetric solution of pre-anti-flexible
Yang-Baxter equation gives a pre-anti-flexible bialgebra. 
Finally, we recall and  link  $\mathcal{O}$-operators of 
anti-flexible algebras   to  bimodules of  pre-anti-flexible algebras and  built symmetric
solutions of anti-flexible Yang-Baxter equation.
\end{abstract}

	\subjclass[2010]{17A20, 17D25,  16T10, 16T15,  17B38,  16T25}
	\keywords{(pre-)anti-flexible algebra, dendriform algebras, 
		(pre-)anti-flexible bialgebra,  Yang-Baxter equation, Rota-Baxter operator}
	\maketitle
	
	\tableofcontents
	
	\numberwithin{equation}{section}
	\tableofcontents
	\numberwithin{equation}{section}
	\allowdisplaybreaks
	
	\section{Introduction and Preliminaries}
	The notion of pre-anti-flexible algebras are introduced in \cite{DBH3} to derive 
	the $\mathcal{O}$-operators of anti-flexible algebras which allow to built the 
	skew-symmetric solutions of  anti-flexible Yang-Baxter equation.
	Pre-anti-flexible algebras are closed dendriform algebras 
	which are introduced by J.-L. Loday (\cite{Loday}).
	Besides, pre-anti-flexible algebras can be considered as a  
	generalization of dendriform algebras and 
	as very well known and widespread in the literature, dendriform algebras
	are also induced by the well known notion of Rota-Baxter operators of weight zero
	(\cite{Aguiar1}) which are introduced around 1960’s by  
	G. Baxter (\cite{Baxter}) and G.-C. Rota (\cite{Rota}). Recently, 
	significant advances contributions on Rota-Baxter operators  and related
	applications are summarized in \cite{Guo} and the references therein.
	Since dendriform algebras are closed to associative algebras,
	pre-anti-flexible algebras are strongly linked to anti-flexible algebras
	(also known as center-symmetric algebras) 
	and themselves associated with Lie algebras (\cite{Hounkonnou_D_CSA}) and other 
	similar deduction are derived and readable 
	in the following diagram  which summarizes underlying relations 
	among pre-anti-flexible algebras (PAFA), 
	anti-flexible algebras (AFA), 
	Lie algebras (LA), associative algebras (AA), dendriform algebras (DA), 
	and finally with what we call
	derived Lie-admissible algebra  of a given pre-anti-flexible algebra (DLAd-PAFA)  
	\begin{eqnarray*}
		\xymatrix{
			&{\bf PAFA\;} (A, \prec, \succ) \ar[dl]_-{\mbox{C2}}\ar[d]^-{\mbox{C6}}\ar[rr]^-{\mbox{C1}}
			&&{\bf AFA\;}(A, \cdot) \ar[d]^-{\mbox{C5}}\\
			{\bf DA\;} (A, \prec_{_1}, \succ_{_1})\ar[dr]_-{\mbox{C3}}
			& {\bf{DLAd-PAFA}\;}(A, \circ )\ar[rr]^-{\mbox{C7}}&&{\bf LA\;} (A, [,]), \\	
			&{\bf AA\;}(A, \cdot_{_1} )\ar[rru]_-{\mbox{C4}}
		}
	\end{eqnarray*}
	where, for any $x,y, z\in A$, the condition C1  translates 
	$x\cdot y=x\prec y+x\succ y$, 
	C2 means $\prec_{_1}:=\prec; \succ_{_1}:=\succ$ and  
	$(x,y,z)_m=0, (x,y,z)_l=0, (x,y,z)_r=0$ which are given by 
	Eqs.~\eqref{eq:biasso} (trivial 
	pre-anti-flexible algebras are dendriform algebras), C3 describes
	$x\cdot_{_1} y= x\prec_{_1} y+x\succ_{_1}y$, C4 expresses the commutator
	$[x,y]=x\cdot_{_1}y-y\cdot_{_1} x$, C5 translates 
	$[x,y]=x\cdot y-y\cdot x$, C6 describes $x\circ y=x\succ y-y\prec x$ and finally 
	C7 expresses $[x,y]=x\circ y-y\circ x$. It is also useful to recall that any
	associative algebra is a trivial anti-flexible algebra. 
	Thus, anti-flexible algebras generalize associative algebras. 
	
	Notice here that, although the goal of this paper  
	is  not to construct a cohomology theory for 
	anti-flexible algebras,  cohomology of associative and Lie algebras, 
	and other algebras are well known. Unfortunately, despite their links to associative algebras and to Lie algebras described above,
 anti-flexible	algebras and pre-anti-flexible algebras  lack a suitable cohomology theory 
	which can justify certain shortcomings 
	on anti-flexible algebras such as for instance among many other,
	coboundary anti-flexible algebras and those of pre-anti-flexible algebras
	which are well known on associative and Lie algebras. 
	That said, all is not lost for avoid talking about some notions on the cohomology  
	of anti-flexible and pre-anti-flexible algebras.
	As proof, the analogue to the classical Yang-Baxter 
	equation on  Lie algebras derived 
	by Drinfeld (\cite{Drinfeld}), that of associative Yang-Baxter 
	on associative algebras (\cite{Aguiar, Bai_Double})
	as well as  $\mathcal{D}$-equation of dendriform algebras (\cite{Bai_Double}), 
	anti-flexible Yang-Baxter equation  recovered in  special consideration of 
	anti-flexible algebras (\cite{DBH3}) as well as 
	pre-anti-flexible Yang-Baxter equation on special class of 
	pre-anti-flexible algebras.  
	Furthermore, alternative $\mathcal{D}$-bialgebras are also
	provided and described on Cayley-Dickson matrix algebras (\cite{Gon}).
	Besides, by keeping  the spirit that dendriform algebras are  viewed as 
	splitted associative algebras (\cite{Bai_spit}), 
	pre-anti-flexible algebras are regarded as splitted anti-flexible algebras and 
	more generally, operadic definition for the notion of
	splitting  algebra structures are introduced
	and provided some equivalence with Manin products of operads in 
	quadratic operads  (\cite{Pei_Bai_Guo}).
	
	Before straight although to the goal of this paper, recall some 
	fundamentals which will be necessary throughout our 
	concern on pre-anti-flexible bialgebras. 
	In this paper, all considered vector spaces are finite-dimensional 
	over a base field  $\mathbb{F}$ whose characteristic is $0$. 
	Many derived results still hold regardless the  dimension of vector spaces
	on which they are stated. For this purpose, we mean by anti-flexible algebra (\cite{Hounkonnou_D_CSA}), 
	a couple $(A, \ast)$ where $A$ is a vector space equipped with a linear product
	"$\ast$" such that for any $x,y, z\in A$, $(x,y,z)=(z,y,x)$, where the triple is defined as
	$(x,y,z):=(x\ast y)\ast z-x\ast (y\ast z)$.
	If in addition $A$ is equipped with two linear maps $l,r:A\rightarrow \End(V)$, where $V$
	is a vector space, such that for any $x,y\in A$ 
	\begin{subequations}
		\begin{eqnarray}\label{eqbimodule1}
			l{(x\ast y)}-l(x)l(y)=r(x)r(y)-r({y \ast x}),
		\end{eqnarray}
		\begin{eqnarray}\label{eqbimodule2}
			\left[l(x), r(y)\right]= \left[l(y), r(x)\right],
		\end{eqnarray}
	\end{subequations}
	then the triple $(l,r,V)$ is called bimodule of $(A, \ast)$.
	\begin{thm}\label{thm_2} \cite{Hounkonnou_D_CSA}
		Let $(  A, \ast)$ and $(B, \circ)$ be two anti-flexible algebras. 
		Suppose that $(l_{  A}, r_{  A}, B)$ and $(l_{B}, r_{B},   A)$
		are bimodules of $(A, \ast)$ and $(B, \circ)$, respectively, where 
		$l_{  A}, r_{  A}:A\rightarrow \End(B)$ and 
		$l_{B}, r_{B}:B\rightarrow \End(A)$ are four linear maps and 
		obeying the relations,
		for any  $x, y \in A$ and for any $a, b \in B,$
		\begin{subequations}
			\begin{eqnarray}\label{eqq1}
				l_{B}(a)(x\ast y) +r_{B}(a)(y\ast x)-r_{B}(l_{  A}(x)a)y-  
				y\ast(r_{B}(a)x) -l_{B}(r_{  A}(x)a)y -  (l_{B}(a)x)\ast y  = 0, 
			\end{eqnarray}
			\begin{eqnarray}\label{eqq2}
				l_{  A}(x)(a\circ b)	+r_{  A}(x)(b\circ a)-r_{  A}(l_{B}(a)x)b- 
				b\circ (r_{  A}(x)a)+l_{  A}(r_{B}(a)x)b	- (l_{  A}(x)a)\circ b=0,
			\end{eqnarray}
			\begin{eqnarray}\label{eqq3}
				\begin{array}{lll}
					y\ast (l_{B}(a)x)+(r_{B}(a)x)\ast y - (r_{B}(a)y)\ast x-l_{B}(l_{  A}(y)a)x+ \cr 
					r_{B}(r_{  A}(x)a)y+l_{B}(l_{  A}(x)a)y  -x\ast (l_{B}(a)y)-r_{B}(r_{  A}(y)a)x=0,
				\end{array}
			\end{eqnarray}
			\begin{eqnarray}\label{eqq4}
				\begin{array}{lll}
					b \circ (l_{  A}(x)a)+(r_{  A}(x)a)\circ b -(r_{  A}(x)b)\circ a-
					l_{  A}(l_{B}(b)x)a+\cr 
					r_{  A}(r_{B}(a)x)b+l_{  A}(l_{B}(a)x)b  -
					a\circ (l_{  A}(x)b)  -r_{  A}(r_{B}(b)x)a=0. 
				\end{array}
			\end{eqnarray}
		\end{subequations}
		Then, there is an anti-flexible algebra structure on $  A \oplus B$ 
		given by for any $x,y\in A$ and any $a, b\in B$
		\begin{eqnarray*}
			(x+a)\star (y+b)= (x \ast y + l_{B}(a)y+r_{B}(b)x)+ (a \circ b + l_{A}(x)b+r_{A}(y)a).
		\end{eqnarray*}
	\end{thm}
	\begin{defi}\cite{DBH3}
		A pre-anti-flexible algebra is a vector space $  A$ equipped with two bilinear products 
		$\prec, \succ:  A\otimes  A \rightarrow   A$ 
		satisfying the following relations
		\begin{subequations}
			\begin{eqnarray}\label{eq_pre_antiflexible_1}
				(x,y,z)_{_m}=(z,y,x)_{_m},  \; \; \; \forall x,y,z\in   A,
			\end{eqnarray}
			\begin{eqnarray}\label{eq_pre_antiflexible_2}
				(x,y,z)_{_l}=(z,y,x)_{_r},  \; \; \; \forall x,y,z\in   A,
			\end{eqnarray}
		\end{subequations}
		where for any $x,y, z\in   A$, 	
		\begin{subequations}\label{eq:biasso}
			\begin{eqnarray}\label{eq_biasso_m}
				(x,y,z)_{_m}:=(x \succ y) \prec z-x \succ (y \prec z),
			\end{eqnarray}
			\begin{eqnarray}\label{eq_biasso_l}
				(x,y,z)_{_l}:=(x\cdot y)\succ z-x\succ (y\succ z),
			\end{eqnarray}
			\begin{eqnarray}\label{eq_biasso_r}
				(x,y,z)_{_r}:=(x\prec y)\prec z-x\prec (y\cdot z),
			\end{eqnarray}
		\end{subequations}
		with $x\cdot y=x\prec y+x\succ y$.
		
		Equivalently, a pre-anti-flexible algebra is a triple 
		$(A, \prec, \succ)$ such that  $A$ is a vector space and
		$\prec, \succ: A\times A \rightarrow A$ are 
		two linear maps  satisfying the relations for any 
		$x,y,z\in A$
		\begin{subequations}
			\begin{eqnarray}\label{eq:pre-antiflexible1}
				(x\succ y) \prec z-x\succ (y\prec z)=(z\succ y)\prec x-
				z\succ(y\prec x),
			\end{eqnarray}
			\begin{eqnarray}\label{eq:pre-antiflexible2}
				(x\prec y +x\succ y )\succ z-x\succ (y\succ z)=
				(z\prec y)\prec x-z\prec (y\prec x+y\succ x).
			\end{eqnarray}
		\end{subequations}
	\end{defi}
	\begin{ex}
		For a given associative $(A, \ast)$, setting for any $x,y\in A$,
		$x\succ y=x\ast y$ (or $x\succ y=y\ast x$) and 
		$x\prec y=0$, then $(A, \prec, \succ)$ is 
		a pre-anti-flexible algebra. Similarly, $(A, \prec, \succ)$ is a 
		pre-anti-flexible algebra by setting for any $x,y\in A$,
		$x\prec y=x\ast y$ (or $x\prec y=y\ast x$) and $x\succ y=0$.
	\end{ex}
	\begin{rmk}\label{rmk_1}
		Let $(  A, \prec, \succ)$ be a pre-anti-flexible algebra.
		\begin{enumerate}
			\item\label{rmk_flex}
			It is well known that the  couple $(  A, \cdot)$ is an 
			anti-flexible algebra (\cite{DBH3}), where 
			for any $x,y\in   A$, $x\cdot y=x\prec y+x\succ y$, i.e.  
			$(x,y,z)=(z,y,x)$, 
			and we will denote this  anti-flexible algebra by $aF(A)$ 
			call it underlying anti-flexible algebra of the pre-anti-flexible algebra
			$(A, \prec, \succ)$.
			\item
			As we can see, if both  sides of  equality  in the  
			Eqs.~\eqref{eq:pre-antiflexible1} and \eqref{eq:pre-antiflexible2} are zero
			i.e. for any $ x,y,z\in   A, (x,y,z)_{_m}=0$, $(x,y,z)_{_l}=0$ 
			and $(x,y,z)_{_r}=0$, then $(A, \prec, \succ)$  
			is a dendriform algebra i.e. the couple $(A, \prec, \succ)$ such that 
			for any $x,y,z\in A,$
			\begin{eqnarray*}
				&&(x\succ y) \prec z-x\succ (y\prec z)=0,\cr
				&&(x\prec y +x\succ y )\succ z-x\succ (y\succ z)=0,\cr
				&&(x\prec y)\prec z-x\prec (y\prec z+y\succ z)=0,
			\end{eqnarray*}
			introduced by J.-L.~Loday (\cite{Loday}). 
			Clearly,  dendriform algebra is a pre-anti-flexible algebra
			and then pre-anti-flexible algebras can  viewed as a generalization 
			of dendriform algebras. 
		\end{enumerate}
	\end{rmk}
	Throughout this paper, if there is  no other consideration,  for any
	$x,y\in A$, $x\ast y$ or $x\cdot y$ will simply written by $xy$. 
	Furthermore, underlying  anti-flexible algebra structure 
	of a given pre-anti-flexible algebra will generally denote by "$\cdot$" as 
	in Remark~\ref{rmk_1}~\eqref{rmk_flex}.
	Moreover, for a given pre-anti-flexible algebra 
	$(A, \prec, \succ)$,  we denote  by 
	$L_{\prec}, R_{\prec}: A\rightarrow \End(A)$ the left  
	and right multiplication operators, respectively, on $(A, \prec)$
	and similarly by $L_{\succ}, R_{\succ}: A\rightarrow \End(A)$ these on 
	$(A, \succ)$ which are  defined as  
	$\forall x,y \in A,$
	\begin{equation*}
		L_{_\prec}(x)y=x\prec y,\;\; R_{_\prec}(x)y=y\prec x,\;\;
		L_{_\succ}(x)y=x\succ y,\;\; R_{_\succ}(x)y=y\succ x.
	\end{equation*}
	\begin{thm}\cite{DBH3}\label{Theo_existance_pre_anti_flexible}
		Let $(A, \ast)$ be an anti-flexible algebra equipped with a 
		non-degenerate  bilinear form $\omega: A\otimes A\rightarrow \mathbb{F}$ satisfying 
		\begin{equation}\label{eq:simplectic_form}
			\omega(x\ast y, z)+\omega(y\ast z, x)+\omega(z\ast x, y)=0,
			\quad \forall x,y\in A.
		\end{equation}
		Then there is a pre-anti-flexible algebra structure 
		"$\prec, \succ$" defined on $  A$  
		satisfying the following relation, for any $x,y, z\in   A$,
		\begin{eqnarray}\label{eq_useful1}
			\omega(x\prec y,z)=\omega(x, y\ast z), \quad 
			\omega(x\succ y, z)=\omega(y, z\ast x).
		\end{eqnarray}
	\end{thm}
	To make this introductory section  devoted to 
	fundamentals necessary for addressed  main issues 
	as short as possible, we end it by outlined the content 
	of this article  as follows.
	In section~\ref{section1}, we prove and generalize that  Rota-Baxter 
	operators on an anti-flexible algebras
	induce  pre-anti-flexible algebras.  
	In Section~\ref{section2},
	we study bimodules and matched pairs of pre-anti-flexible algebras. 
	Precisely, we derive the dual bimodules of 
	bimodules of an pre-anti-flexible algebras.
	In Section~\ref{section3}, 
	we establish the equivalences among matched pair of the underlying 
	anti-flexible algebras of pre-anti-flexible algebras,
	matched pair of pre-anti-flexibles algebras,  
	and to   pre-anti-flexible bialgebras.
	In Section~\ref{section4}, 
	we rule on a special class of pre-anti-flexible bialgebras 
	which lead to the introduction
	of the pre-anti-flexible Yang-Baxter equation. 
	A symmetric solution of the  pre-anti-flexible Yang-Baxter
	equation gives a such  pre-anti-flexible bialgebra. 
	Finally in Section~\ref{section5}, 
	we recall the notions of  $\mathcal{O}$-operators of  anti-flexible algebras
	and intertwine this notion to that  of bimodules of  pre-anti-flexible algebras
	and use the relationships among
	them to  provide the
	symmetric solutions of pre-anti-flexible Yang-Baxter equation in  
	pre-anti-flexible bialgebras.
	\section{Rota-Baxter operators and pre-anti-flexible algebras}\label{section1}
	In this section, we are going to express pre-anti-flexible algebras
	in terms of  Rota-Baxter operator of weight zero defined on anti-flexible algebras.
	\begin{defi}
		Let $(A,\ast)$ be an anti-flexible algebra. 
		A  Rota-Baxter operator ($\mathrm{R_B}$) of weight zero on $A$ 
		is a linear operator $\mathrm{R_B}:A\rightarrow A$ satisfying 
		\begin{eqnarray}\label{eq:Rota-Baxter}
			\mathrm{R_B}(x)\ast \mathrm{R_B}(y)=
			\mathrm{R_B}(x\ast \mathrm{R_B}(y)+\mathrm{R_B}(x)\ast y), \;\forall x,y\in A.
		\end{eqnarray} 
	\end{defi}
	\begin{thm}
		Let  $(A, \ast)$ be an anti-flexible algebra equipped with a linear map 
		$\alpha:A\rightarrow A$. Consider the bilinear products "$\prec, \succ$"
		given by for any $x,y\in A,$
		\begin{eqnarray}\label{eq:pre-anti-flexible-Rota-Baxter}
			x\succ y=\alpha(x)\ast y,\;\quad 
			x\prec y=x\ast \alpha(y).
		\end{eqnarray}
		Then the triple $(A, \prec, \succ)$ is a pre-anti-flexible algebra
		if and only if for any $x,y,z\in A,$
		\begin{eqnarray}\label{eq:identity-RB-pre-anti-flexible}
			(\alpha(x)\ast \alpha(y)-\alpha(x\ast \alpha(y)+\alpha(x)\ast y))\ast z+
			z\ast (\alpha(y)\ast \alpha(x)-\alpha(y\ast \alpha(x)+\alpha(y)\ast x))=0.
		\end{eqnarray}
	\end{thm}
	\begin{proof}
		For any $x,y,z\in A$,  we have
		\begin{eqnarray*}
			(x,y,z)_m=(\alpha(z), y, \alpha(x))=(\alpha(x), y, \alpha(z))=(z, y, x)_m.
		\end{eqnarray*}
		Thus the bilinear products given by Eq.~\eqref{eq:pre-anti-flexible-Rota-Baxter} satisfy Eq.~\eqref{eq_pre_antiflexible_1}.
		
		Besides, we have 
		\begin{eqnarray*}
			(x,y,z)_l=-(\alpha(x)\ast \alpha(y)-\alpha(x\ast \alpha(y)+\alpha(x)\ast y))\ast z+(z, \alpha(y), \alpha(x))
		\end{eqnarray*}
		and 
		\begin{eqnarray*}
			(z, y, x)_r=z\ast (\alpha(y)\ast \alpha(x)-
			\alpha(y\ast \alpha(x)+\alpha(y)\ast x))+(\alpha(x), \alpha(y), z).
		\end{eqnarray*}
		Therefore, Eq.~\eqref{eq_pre_antiflexible_2} is equivalent to 
		Eq.~\eqref{eq:identity-RB-pre-anti-flexible}.
	\end{proof}
	It is obvious to remark that the previous theorem generalizes
	the following corollary which links Rota-Baxter operators to 
	pre-anti-flexible algebras.
	\begin{cor}
		If $\alpha:A\rightarrow A$ is a Rota-Baxter operator
		defined on an anti-flexible algebra $(A, \ast)$
		i.e. $\alpha$ is a linear map satisfies Eq.~\eqref{eq:Rota-Baxter}, then there is 
		a pre-anti-flexible algebra structure "$\prec, \succ$" on $A$ given by 
		Eq.~\eqref{eq:pre-anti-flexible-Rota-Baxter}. 
	\end{cor}
	
	\section{Bimodules and matched pair of  pre-anti-flexible algebras}\label{section2}
	In this section, we provide bimodules and dual bimodules of 
	pre-anti-flexible algebras. We  also
	introduce matched pair of pre-anti-flexible algebras and 
	equivalently link them to a matched pair of their  
	underlying anti-flexible algebras. Finally, we define anti-flexible bialgebras 
	and establish related identities.
	\begin{defi}
		Let $(  A, \prec, \succ)$ be a pre-anti-flexible algebra and $V$ be a vector space. 
		Consider the four linear maps
		$l_{_\succ}, l_{_\prec},r_{_\succ}, r_{_\prec}:  A\rightarrow \End(V) $. 
		The quintuple $(l_{_\succ},r_{_\succ},  l_{_\prec}, r_{_\prec}, V)$
		is called a bimodule of $(  A, \prec, \succ)$ if for any $x,y\in   A$,
		\begin{subequations}
			\begin{eqnarray}\label{eq_bimodule_pre_anti_flexible1}
				[r_{_\prec}(x), l_{_\succ}(y)]=[r_{\prec}(y), l_{_\succ}(x)],
			\end{eqnarray}
			\begin{eqnarray}\label{eq_bimodule_pre_anti_flexible2}
				l_{_\prec}(x\succ y)-l_{_\succ}(x)l_{_\prec}(y)=
				r_{_\prec}(x)r_{_\succ}(y)-r_{_\succ}(y\prec x),
			\end{eqnarray}
			\begin{eqnarray}\label{eq_bimodule_pre_anti_flexible3}
				l_{_\succ}(x\cdot y)-l_{_\succ}(x)l_{_\succ}(y)=
				r_{_\prec}(x)r_{_\prec}(y)-r_{_\prec}(y\cdot x),
			\end{eqnarray}
			\begin{eqnarray}\label{eq_bimodule_pre_anti_flexible4}
				r_{_\succ}(x)l_{\cdot}(y)-l_{_\succ}(y)r_{_\succ}(x)=
				r_{_\prec}(y)l_{_\prec}(x)-l_{_\prec}(x)r_{\cdot}(y),
			\end{eqnarray}
			\begin{eqnarray}\label{eq_bimodule_pre_anti_flexible5}
				r_{_\succ}(x)r_{\cdot}(y)-r_{_\succ}(y\succ  x)=
				l_{_\prec}(x\prec y)-l_{_\prec}(x)l_{\cdot}(y),
			\end{eqnarray}
		\end{subequations}
		where, $x\cdot y=x\prec y+x\succ y, l_{_\cdot}=l_{_\prec}+l_{_\succ}$ and 
		$r_{_\cdot}=r_{_\prec}+r_{_\succ}$.
	\end{defi}
	\begin{pro}
		Let $(  A, \prec, \succ)$ be a pre-anti-flexible algebra and $V$ be a vector space. 
		Consider the four linear maps
		$l_{_\succ}, l_{_\prec},r_{_\succ}, r_{_\prec}:   A\rightarrow \End(V) $. 
		The quintuple $(l_{_\succ},r_{_\succ},  l_{_\prec}, r_{_\prec}, V)$
		is called a bimodule of $(  A, \prec, \succ)$ if and only if there is a 
		pre-anti-flexible algebra defined on $  A\oplus V$ by for any 
		$x, y\in   A$, $u, v\in V$, 
		\begin{eqnarray}\label{eq_pre_anti_flexible_bimodule}
			\begin{array}{ccc}
				(x+u)\prec'(y+v)=x\prec y+l_{_\prec}(x)v+r_{_\prec}(y)u,\cr
				(x+u)\succ'(y+v)=x\succ y+l_{_\succ}(x)v+r_{_\succ}(y)u,
			\end{array}
		\end{eqnarray}
		and 
		$ 
		(x+u)\cdot'(y+v)=(x+u)\prec'(y+v)+(x+u)\succ'(y+v).
		$
	\end{pro}
	\begin{proof}
		According to  Eq.~\eqref{eq_pre_anti_flexible_bimodule} 
		we have for any $x,y,z\in   A$ and for any $u,v,w\in V$, 
		\begin{eqnarray*}
			(x+u, y+v, z+w)_{_m}&=&((x+u)\succ'(y+v) )\prec' (z+w)-(x+u)\succ'((y+v)\prec'(z+w))\cr 
			&=&(x,y,z)_{_m}+\{l_{_\prec}(x\succ y)-l_{_\succ}(x)(l_{_\prec}(y))\}w\cr 
			&+&\{r_{_\prec}(z)l_{_\succ}(x)-l_{_\succ}(x)r_{_\prec}(z)\}v+
			\{r_{_\prec}(z)r_{_\succ}(y)-r_{_\prec}(y\prec z)\}u.
		\end{eqnarray*}
		\begin{eqnarray*}
			(x+u, y+v, z+w)_{_l}&=&((x+u)\cdot'(y+v) )\succ' (z+w)-(x+u)\succ'((y+v)\succ'(z+w))\cr 
			&=&(x,y,z)_{_l}+\{ l_{\cdot}(x\cdot y)-l_{_\succ}(x)l_{_\succ}(y) \}w\cr 
			&+&\{r_{_\succ}(z)l_{\cdot}(x)-l_{_\succ}(x)r_{_\succ}(z)\}v+
			\{r_{_\succ}(z)r_{\cdot}(y)-r_{_\succ}(y\succ z)\}u.
		\end{eqnarray*}
		\begin{eqnarray*}
			(x+u, y+v, z+w)_{_r}&=&((x+u)\prec'(y+v) )\prec' (z+w)-(x+u)\prec'((y+v)\cdot'(z+w))\cr 
			&=&(x,y,z)_{_r}+\{l_{_\prec}(x\prec y)-l_{_\prec}(x)l_{\cdot}(y) \}w\cr 
			&+&\{r_{_\prec}(z)l_{_\prec}(x)-l_{_\prec}(x)r_{\cdot}(z) \}v+
			\{r_{_\prec}(z)r_{_\prec}(y)-r_{_\prec}(y\cdot z)\}u.
		\end{eqnarray*}
		If in addition the four linear maps $l_{_\succ}, l_{_\prec},r_{_\succ}, r_{_\prec}$ 
		satisfy Eqs.~\eqref{eq_bimodule_pre_anti_flexible1}~-~
		\eqref{eq_bimodule_pre_anti_flexible5}, 
		then for any $x,y,z\in   A$, and for any  
		$u, v,w\in V$ the following conditions
		are satisfied, 
		\begin{eqnarray*}
			(x+u, y+v, z+w)_{_m}=(z+w, y+v, x+u)_{_m},\;\; 
			(x+u, y+v, z+w)_{_l}=(z+w, y+v, x+u)_{_r}.
		\end{eqnarray*}
		Therefore, holds the equivalence.
	\end{proof}
	In the following of this paper, for a given pre-anti-flexible algebra
	$(A, \prec, \succ)$, a vector space  $V$ and a linear map 
	$\varphi : A\rightarrow \End(V)$, 
	its dual linear map is defined as $\varphi^* : A\rightarrow \End(V^*)$ by
	\begin{eqnarray}\label{eq_dual_map}
		\langle \varphi^*(x)u^*, v\rangle=
		\langle u^*, \varphi(x)v\rangle , \quad \forall x\in A, \; v\in V, u^*\in V^*,
	\end{eqnarray}
	where $\langle , \rangle $ is the usual pairing between $V$ and $V^*$. 
	In addition, we also denote by $\sigma$ a linear map from 
	$V\otimes V$ into itself or from $V^*\otimes V^*$ into itself define by for any 
	$u,v\in V$, $u^*, v^*\in V^*$, 
	$\sigma(u\otimes v)=v\otimes u$ and $\sigma(u^*\otimes v^*)=v^*\otimes u^*$.
	\begin{pro}\label{prop_operation_bimodule_pre_anti_flexible}
		Let $(l_{_\succ},r_{_\succ},  l_{_\prec}, r_{_\prec}, V)$ be a bimodule of a  
		pre-anti-flexible algebra $(  A, \prec, \succ)$, where 
		$V$ is a vector space and $l_{_\succ}, l_{_\prec},r_{_\succ}, r_{_\prec}:  
		A\rightarrow \End(V) $ 
		are four linear maps. We have:
		\begin{enumerate}
			\item $(l_{_\succ},0,  0, r_{_\prec}, V)$ , 
			$(r^*_{_\prec}, l^*_{_\prec} , l^*_{_\succ}, r^*_{_\succ}  , V^*)$and
			$(r^*_{_\prec}, 0 , 0,l^*_{_\succ},  V^*)$   
			are bimodules of the pre-anti-flexible algebra $(  A, \prec, \succ)$,
			\item\label{eq:one} $(l_{\cdot}, r_{\cdot}, V)$, $(l_{_\succ}, r_{_\prec}, V)$, 
			$(r^*_{\cdot}  , l^*_{\cdot}  , V^*)$ and 
			$(r^*_{_\prec}  , l^*_{_\succ}  , V^*)$ are bimodules of 
			the underlying anti-flexible algebra $aF(A)$ of $(  A, \prec, \succ)$,
		\end{enumerate}
		where,  $l_{_\succ}+l_{_\prec}=l_{\cdot}, r_{_\succ}+r_{_\prec}=r_{\cdot}$.
	\end{pro}
	\begin{proof}
		It is well known that 
		$aF(A)$ is an anti-flexible algebra.
		\begin{enumerate}
			\item
			Since $(l_{_\succ},r_{_\succ},  l_{_\prec}, r_{_\prec}, V)$ is a bimodule of the 
			pre-anti-flexible algebra $( A , \prec, \succ)$, then 
			the four linear maps $l_{_\succ}, l_{_\prec},r_{_\succ}, r_{_\prec}:
			A \rightarrow \End(V)$ satisfy 
			Eqs.~\eqref{eq_bimodule_pre_anti_flexible1}~-~\eqref{eq_bimodule_pre_anti_flexible5}
			which still satisfied by setting
			$ l_{_\prec}=0$ and $r_{_\succ}=0$. Thus $(l_{_\succ},0,  0, r_{_\prec}, V)$  is a 
			bimodule of the pre-anti-flexible algebra $( A , \prec, \succ)$.
			Furthermore, using Eq.~\eqref{eq_dual_map}, we deduce that 
			$(r^*_{_\prec}, l^*_{_\prec} , l^*_{_\succ}, r^*_{_\succ}  , V^*)$ and
			$(r^*_{_\prec}, 0 , 0,l^*_{_\succ}, V^*)$   are also
			bimodules of the pre-anti-flexible algebra $( A , \prec, \succ)$.
			\item 
			Since the four linear maps $l_{_\succ}, l_{_\prec},r_{_\succ}, r_{_\prec}: A \rightarrow \End(V)$ 
			satisfy Eqs.~\eqref{eq_bimodule_pre_anti_flexible1}~-~\eqref{eq_bimodule_pre_anti_flexible5},
			then both linear maps $l_{_\succ}$ and $r_{_\prec}$ satisfy
			Eqs.~\eqref{eq_bimodule_pre_anti_flexible1} and \eqref{eq_bimodule_pre_anti_flexible3} 
			which is exactly  Eqs.~\eqref{eqbimodule2} and \eqref{eqbimodule1}, respectively. 
			Thus $(l_{_\succ}, r_{_\prec},V)$ is a bimodule of 
			$aF(A)$. In view Eqs.~\eqref{eq_bimodule_pre_anti_flexible1} 
			and \eqref{eq_bimodule_pre_anti_flexible4}, 
			we have for any $x,y\in  A $,
			\begin{eqnarray*}
				[l_{_\cdot}(x), r_{_\cdot}(y)]-	[l_{_\cdot}(y), r_{_\cdot}(x)]&=&
				\{l_{_\prec}(x)r_{_\cdot}(y)+r_{_\succ}(x)l_{_\cdot}(y)-
				l_{_\succ}(y)r_{_\succ}(x)-r_{_\prec}(y)l_{_\prec}(x)  \}\cr 
				&-&\{ l_{_\prec}(y)r_{_\cdot}(x)+r_{_\succ}(y)l_{_\cdot}(x)-
				l_{_\succ}(x)r_{_\succ}(y)-r_{_\prec}(x)l_{_\prec}(y)  \}\cr 
				&+&\{[l_{_\succ}(x), r_{_\prec}(y)]-[l_{_\succ}(y), r_{_\prec}(x)]  \}=0.
			\end{eqnarray*}
			In addition,  considering Eqs.~\eqref{eq_bimodule_pre_anti_flexible2},
			\eqref{eq_bimodule_pre_anti_flexible3} and \eqref{eq_bimodule_pre_anti_flexible5},
			we have for any $x,y\in  A $,
			\begin{eqnarray*}
				l_{_\cdot}(x\cdot y)-l_{_\cdot}(x)l_{_\cdot}(y)-r_{_\cdot}(x)r_{_\cdot}(y)+r_{_\cdot}(y\cdot x)
				&=&\{l_{_\succ}(x\cdot y) -l_{_\succ}(x)l_{_\succ}(y)-r_{_\prec}(x)r_{_\prec}(y)\cr 
				&+&r_{_\prec}(y\cdot x)\}+ \{l_{_\prec}(x\succ y)-l_{_\succ}(x)l_{_\prec}(y) \cr 
				&+&r_{_\succ}(y\prec x)-r_{_\prec}(x)r_{_\succ}(y)  \}+\{l_{_\prec}(x\prec y) \cr 
				&+& r_{_\prec}(y\succ x)-l_{_\prec}(x)l_{_\cdot}(y)-r_{_\succ}(x)r_{_\cdot}(y)\}
				=0.
			\end{eqnarray*} 
			Therefore both $( l_{_\cdot}, r_{_\cdot}, V)$ and  $(l_{_\succ}, r_{_\prec}, V)$ 
			are  bimodules of $aF(A)$.
			According to  Eq.~\eqref{eq_dual_map}, both 
			$(r^*_{_\cdot}  , l^*_{_\cdot}  , V^*)$ and 
			$(r^*_{_\prec}  , l^*_{_\succ}  , V^*)$ are bimodules of  $aF(A)$. 
		\end{enumerate}
	\end{proof}
	\begin{ex}
		Consider a pre-anti-flexible algebra $(A, \prec, \succ)$.
		We have $( L_{_\prec}, R_{_\prec} , L_{_\succ}, R_{_\succ}  , A)$ and 
		$(L_{_\succ},0,  0, R_{_\prec}, A)$ are bimodules of
		$(A, \prec, \succ)$. Besides,  
		$(R^*_{_\prec}, L^*_{_\prec} , L^*_{_\succ}, R^*_{_\succ}  , A^*)$ and
		$(R^*_{_\prec}, 0 , 0, L^*_{_\succ}, A^*)$   
		are also bimodules of the pre-anti-flexible algebra $(  A, \prec, \succ)$.
	\end{ex}
	\begin{rmk}\label{rmk_useful}
		For a given bimodule $(l_{_\succ},r_{_\succ},  l_{_\prec}, r_{_\prec}, V)$ 
		of a pre-anti-flexible algebra $(A, \prec, \succ)$ we have
		\begin{enumerate}
			\item 
			If both side of 
			Eqs.~\eqref{eq_bimodule_pre_anti_flexible1}~-~\eqref{eq_bimodule_pre_anti_flexible5} 
			are zero i.e.
			the linear maps $l_{_\succ}, l_{_\prec},r_{_\succ}, r_{_\prec}$ satisfy
			\begin{eqnarray*}
				r_{_\prec}(x)l_{_\succ}(y)=l_{_\succ}(y)r_{_\prec}(x),\;
				l_{_\prec}(x\succ y)=l_{_\succ}(x)l_{_\prec}(y),\; 
				r_{_\prec}(x)r_{_\succ}(y)=r_{_\succ}(y\prec x),\\
				l_{_\succ}(x\cdot y)=l_{_\succ}(x)l_{_\succ}(y),\; 
				r_{_\prec}(x)r_{_\prec}(y)=r_{_\prec}(y\cdot x),\;
				r_{_\succ}(x)l_{\cdot}(y)=l_{_\succ}(y)r_{_\succ}(x),\\
				r_{_\prec}(y)l_{_\prec}(x)=l_{_\prec}(x)r_{\cdot}(y),\;
				r_{_\succ}(x)r_{\cdot}(y)=r_{_\succ}(y\succ  x),\;
				l_{_\prec}(x\prec y)=l_{_\prec}(x)l_{\cdot}(y),
			\end{eqnarray*}  
			with $l_{\cdot}=l_{_\succ}+l_{_\prec}$ and $r_{\cdot}=r_{_\succ}+r_{_\prec}$.
			\item 
			For any $x,y \in A$ we have 
			\begin{eqnarray}\label{eq:useful}
				L_{\cdot}(x)L_{\cdot}(y)+R_{\cdot}(x) R_{\cdot}(y)&=&
				(L_{\prec}(x)+L_{\succ}(x))(L_{\prec}(y)+L_{\succ}(y))+
				(R_{\prec}(x)+R_{\succ}(x))(R_{\prec}(y)+R_{\succ}(y))\cr 
				&=&(L_{\succ}(x)L_{\succ}(y)+R_{\prec}(x)R_{\prec}(y))+
				(L_{\succ}(x)L_{\prec}(y)+L_{\prec}(x)R_{\succ}(y))\cr
				&+&(L_{\prec}(x)L_{\succ}(y)+L_{\prec}(x)L_{\prec}(y)+
				(R_{\succ}(x)R_{\succ}(y)+R_{\succ}(x)L_{\prec}(y) )\cr 
				L_{\cdot}(x)L_{\cdot}(y)+R_{\cdot}(x) R_{\cdot}(y)&=&L_{\cdot}(x\cdot y)+R_{\cdot}(y\cdot x)
			\end{eqnarray}
			\item Besides, for any $x,y\in A$
			\begin{eqnarray}\label{eq:useful1}
				[L_{\cdot}(x), R_{\cdot}(y)]-[L_{\cdot}(y), R_{\cdot}(x)]&=&
				\{L_{_\prec}(x)R_{_\cdot}(y)+R_{_\succ}(x)L_{_\cdot}(y)-L_{_\succ}(y)R_{_\succ}(x)-R_{_\prec}(y)L_{_\prec}(x)  \}\cr 
				&-&\{ L_{_\prec}(y)R_{_\cdot}(x)+R_{_\succ}(y)L_{_\cdot}(x)-L_{_\succ}(x)R_{_\succ}(y)-R_{_\prec}(x)L_{_\prec}(y)  \}\cr 
				&+&\{[L_{_\succ}(x), R_{_\prec}(y)]-[L_{_\succ}(y), R_{_\prec}(x)]  \}=0
			\end{eqnarray}
			\item\label{dual-bimodule}
			Both dendriform and pre-anti-flexible algebras have the same shape of  dual bimodules.
			This fact induces some consequences which we will derive and explain in the following of this paper.
		\end{enumerate}
	\end{rmk}
	\begin{thm}\label{Theo_pre_Sum}
		Let $( A , \prec_{_A }, \succ_{_A })$  be a pre-anti-flexible algebra. Suppose there is a pre-anti-flexible algebra structure 
		"$ \prec_{_{  A^*}}, \succ_{_{A^*}}$" on  $ A^*$. The following statements are equivalent:
		\begin{enumerate}
			\item\label{1}  $(R^*_{\prec_{_  A}},L^*_{\succ_{_  A}}, R^*_{\prec_{_{  A^*}}}, L^*_{\succ_{_{ A^*}}}, A,  A^*)$ is a matched pair of anti-flexible algebras $aF(A)$ and $aF(A^*)$.
			\item\label{2} 	There exists an anti-flexible algebra structure on $ A\oplus A^*$ given by for any $x,y\in  A$ and for any $a,b\in  A^*$, 
			\begin{eqnarray}\label{eq_anti_flexible_sum}
				(x+a)\star(y+b)=
				(x\cdot y+R^*_{\prec_{_ A}}(a)y+L^*_{\succ_{_ A}}(b)x)+
				(a\circ b+R^*_{\prec_{_{ A^*}}}(x)b+L^*_{\succ_{_{ A^*}}}(y)a),
			\end{eqnarray} where $x\cdot y= x\prec_{_ A}y+ x\succ_{_ A} y$, $a\circ b=   a\prec_{_{ A^*}}b+a\succ_{_{ A^*}}b$, 
			and  a non-degenerate closed skew-symmetric  bilinear form $\omega$  defined on $ A\oplus A^*$ given by  
			for any $x,y\in  A$ and for any $a,b \in  A^*$, 
			\begin{eqnarray}\label{eq_skew_symmetric_form}
				\omega(x+a, y+b)=\langle x,b\rangle -\langle y,a\rangle.
			\end{eqnarray} 
		\end{enumerate} 
	\end{thm}
	\begin{proof}
		Let $( A , \prec_{_A }, \succ_{_A })$  be a pre-anti-flexible algebra. 
		Suppose there is a pre-anti-flexible algebra structure 
		"$ \prec_{_{  A^*}}, \succ_{_{A^*}}$" on  $  A^*$.
		\begin{itemize}
			\item[$\eqref{1}\Longrightarrow \eqref{2}$]
			Suppose that $(R^*_{\prec_{_ A}},L^*_{\succ_{_ A}}, R^*_{\prec_{_{ A^*}}},L^*_{\succ_{_{ A^*}}}, A,  A^*)$ 
			is a matched pair of anti-flexible algebras $aF(A)$ and $aF(A^*)$.
			Then, there is an anti-flexible algebra structure "$\star$" on $ A\oplus A^*$ 
			given by for any $x,y\in  A$ and for any $a,b, \in  A^*$, 
			\begin{eqnarray*}
				(x+a)\star(y+b)=
				(x\cdot y+R^*_{\prec_{_ A}}(a)y+L^*_{\succ_{_ A}}(b)x)+
				(a\circ b+R^*_{\prec_{_{ A^*}}}(x)b+L^*_{\succ_{_{ A^*}}}(y)a), 
			\end{eqnarray*}
			where, $\forall x,y\in  A$, $\forall a,b\in  A^*$, 
			\begin{eqnarray*}
				x\cdot y= x\prec_{_ A}y+ x\succ_{_ A} y, a\circ b=   a\prec_{_{ A^*}}b+a\succ_{_{ A^*}}b.
			\end{eqnarray*}
			Besides, considering the skew-symmetric bilinear form $\omega$ defined on $ A\oplus A^*$ by 
			Eq.~\eqref{eq_skew_symmetric_form}, we have for any $x,y\in  A$ and for any $a,b\in  A^*$,
			\begin{eqnarray*}
				&&\omega((x+a)\star(y+b), (z+c))+\omega((y+b)\star(z+c), (x+a))+\omega((z+c)\star(x+a), (y+b))\cr 
				&&=
				\langle x\cdot y, c\rangle +\langle y, c\prec_{_{ A^*}} a\rangle +\langle x, b\succ_{_{ A^*}} c\rangle 
				-\langle z, a\circ b\rangle -\langle z\prec_{_ A} x, b\rangle  -\langle y\succ_{_{ A}} z, a\rangle \cr&&+
				\langle y\cdot z, a\rangle +\langle z, a\prec_{_{ A^*}} b\rangle +\langle y, c\succ_{_{ A^*}} a\rangle 
				-\langle x, b\circ c\rangle -\langle x\prec_{_ A} y, c\rangle -\langle z\succ_{_{ A}} x, b\rangle \cr&& 
				+\langle z\cdot y-x, b\rangle +\langle x, b\prec_{_{ A^*}} c\rangle +\langle z, a\succ_{_{ A^*}} b\rangle
				-\langle y, c\circ a\rangle -\langle y\prec_{_ A} z, a\rangle -\langle x\succ_{_{ A}} y, c\rangle =0.
			\end{eqnarray*}
			Clearly, $\omega$ is closed and $\omega( A,  A)=0=\omega( A^*, A^*)$, then $( A, \cdot)$ 
			and $( A^*, \circ)$ are Lagrangian anti-flexible subalgebras 
			of the anti-flexible algebra $( A\oplus A^*, \star)$.
			
			\item[$\eqref{2}\Longrightarrow \eqref{1}$]
			
			Suppose that there exists an anti-flexible algebra structure "$\star$" on $ A\oplus A^*$ given by Eq.~\eqref{eq_anti_flexible_sum}  
			and  a non-degenerate closed skew-symmetric  bilinear form on $ A\oplus A^*$ given by Eq.~\eqref{eq_skew_symmetric_form}. 
			According to Proposition~\ref{prop_operation_bimodule_pre_anti_flexible} the triple
			$(R^*_{\prec_{_ A}},L^*_{\succ_{_ A}},  A^*)$ is a bimodule 
			of $aF(A)$ and $(R^*_{\prec_{_{ A^*}}},L^*_{\succ_{_{ A^*}}},  A)$ is a bimodule of $aF(A^*)$, thus "$\star$" defines 
			an anti-flexible algebra structure on 
			$ A\oplus A^*$ if $(R^*_{\prec_{_ A}},L^*_{\succ_{_ A}},R^*_{\prec_{_{ A^*}}},L^*_{\succ_{_{ A^*}}},  A, A^*)$ is a 
			matched pair of the anti-flexible algebras $aF(A)$ and $aF(A^*)$. 
		\end{itemize}
		Therefore, holds the conclusion.
	\end{proof}
	
	\begin{thm}\label{thm_matchedpair}
		Let $( A, \prec_{_ A}, \succ_{_ A})$ and $(B, \prec_{_B}, \succ_{_B})$ be two pre-anti-flexible algebras.
		Suppose that there are four linear maps
		$ l_{_{\succ_ A}}, r_{_{\succ_ A}}, l_{_{\prec_ A}}, r_{_{\prec_ A}}: A\rightarrow \End(B)$ 
		such that $(l_{_{\succ_ A}}, r_{_{\succ_ A}}, l_{_{\prec_ A}}, r_{_{\prec_ A}}, B)$ is a bimodule of 
		$( A, \prec_{_ A}, \succ_{_ A})$ and another four linear maps 
		$ l_{_{\succ_B}}, r_{_{\succ_B}}, l_{_{\prec_B}}, r_{_{\prec_B}}:B\rightarrow \End( A)$ 
		such that $(l_{_{\succ_B}}, r_{_{\succ_B}}, l_{_{\prec_B}}, r_{_{\prec_B}},  A)$ is a bimodule of 
		$(B, \prec_{_B}, \succ_{_B})$. If in addition the eight linear maps 
		$l_{_{\succ_ A}}, r_{_{\succ_ A}}, l_{_{\prec_ A}},$ $r_{_{\prec_ A}}, 
		l_{_{\succ_B}}, r_{_{\succ_B}}, l_{_{\prec_B}}$, and $r_{_{\prec_B}}$
		satisfying the relations, for any $ x,y\in A$, and for any $a,b\in B$,
		\begin{subequations}
			\begin{eqnarray}\label{eq_matched_pre_1}
				(l_{_\succ{_{_B}}}(a)x)\prec_{_ A}y+  l_{_{\succ_{_B}}}(r_{_\succ{_ A}}(x)a)y-
				l_{_{\succ{_B}}}(a)(x\prec_{_ A} y)=\cr r_{_{\prec_{_B}}}(a)(y\succ_{_ A}x)-
				y\succ_{_ A}(r_{_{\prec_{_B}}}(a)x)-r_{_{\succ_{_B}}}(l_{_{\prec_{_ A}}}(x)a)y,
			\end{eqnarray}
			\begin{eqnarray}\label{eq_matched_pre_2}
				(l_{_{\succ_{_ A}}}(x)b)\prec_{_B}a+l_{_{\prec_{_ A}}}(r_{_{\succ_{_B}}}(b)x)a-
				l_{_{\succ_{_ A}}}(x)(b\prec_{_B}a)=\cr 
				r_{_{\prec_{_ A}}}(x)(a\succ_{_B} b)-a\succ_{_B}(r_{_{\prec_{_ A}}}(x)b)-
				r_{_{\succ_{_ A}}}(l_{_{\prec_{_B}}}(b)x)a,
			\end{eqnarray}
			\begin{eqnarray}\label{eq_matched_pre_3}
				(l_{_{\cdot_B}}(a)x )\succ_{_ A}y+l_{_{\succ_{_B}}}(r_{_{ \cdot_ A}}(x)a)y-
				l_{_{\succ_{B}}}(a)(x\succ_{_ A}y) =\cr
				r_{_{\prec_{_B}}}(a)(y\prec_{_ A} x)-y\prec_{_ A}(r_{\cdot_{_B}}(a)x)-
				r_{_{\prec_{_B}}}(l_{{\cdot_{_ A}}}(x)a)y,
			\end{eqnarray}
			\begin{eqnarray}\label{eq_matched_pre_4}
				r_{_{\succ_{B}}}(a)(x\cdot_{_ A} y)-x\succ{_{_ A}}(r_{_\succ{_B}}(a)y)-
				r_{_{_\succ{_B}}}(l_{_{\succ{_ A}}}(y)a)x =\cr
				(l_{_{\prec_{_B}}}(a)y)\prec_{_ A}x+l_{_{\prec_{_B}}}(r_{_{\prec_{_ A}}}(y)a)x-
				l_{_{\prec_{_B}}}(a)(y\cdot_{_ A} x), 
			\end{eqnarray}
			\begin{eqnarray}\label{eq_matched_pre_5}
				(l_{_{\cdot_ A}}(x)b )\succ_{_B}a+l_{_{\succ_{_ A}}}(r_{_{\cdot_B}}(b)x)a-
				l_{_{\succ_{ A}}}(x)(b\succ_{_B}a) =\cr
				r_{_{\prec_{_ A}}}(x)(a\prec_{_B} b)-a\prec_{_B}(r_{\cdot_{_ A}}(x)b)-
				r_{_{\prec_{_ A}}}(l_{{\cdot_{_B}}}(b)x)a,
			\end{eqnarray}
			\begin{eqnarray}\label{eq_matched_pre_6}
				r_{_{\succ_{ A}}}(x)(a\cdot_{_B} b)-a\succ{_{_B}}(r_{_\succ{_ A}}(x)b)-
				r_{_{_\succ{_ A}}}(l_{_{\succ{_B}}}(b)x)a=\cr
				(l_{_{\prec_{_ A}}}(x)b)\prec_{_ A}a+l_{_{\prec_{_ A}}}(r_{_{\prec_{_B}}}(b)x)a-
				l_{_{\prec_{_ A}}}(x)(b\cdot_{_B} a),
			\end{eqnarray}
			\begin{eqnarray}\label{eq_matched_pre_7}
				(r_{_{\succ_B}}(a)x)\prec_{_ A} y+l_{_{\prec_B}}( l_{_{\succ_ A}}(x)a)y
				-x\succ_{_ A}(l_{_{\prec_{_B}}}(a)y)-r_{_{\succ_{_B}}}(r_{_{\prec_{_ A}}}(y)a  )x=\cr 
				(r_{_{\succ_B}}(a)y)\prec_{_ A} x+l_{_{\prec_B}}( l_{_{\succ_ A}}(y)a)x
				-y\succ_{_ A}(l_{_{\prec_{_B}}}(a)x)-r_{_{\succ_{_B}}}(r_{_{\prec_{_ A}}}(x)a  )y,
			\end{eqnarray}
			\begin{eqnarray}\label{eq_matched_pre_8}
				(r_{_{\succ_ A}}(x)a)\prec_{_B} b+l_{_{\prec_ A}}( l_{_{\succ_B}}(a)x)b
				-a\succ_{_B}(l_{_{\prec_{_ A}}}(x)b)-r_{_{\succ_{_ A}}}(r_{_{\prec_{_B}}}(b)x  )a=\cr
				(r_{_{\succ_ A}}(x)b)\prec_{_B} a+l_{_{\prec_ A}}( l_{_{\succ_B}}(b)x)a
				-b\succ_{_B}(l_{_{\prec_{_ A}}}(x)a)-r_{_{\succ_{_ A}}}(r_{_{\prec_{_B}}}(a)x  )b, 
			\end{eqnarray}
			\begin{eqnarray}\label{eq_matched_pre_9}
				(r_{ \cdot_{_B}}(a)x)\succ_{_ A} y +l_{_{_\succ{_B}}}(l_{_{\cdot_{_ A}}}(x)a)y
				-x\succ_{_ A}(l_{_{\succ_{_B}}}(a)y)-r_{_{\succ_{_B}}}(r_{_{\succ_{_ A}}}(y)a)x=\cr 
				(r_{_{\prec{_B}}}(a)y)\prec_{_ A} x+l_{_{\prec_{_B}}}(l_{_{\prec_{_ A}}}(y) a)x
				-y\prec_{_ A}(l_{\cdot_{_B}} (a)x)-r_{_{\prec_{_B}}}(r_{\cdot_{_ A}}(x)a)y ,
			\end{eqnarray}
			\begin{eqnarray}\label{eq_matched_pre_10}
				(r_{\cdot_{_ A}}(x)a)\succ_{_B} b+l_{_{_\succ{_ A}}}(l_{_{\cdot_{_B}}}(a)x)b
				-a\succ_{_B}(l_{_{\succ_{_ A}}}(x)b)-r_{_{\succ_{_ A}}}(r_{_{\succ_{_B}}}(b)x)a=\cr 
				(r_{_{\prec{_ A}}}(x)b)\prec_{_B} a+ l_{_{\prec_{_ A}}}(l_{_{\prec_{_B}}}(b) x)a
				-b\prec_{_B}(l_{\cdot_{_ A}} (x)a)-r_{_{\prec_{_ A}}}(r_{\cdot_{_B}}(a)x)b,
			\end{eqnarray}
		\end{subequations}
		there is a pre-anti-flexible product on $ A\oplus B$ given by, for any  $x,y\in A$, and for any $a,b\in B$,
		\begin{eqnarray}\label{eq_pre_anti_matched}
			\begin{array}{cccc}
				(x+a)\prec (y+b)= \{x\prec_{_ A} y+l_{_{\prec_{_B}}}(a)y+r_{_{\prec_{_B}}}(b)x\}+
				\{a\prec_{_B} b+l_{_{\prec_{_ A}}}(x)b+r_{_{\prec_{_ A}}}(y)a\}, \cr
				(x+a)\succ (y+b)=
				\{ x\succ_{_ A} y+l_{_{\succ_{_B}}}(a)y+r_{_{\succ_{_B}}}(b)x\}+
				\{a\succ_{_B} b+l_{_{\succ_{_ A}}}(x)b+r_{_{\succ_{_ A}}}(y)a\}, 
			\end{array}
		\end{eqnarray}
		where for any $x,y, \in  A$ and any $a,b\in B$,  
		\begin{eqnarray*}
			x\cdot_{_ A} y=x\prec_{_ A} y+x\succ_{_ A} y,\quad  l_{ \cdot_{_ A}}
			=l_{_\prec{_{_ A}}}+l_{_\succ{_{_ A}}}, \quad r_{\cdot_{_ A}}
			=r_{_\prec{_{_ A}}}+r_{_\succ{_{_ A}}},\cr
			a\cdot_{_B} b=a\prec_{_B} b+a\succ_{_B} b,\quad l_{\cdot_{_B}}
			=l_{_\prec{_{_B}}}+l_{_\succ{_{_B}}}, \quad  r_{\cdot{_B}}
			=r_{_\prec{_{_B}}}+r_{_\succ{_{_B}}}.
		\end{eqnarray*}
	\end{thm}
	\begin{proof}
		
		Let $( A, \prec_{_ A}, \succ_{_ A})$ and $(B, \prec_{_B}, \succ_{_B})$ be two pre-anti-flexible algebras.
		Suppose in addition that  $(l_{_{\succ_ A}}, r_{_{\succ_ A}}, l_{_{\prec_ A}}, r_{_{\prec_ A}}, B)$ 
		is a bimodule of $( A, \prec_{_ A}, \succ_{_ A})$ and  $(l_{_{\succ_B}}, r_{_{\succ_B}}, l_{_{\prec_B}}, r_{_{\prec_B}},  A)$ 
		is a bimodule of $(B, \prec_{_B}, \succ_{_B})$, where 
		$ l_{_{\succ_B}}, r_{_{\succ_B}}, l_{_{\prec_B}}, r_{_{\prec_B}}:B\rightarrow \End( A)$ and 
		$ l_{_{\succ_ A}}, r_{_{\succ_ A}}, l_{_{\prec_ A}}, r_{_{\prec_ A}}: A\rightarrow \End(B)$ 
		are eight linear maps. Considering the product given in Eq.~\eqref{eq_pre_anti_matched}, we have
		for any $x,y, z\in  A$ and any $a,b,c\in B$, 
		\begin{eqnarray*}
			(x+a, y+b, z+c)_{_m}&=&(x,y,z)_{_m}+(a,b,c)_{_m}+
			\{l_{_{\prec_B}}(a\succ_{_B}b)z -l_{_\succ{_B}}(a)( l_{_{\prec_B}}(b)z) \}\cr
			&+&\{(l_{_\succ{_{_B}}}(a)y)\prec_{_ A}z+  l_{_\succ{_B}}(r_{_\succ{_ A}}(y)a)z-
			l_{_{\succ{_B}}}(a)(y\prec_{_ A} z) \}\cr 
			&+&\{r_{_{\prec_{_B}}}(c)(x\succ_{_ A} y)-x\succ_{_ A}(r_{_{\prec_{_B}}}(c)y)-
			r_{_{\succ_{_B}}}(l_{_{\prec_{_ A}}}(y)c)x  \}\cr 
			&+&\{ r_{_{\prec_{_ A}}}(z)(a\succ_{_B} b)-a\succ_{_B}(r_{_{\prec_{_ A}}}(z)b)-
			r_{_{\succ_{_ A}}}(l_{_{\prec_{_B}}}(b)z)a  \}\cr
			&+&\{ (l_{_{\succ_{_ A}}}(x)b)\prec_{_B}c+l_{_{\prec_{_ A}}}(r_{_{\succ_{_B}}}(b)x)c-
			l_{_{\succ_{_ A}}}(x)(b\prec_{_B}c) \}\cr 
			&+&\{r_{_{\prec_{_B}}}(c)(l_{_{\succ_{_B}}}(a)y)-
			l_{_{\succ_{_B}}}(a)(r_{_{\prec_{_B}}}(c)y)  \}+\{ (r_{_{\succ_B}}(b)x)\prec_{_ A} z
			\cr &+&l_{_{\prec_B}}( l_{_{\succ_ A}}(x)b)z
			-x\succ_{_ A}(l_{_{\prec_{_B}}}(b)z)
			-r_{_{\succ_{_B}}}(r_{_{\prec_{_ A}}}(z)b  )x \}\cr
			&+&\{ r_{_{\prec_{_ A}}}(z)(r_{_{\succ_{_ A}}}(y)a)-
			r_{_{\succ_{_ A}}}(y\prec_{_ A}z)a \}+
			\{ (r_{_{\succ_ A}}(y)a)\prec_{_B} c
			\cr&+&l_{_{\prec_ A}}( l_{_{\succ_B}}(a)y)c
			-a\succ_{_B}(l_{_{\prec_{_ A}}}(y)c)
			-r_{_{\succ_{_ A}}}(r_{_{\prec_{_B}}}(c)y  )a  \}\cr 
			&+&\{r_{_{\prec_{_B}}}(c)(r_{_{\succ_{_B}}}(b)x)-
			r_{_{\succ_{_B}}}(b\prec_{_B}c)x  \}\cr
			&+&\{ l_{_{\prec_ A}}(x\succ_{_ A}y)c -
			l_{_\succ{_ A}}(x)( l_{_{\prec_ A}}(y)c) \}+
			\{ r_{_{\prec_{_ A}}}(z)(l_{_{\succ_{_ A}}}(x)b)-
			l_{_{\succ_{_ A}}}(x)(r_{_{\prec_{_ A}}}(z)b) \}
		\end{eqnarray*}
		%%%%%%%%%%%%%%%%5
		%%%%%%%%%%%%%%%%%
		%%%%%%%%%%%%%%%%%
		\begin{eqnarray*}
			(x+a, y+b, z+c)_{_l}&=&(x,y,z)_{_l}+(a,b,c)_{_l}+
			\{r_{_{\succ_{_A}}}(z)(r_{\cdot_{A}}(y)a)-r_{_{\succ_{_A}}}(y\succ_{_A} y)a \}\cr
			&+&\{(l_{_{\cdot_B}}(a)y )\succ_{_A}z+l_{_{\succ_{_B}}}(r_{_{\cdot_A}}(y)a)z-
			l_{_{\succ_{B}}}(a)(y\succ_{_A}z) \}\cr
			&+&\{r_{_{\succ_{B}}}(c)(x\ast_{_A} y)-x\succ{_{_A}}(r_{_\succ{_B}}(c)y)-
			r_{_{_\succ{_B}}}(l_{_{\succ{_A}}}(y)c)x  \}\cr 
			&+&\{(l_{_{\cdot_A}}(x)b )\succ_{_B}c+l_{_{\succ_{_A}}}(r_{_{\cdot_B}}(b)x)c-
			l_{_{\succ_{A}}}(x)(b\succ_{_B}c)  \}\cr
			&+&\{ r_{_{\succ_{A}}}(z)(a\cdot_{_B} b)-a\succ{_{_B}}(r_{_\succ{_A}}(z)b)-
			r_{_{_\succ{_A}}}(l_{_{\succ{_B}}}(b)z)a \}\cr
			&+&\{ l_{_\succ{_B}}(a\cdot_{_B} b)z-l_{_{\succ_{_B}}}(a)(l_{_{\succ_{_B}}}(b)z)\}+ 
			\{(r_{\cdot_{_B}}(b)x)\succ_{_A} z   \cr 
			&+&l_{_{_\succ{_B}}}(l_{_{\cdot_{_A}}}(x)b)z-x\succ_{_A}(l_{_{\succ_{_B}}}(b)z)-
			r_{_{\succ_{_B}}}(r_{_{\succ_{_A}}}(z)b)x \}\cr
			&+&\{  l_{_\succ{_A}}(x\cdot_{_A} y)c-l_{_{\succ_{_A}}}(x)(l_{_{\succ_{_A}}}(y)c) \}+ 
			\{  (r_{\cdot_{_A}}(y)a)\succ_{_B} c   \cr 
			&+&l_{_{_\succ{_A}}}(l_{_{\cdot_{_B}}}(a)y)c-a\succ_{_B}(l_{_{\succ_{_A}}}(y)c)-
			r_{_{\succ_{_A}}}(r_{_{\succ_{_B}}}(c)y)a \}\cr
			&+&\{ r_{_{\succ_{_B}}}(c)(l_{_\cdot{_B}}(a)y)-l_{_{_{\succ{_B}}}}(a)(r_{_{\succ{_B}}}(c)y) \}\cr 
			&+&\{ r_{_{\succ_{_B}}}(c)(r_{_\cdot{_B}}(b)x)-r_{_{\succ_{_B}}}(b\succ_{_B} c)x \}+
			\{  r_{_{\succ_{_A}}}(z)(l_{_\cdot{_A}}(x)b)-l_{_{_{\succ{_A}}}}(x)(r_{_{\succ{_A}}}(z)b)  \}
		\end{eqnarray*}
		%%%%%%%%%%%%%%%%5
		%%%%%%%%%%%%%%%%%
		%%%%%%%%%%%%%%%%%
		\begin{eqnarray*}
			(z+c, y+b, x+a)_{_r}&=&(z,y,x)_{_r}+(c,b,a)_{_r}
			+\{r_{_{\prec_{_B}}}(a)(r_{_{\prec_{_B}}}(b)z)-r_{_{\prec_{_B}}}(b\ast_{_B} a)z  \}\cr 
			&+&\{(l_{_{\prec_{_B}}}(c)y)\prec_{_A}x+l_{_{\prec_{_B}}}(r_{_{\prec_{_A}}}(y)c)x-
			l_{_{\prec_{_B}}}(c)(y\cdot_{_A} x)  \}\cr
			&+&\{r_{_{\prec_{_B}}}(a)(z\prec_{_A} y)-z\prec_{_A}(r_{\cdot_{_B}}(a)y)-
			r_{_{\prec_{_B}}}(l_{{\cdot_{_A}}}(y)a)z  \}\cr 
			&+&\{ (l_{_{\prec_{_A}}}(z)b)\prec_{_A}a+l_{_{\prec_{_A}}}(r_{_{\prec_{_B}}}(b)z)a-
			l_{_{\prec_{_A}}}(z)(bAst_{_B} a)  \}\cr
			&+&\{ r_{_{\prec_{_A}}}(x)(c\prec_{_B} b)-c\prec_{_B}(r_{\cdot_{_A}}(x)b)-
			r_{_{\prec_{_A}}}(l_{{\cdot_{_B}}}(b)x)c  \}\cr
			&+&\{l_{_{\prec_{_B}}}(c\prec_{_B} b)x-l_{_{\prec_{_B}}}(c)(l_{_{\cdot_{_B}}}(b)c)  \}+
			\{(r_{_{\prec{_B}}}(b)z)\prec_{_A} x\cr
			&+&l_{_{\prec_{_B}}}(l_{_{\prec_{_A}}}(z) b)x-z\prec_{_A}(l_{\cdot_{_B}} (b)x)-
			r_{_{\prec_{_B}}}(r_{\cdot_{_A}}(x)b)z  \}\cr
			&+&\{r_{_{\prec{_B}}}(a)(l_{_{_\prec{_B}}}(c)y)-l_{_{\prec{_B}}}(c)(r_{_{\cdot{_B}}}(a)y)  \}+
			\{    (r_{_{\prec{_A}}}(y)c)\prec_{_B} a          \cr
			&+& l_{_{\prec_{_A}}}(l_{_{\prec_{_B}}}(c) y)a-c\prec_{_B}(l_{\cdot_{_A}} (y)a)-
			r_{_{\prec_{_A}}}(r_{\cdot_{_B}}(a)y)c  \}\cr
			&+&\{r_{_{\prec_{_A}}}(x)(r_{_{\prec_{_A}}}(y)c)-r_{_{\prec_{_A}}}(y\cdot_{_A} x)c   \}\cr 
			&+&\{l_{_{\prec_{_A}}}(z\prec_{_A} y)a-l_{_{\prec_{_A}}}(z)(l_{_{\cdot_{_A}}}(y)z)\}
			+\{r_{_{\prec{_A}}}(x)(l_{_{_\prec{_A}}}(z)b)-l_{_{\prec{_A}}}(z)(r_{_{\cdot{_A}}}(x)b)   \}
		\end{eqnarray*}
		%%%%%%%%%%%%%%%%
		Besides, for any $x,y, z\in A$ and for any $a,b,c\in B$, 
		$
		(x+a,y+b,z+c)_{_m}=(z+c,y+b,x+a)_{_m}
		$
		and 
		$
		(x+a,y+b,z+c)_{_l}=(z+c,y+b,x+a)_{_r}
		$ are equivalent to   
		$(l_{_{\succ_A}}, r_{_{\succ_A}}, l_{_{\prec_A}}, r_{_{\prec_A}}, B)$ is a bimodule of 
		$(A, \prec_{_A}, \succ_{_A})$,  $(l_{_{\succ_B}}, r_{_{\succ_B}}, l_{_{\prec_B}}, r_{_{\prec_B}}, A)$ 
		is a bimodule of $(B, \prec_{_B}, \succ_{_B})$ and 
		Eqs.~\eqref{eq_matched_pre_1}~-~\eqref{eq_matched_pre_10} are satisfied.
		%%%%%%%%%%%%%%%%%
		%%%%%%%%%%%%%%%%%
	\end{proof}
	\begin{defi}
		Let $(A, \prec_{_A}, \succ_{_A})$ and $(B, \prec_{_B}, \succ_{_B})$ be two 
		pre-anti-flexible algebras. Suppose that there are four linear maps
		$ l_{_{\succ_A}}, r_{_{\succ_A}}, l_{_{\prec_A}}, r_{_{\prec_A}}:A\rightarrow \End(B)$ such that 
		$(l_{_{\succ_A}}, r_{_{\succ_A}}, l_{_{\prec_A}}, r_{_{\prec_A}}, B)$ is a bimodule of 
		$(A, \prec_{_A}, \succ_{_A} )$ and  another four linear maps 
		$l_{_{\succ_B}}, r_{_{\succ_B}}, l_{_{\prec_B}}, r_{_{\prec_B}}:B\rightarrow \End(A)$ 
		such that 
		$(l_{_{\succ_B}}, r_{_{\succ_B}}, l_{_{\prec_B}}, r_{_{\prec_B}}, A)$ 
		is a bimodule of $(B, \prec_{_B}, \succ_{_B} )$ and
		Eqs.~\eqref{eq_matched_pre_1}~-~\eqref{eq_matched_pre_10} hold.
		Then we call the ten-tuple $(A, B,l_{_{\succ_A}}, r_{_{\succ_A}}, l_{_{\prec_A}},
		r_{_{\prec_A}},l_{_{\succ_B}}, r_{_{\succ_B}}, l_{_{\prec_B}}, r_{_{\prec_B}} )$ a 
		{\bf matched pair of the pre-anti-flexible} algebras $(A, \prec_{_A}, \succ_{_A})$ 
		and $(B, \prec_{_B}, \succ_{_B})$.
		We also denote the pre-anti-flexible algebra defined by Eq.~\eqref{eq_pre_anti_matched} by 
		$A\bowtie^{l_{_{\succ_A}}, r_{_{\succ_A}}, l_{_{\prec_A}}, 
			r_{_{\prec_A}}}_{l_{_{\succ_B}}, r_{_{\succ_B}}, l_{_{\prec_B}}, r_{_{\prec_B}}} B$ or simply by 
		$A \bowtie B$.
	\end{defi}
	\begin{cor}\label{corollary_matched_pair_pre}
		If 
		$(A, B, l_{_{\succ_A}}, r_{_{\succ_A}}, l_{_{\prec_A}}, r_{_{\prec_A}}, 
		l_{_{\succ_B}}, r_{_{\succ_B}}, l_{_{\prec_B}}, r_{_{\prec_B}} )$
		is  a matched pair of the pre-anti-flexible algebras 
		$(A, \prec_{_A}, \succ_{_A})$ and $(B, \prec_{_B}, \succ_{_B})$ then
		$(l_{_{\cdot_{_A}}}, r_{_{\cdot_{_A}}}, l_{_{\cdot_{_B}}},  r_{_{\cdot_{_B}}}, A, B  )$ 
		is a matched pair of anti-flexible algebras  $aF(A)$ and $aF(B)$.
	\end{cor}
	\begin{proof}
		Suppose that $(A, B, l_{_{\succ_A}}, r_{_{\succ_A}}, l_{_{\prec_A}}, 
		r_{_{\prec_A}}, l_{_{\succ_B}}, r_{_{\succ_B}}, l_{_{\prec_B}}, r_{_{\prec_B}}  )$
		is  a matched pair of the pre-anti-flexible algebras $(A, \prec_{_A}, \succ_{_A})$ 
		and $(B, \prec_{_B}, \succ_{_B})$, where 
		$(l_{_{\succ_A}}, r_{_{\succ_A}}, l_{_{\prec_A}}, r_{_{\prec_A}}, B)$ 
		is a bimodule of the pre-anti-flexible algebra $(A,  \prec_{_A}, \succ_{_A})$ and 
		$(l_{_{\succ_B}}, r_{_{\succ_B}}, l_{_{\prec_B}}, r_{_{\prec_B}}, A)$ 
		is a bimodule of the pre-anti-flexible algebra $(B, \prec_{_B}, \succ_{_B})$. 
		According to Proposition~\ref{prop_operation_bimodule_pre_anti_flexible}, 
		$(l_{_{\cdot_{_A}}}, r_{_{\cdot_{_A}}}, B)$
		is a bimodule of the anti-flexible algebra $aF(A)$ and 
		$(l_{_{\cdot_{_B}}},  r_{_{\cdot_{_B}}}, A)$ is a bimodule of the anti-flexible algebra $aF(B)$.
		In addition, underlying anti-flexible product defined on 
		$A\oplus B$ in Eq.~\eqref{eq_pre_anti_matched}  is exactly that obtained from the matched pair  
		$(l_{_{\cdot_{_A}}}, r_{_{\cdot_{_A}}},l_{_{\cdot_{_B}}},  r_{_{\cdot_{_B}}}, A, B)$
		of anti-flexible algebra $aF(A)$ and $aF(B)$.
	\end{proof}
	\begin{thm}
		Let $(A, \prec_{_A}, \succ_{_A})$  be a pre-anti-flexible algebra. 
		Suppose there is a pre-anti-flexible algebra structure 
		"$ \prec_{_{A^*}}, \succ_{_{A^*}}$" on its dual space $A^*$. 
		The following statements are equivalent:
		\begin{enumerate}
			\item\label{un}
			$(A, A^*, R^*_{\prec_{_A}},L^*_{\succ_{_A}}, R^*_{\prec_{_{A^*}}}, L^*_{\succ_{_{A^*}}})$ 
			is a matched pair of anti-flexible algebras $aF(A)$ and $aF(A^*)$. 
			\item\label{deux}
			$(A, A^*, -R^*_{\succ_{_A}},-L^*_{\prec_{_A}} ,R^*_{\cdot_{_A}},L^*_{\cdot_{_A}} , 
			-R^*_{\succ_{_{A^*}}},-L^*_{\prec_{_{A^*}}}, 
			R^*_{\circ_{_{A^*}}}, L^*_{\circ_{_{A^*}}})$
			is a matched pair of the  pre-anti-flexible algebras $(A, \prec_{_A}, \succ_{_A})$ and 
			$(A^*, \prec_{_{A^*}}, \succ_{_{A^*}})$.
		\end{enumerate}
	\end{thm}
	\begin{proof}
		Let $(A, \prec_{_A}, \succ_{_A})$  be a pre-anti-flexible algebra. 
		Suppose there is a pre-anti-flexible algebra structure 
		"$ \prec_{_{A^*}}, \succ_{_{A^*}}$" on its dual space $A^*$.
		According to Corollary~\ref{corollary_matched_pair_pre}, 
		we have $\eqref{deux} \Longrightarrow \eqref{un}$.
		
		If $(A, A^*, R^*_{\prec_{_A}},L^*_{\succ_{_A}}, R^*_{\prec_{_{A^*}}}, L^*_{\succ_{_{A^*}}})$ 
		is a matched pair of anti-flexible algebras $aF(A)$ and $aF(A^*)$, 
		According to Theorem~\ref{Theo_pre_Sum}, 
		there exists an anti-flexible algebra structure on $A\oplus A^*$ given by Eq.~\eqref{eq_anti_flexible_sum}
		and  a non-degenerate closed skew-symmetric  bilinear form on $A\oplus A^*$ 
		given by  Eq.~\eqref{eq_skew_symmetric_form}. In view of Theorem~\ref{Theo_existance_pre_anti_flexible}, 
		there exists  a pre-anti-flexible algebra structure "$\prec, \succ$" defined on 
		$A\oplus A^*$ and satisfying Eq.~\eqref{eq_useful1} 
		i.e. for any $x,y, z\in A$ and for any $a, b, c\in A^*$, 
		\begin{eqnarray*}
			\omega((x+a)\prec (y+b), (z+c))=\omega((x+a), (y+b)\star(z+c) ),\cr
			\omega((x+a)\succ (y+b), (z+c))=\omega((y+b), (z+c)\star(x+a)),
		\end{eqnarray*}
		where "$\star$" is given by Eq.~\eqref{eq_anti_flexible_sum}.
		More precisely, we have for any $x,y,z\in A$ and for any $a, b, c\in A^*$,
		\begin{eqnarray*}
			\omega(x+a, (y+b)\star(z+c) )&=& 
			\omega(x+a,    (y\cdot z+R^*_{\prec_{_A}}(b)z+L^*_{\succ_{_A}}(c)y)\cr&+&
			(b\circ c+R^*_{\prec_{_{A^*}}}(y)c+L^*_{\succ_{_{A^*}}}(z)b))\cr &=&
			\langle x,b\circ c+R^*_{\prec_{_{A^*}}}(y)c+L^*_{\succ_{_{A^*}}}(z)b) \rangle 
			\cr&-&\langle  y\cdot z+R^*_{\prec_{_A}}(b)z+L^*_{\succ_{_A}}(c)y, a\rangle
			\cr &=&\langle  x, b\circ c\rangle+
			\langle x\prec_{_{A}}y, c\rangle+
			\langle z\succ_{_{A}}x, b\rangle
			-\langle  y\cdot z, a\rangle\cr
			&-&\langle  z, a\prec_{_{_{A^*}}}b\rangle
			-\langle  y, c\succ_{_{_{A^*}}}a\rangle\cr
			&=&\langle  x\prec_{_{A}}y+L_{\circ}^*(b)x-R_{_{_{\succ_{A^*}}}}^*(a)y, c\rangle
			\cr&-&
			\langle  z,  a\prec_{_{_{A^*}}}b+L_{\cdot}^*(y)a-R_{_{\succ_{A}}}^*(x)b\rangle \cr
			&=&\omega((x\prec_{_{A}}y-R_{_{_{\succ_{A^*}}}}^*(a)y+L_{\circ}^*(b)x)\cr&+&
			( a\prec_{_{_{A^*}}}b-R_{_{\succ_{A}}}^*(x)b+L_{\cdot}^*(y)a), z+c).
		\end{eqnarray*}
		Thus
		\begin{eqnarray*}
			(x+a)\prec (y+b)
			=(x\prec_{_{A}}y-R_{_{_{\succ_{A^*}}}}^*(a)y+L_{\circ}^*(b)x)
			+( a\prec_{_{_{A^*}}}b-R_{_{\succ_{A}}}^*(x)b+L_{\cdot}^*(y)a).
		\end{eqnarray*}
		%\begin{eqnarray*}
		%\omega(y+b, (z+c)\star(x+a))&=&
		%\omega (y+b, 
		%(z\cdot x +R_{_{\prec_{_{  A^*}}}}^*(c)x+L_{_{\succ_{_{ A^*}}}}^*(a)z)+
		%(c\circ a+ R_{_{\prec_{A}}}^*(z)a+L_{_{\succ_{A}}}^*(x)c))\cr 
		%&=& \langle y, c\circ a \rangle+  
		%\langle y, R_{_{\prec_{A}}}^*(z)a \rangle+  
		%\langle y, L_{_{\succ_{A}}}^*(x)c \rangle-\cr && 
		%\langle z\cdot x, b\rangle-
		%\langle R_{_{\prec_{_{  A^*}}}}^*(c)x, b \rangle- 
		%\langle L_{_{\succ_{_{ A^*}}}}^*(a)z, b \rangle\cr   
		%&=& \langle R_{\circ}^*(a), c\rangle+  
		%\langle y\prec_{_  A} a\rangle+  
		%\langle x\succ_{_  A} y, c\rangle- \cr 
		%&& 
		%\langle z, R_{\cdot}^*(x)b \rangle-
		%\langle x, b\prec_{_{  A^*}} c \rangle-
		%\langle z, a\succ_{_{ A^*}}\rangle\cr 
		%&=&     
		%\langle x\succ_{_  A} y+R_{\circ}(a)y-L_{_{\prec_{_{  A^*}}}}(b)x,c\rangle-
		%\langle z, a\succ_{_{ A^*}} b+R_{\cdot}(x)b-L_{_{\prec_{_  A}}}^*(y)a \rangle          
		%\end{eqnarray*}
		Similarly, we have
		\begin{eqnarray*}
			(x+a)\succ (y+b)=(x\prec{_{_A}}y+R_{\circ}^*(a)y-L_{_{\prec_{A^*}}}^*(b)x)+
			(a\prec{_{_{A^*}}}b+R_{\cdot}^*(x)b-L_{_{\prec_{A}}}^*(y)a)
		\end{eqnarray*}
		Therefore, $( A, A^*, -R_{\succ_{_A}}^*,-L^*_{\prec_{_A}} ,R^*_{\cdot},L^*_{\cdot} , 
		-R^*_{\succ_{_{A^*}}}, -L^*_{\prec_{_{A^*}}}, 
		R^*_{\circ}, L^*_{\circ})$
		is a matched pair of the pre-anti-flexible algebras $(A, \prec_{_A}, \succ_{_A})$ and 
		$(A^*,  \prec_{_{A^*}}, \succ_{_{A^*}})$. Hence \eqref{un} $\Longrightarrow$ \eqref{deux}
	\end{proof}
	\section{Pre-anti-flexible bialgebras}\label{section3}
	In this section, we are going to provide the definition of  a pre-anti-flexible bialgebra 
	and provide they equivalent notions previously announced. 
	To achieve this goal, we have
	\begin{thm}
		Let $(A, \prec_{_A}, \succ_{_A})$  be a pre-anti-flexible algebra whose products are given 
		by two linear maps 
		$\beta_{_{\succ}}^*, \beta_{_{\prec}}^*: A\otimes A\rightarrow A$. 
		Suppose in addition that there is a pre-anti-flexible algebra structure 
		"$ \prec_{_{A^*}}, \succ_{_{A^*}}$" on $A^*$ given by: 
		$\Delta_{_{\succ}}^*, \Delta_{_{\prec}}^*: A^*\otimes A^*\rightarrow A^*$. 
		Then the following relations are equivalent:
		\begin{enumerate}
			\item  $(A, A^*, R^*_{\prec_{_A}},L^*_{\succ_{_A}}, R^*_{\prec_{_{A^*}}}, L^*_{\succ_{_{A^*}}})$ 
			is  matched pair of anti-flexible algebras $aF(A)$ and $aF(A^*)$.
			\item
			The fourth linear maps $\beta_{_{\succ}}, \beta_{_{\prec}}:A^*\rightarrow A^*\otimes A^*$ and 
			$\Delta_{_{\succ}}, \Delta_{_{\prec}}:A\rightarrow A\otimes A$ satisfying
			for any $x,y\in A$ and for any $a,b\in A^*$,
			\begin{subequations}
				\begin{eqnarray}\label{eq_pre_matched_1}
					\begin{array}{llll}
						\Delta_{_{\succ}}(x\cdot y)-(R_{_{\prec_A}}(y)\otimes \id)\Delta_{_{\succ}}(x)-
						(\id\otimes L_{\cdot}(x))\Delta_{_{\succ}}(y)=\cr  
						\sigma(\id \otimes L_{_{\succ_{A}}}(y))\Delta_{_{\prec}}(x)+
						\sigma(R_{\cdot}(x)\otimes \id)\Delta_{_{\prec}}(y)-\sigma \Delta_{_{\prec}}(y\cdot x),
					\end{array}
				\end{eqnarray}
				\begin{eqnarray}\label{eq_pre_matched_3}
					\begin{array}{llll}
						&&(\sigma(L_{\cdot}(y)\otimes \id-\id\otimes R_{_{\prec_{_A}}}(y)))\Delta_{_{\prec}}(x)
						+(L_{_{\succ_{A}}}(x)\otimes\id-\id\otimes R_{\cdot}(x))\Delta_{_{\succ}}(y)=\cr&&
						(\sigma(L_{\cdot}(x)\otimes\id  -\id\otimes R_{_{\prec_{_A}}}(x) ) )\Delta_{_{\prec}}(y)+
						(L_{_{\succ_{_{A}}}}(y)\otimes\id -\id\otimes R_{{\cdot}}(y))\Delta_{_{\succ}}(x),
					\end{array}
				\end{eqnarray}
				\begin{eqnarray}\label{eq_pre_matched_2}
					\begin{array}{llll}
						\beta_{_{\succ}}(a\circ b)-
						(R_{_{{\prec_{A^*}}}}(b)\otimes \id)\beta_{_{\succ}}(a)-
						(\id\otimes L_{\circ}(a))\beta_{_{\succ}}(b)=\cr
						\sigma(\id \otimes L_{_{\succ_{A^*}}}(b))\beta_{_{\prec}}(a)+
						\sigma(R_{\circ}(a)\otimes \id)\beta_{_{\prec}}(b)-\sigma \beta_{_{\prec}}(b\circ a),
					\end{array}
				\end{eqnarray}
				\begin{eqnarray}\label{eq_pre_matched_4}
					\begin{array}{llll}
						&&(\sigma(L_{\circ}(b)\otimes \id-\id\otimes R_{_{\prec_{A^*}}}(b))\beta_{_{\prec}}(a)+
						(L_{_{\succ_{A^*}}}(a)\otimes\id-\id\otimes R_{\circ}(a))\beta_{_{\succ}}(b)=\cr &&
						(\sigma(L_{\circ}(a)\otimes\id -\id\otimes R_{_{\prec_{A^*}}}(a)  ) )\beta_{_{\prec}}(b)+
						(L_{_{\succ_{A^*}}}(b)\otimes\id-\id\otimes R_{{\circ}}(b) )\beta_{_{\succ}}(a),
					\end{array}
				\end{eqnarray}
			\end{subequations}
			where
			$R_{\cdot}=R_{_{\succ_{A}}}+R_{_{\prec_{A}}}$, 
			$L_{\cdot}=L_{_{\succ_{A}}}+L_{_{\prec_{A}}}$, 
			$R_{\circ}=R_{_{\succ_{A^*}}}+R_{_{\prec_{A^*}}}$ and 
			$L_{\circ}=L_{_{\succ_{A^*}}}+L_{_{\prec_{A^*}}}.$  
		\end{enumerate}
	\end{thm}
	\begin{proof}
		According to Remark~\ref{rmk_1}, $(R^*_{\prec_{_A}},L^*_{\succ_{_A}},  A^*)$ and 
		$(R^*_{\prec_{_{A^*}}}, L^*_{\succ_{_{A^*}}},A)$ are bimodules of $aF(A)$ and $aF(A^*)$, respectively.
		Taking into account the following, for any $x,y\in A$ and any $a,b\in A^*$
		\begin{eqnarray*}
			&&\langle \sigma\circ \Delta_{_{\prec}}(x\cdot y), a\otimes b \rangle=
			\langle \Delta_{_{\prec}}(x\cdot y), b\otimes a \rangle=
			%\left< x\cdot y, b \prec_{_{A^*}} a\right>=
			\langle x\cdot y, R_{_{\prec_{A^*}}}(a)b \rangle=
			\langle  R_{_{\prec_{A^*}}}^*(a)(x\cdot y), b \rangle, \cr 
			&& \langle \Delta_{_{\succ}}(y\cdot x), a\otimes b \rangle=
			\langle y\cdot x, a\succ_{_{A^*}} b  \rangle=
			\langle y\cdot x, L_{_{_{\succ_{A^*}}}}(a)b \rangle=
			\langle L_{_{{\succ_{A^*}}}}^*(a)(y\cdot x), b \rangle, \cr
			&&\langle (R_{_{\prec_A}}(x)\otimes \id)\Delta_{_{\succ}}(y), a\otimes b  \rangle=
			\langle \Delta_{_{\succ}}(y),R_{_{\prec_A}}^*(x)a\otimes b   \rangle=
			\langle L_{_{\succ_{A^*}}}^*(R_{_{\prec_A}}^*(x)a)y, b  \rangle,  \cr 
			&&\langle (\id\otimes L_{\cdot}(y))\Delta_{_{\succ}}(x), a\otimes b  \rangle=
			\langle x, a\succ_{_{A^*}}(L_{\cdot}^*(y)b)   \rangle=
			\langle L_{_{_{\succ_{A^*}}}}^*(a)x ,L_{\cdot}^*(y)b  \rangle=
			\langle y\cdot (L_{_{{\succ_{A^*}}}}^*(a)x) , b \rangle, \cr 
			%%
			%%%
			&&\langle \sigma(\id \otimes L_{_{\succ_{A}}}(x))\Delta_{_{\prec}}(y), a\otimes b  \rangle=
			\langle y, b\prec_{_{A^*}} (L_{_{\succ_{_{A}}}}^*(x)a)  \rangle=
			\langle R_{_{\prec_{_{A^*}}}}^*(L_{_{\succ_{_{A}}}}^*(x)a)y, b \rangle, \cr
			&&\langle \sigma(R_{\cdot}(y)\otimes \id)\Delta_{_{\prec}}(x), a\otimes b  \rangle=
			\langle x, (R_{\cdot}^*(y)b)\prec_{_{A^*}} a   \rangle=
			\langle R_{_{{\prec_{_{A^*}}}}}^*(a)x ,R_{\cdot}^*(y)b \rangle=
			\langle (R_{_{{\prec_{_{A^*}}}}}^*(a)x)\cdot y ,b \rangle,
		\end{eqnarray*}
		we deduce that Eq.~\eqref{eq_pre_matched_1} is equivalent to Eq.~\eqref{eqq1}. 
		
		%\begin{eqnarray*}
		%	&& \langle \sigma\circ \beta_{_{\prec}}(a\circ b), x\otimes y  \rangle=
		%	\langle a\circ b, y\prec_{A} x  \rangle=
		%	\langle a\circ b, R_{_{\prec_{A}}}(x)y  \rangle=
		%	\langle R_{_{\prec_{A}}}^*(x)(a\circ b), y  \rangle, \cr
		%	%%
		%	%%  
		%	&& \langle \beta_{_{\succ}}(b\circ a), x\otimes y  \rangle=
		%	\langle b\circ a, x\succ_{_{A}}y  \rangle=
		%	\langle b\circ a, L_{_{\succ_{A}}}(x)y \rangle=
		%	\langle L_{_{\succ_{A}}}^*(x)(b\circ a), y \rangle, \cr 
		%	%%
		%	%% 
		%	&& \langle (R_{_{{\prec_{A^*}}}}(a)\otimes \id)\beta_{_{\succ}}(b), x\otimes y \rangle=
		%	%\left<\beta_{_{\succ_{A}}}(b), \mathfrak{r}_{_{{\prec_{A^*}}}}^*(a)x\otimes y \right>=
		%	\langle b, L_{_{\succ_{_{A}}}}(R_{_{{\prec_{A^*}}}}^*(a)x)y  \rangle=
		%	\langle L_{_{\succ_{_{A}}}}^*(R_{_{{\prec_{A^*}}}}^*(a)x)b ,y \rangle, \cr 
		%	%%
		%	%% 
		%	&& \langle (\id\otimes L_{\circ}(b))\beta_{_{\succ}}(a), x\otimes y  \rangle=
		%	\langle a, x\succ_{_{A}}(L_{_{\circ}}^*(b)y)  \rangle=
		%	\langle L_{_{\succ_{_A}}}^*(x)a, L_{_{\circ}}^*(b)y  \rangle=
		%	\langle  b\circ(L_{_{\succ_{_A}}}^*(x)a), y \rangle,  \cr 
		%	%%
		%	%% 
		%	&& \langle \sigma(\id \otimes L_{_{\succ_{A^*}}}(a))\beta_{_{\prec}}(b), x\otimes y  \rangle=
		%	\langle b, y\prec_{A}(L_{_{\succ_{A^*}}}^*(a)x)  \rangle=
		%	\langle R_{_{\prec_{A}}}^*(L_{_{\succ_{A^*}}}^*(a)x)b, y  \rangle, \cr 
		%	%%
		%	%% 
		%	&& \langle \sigma(R_{\circ}(b)\otimes \id)\beta_{_{\prec}}(a), x\otimes y  \rangle=
		%	\langle a, (R_{_{\circ}}^*(b)y)\prec_{_A} x  \rangle=
		%	\langle R_{_{\prec_{_A}}}^*(x)a , R_{_{\circ}}^*(b)y \rangle=
		%	\langle (R_{_{\prec_{_A}}}^*(x)a) \circ b ,y  \rangle,
		%\end{eqnarray*}
		
		Similarly, we get  equivalence between Eqs.~\eqref{eq_pre_matched_2}  and \eqref{eqq2}.
		
		Besides, taking into account the following,
		\begin{eqnarray*}
			&&\langle (\sigma(L_{\cdot}(y)\otimes \id)\Delta_{_{\prec}}(x),  a\otimes b\rangle=
			\langle x, (L_{\cdot}^*(y)b) \prec_{_{A^*}} a \rangle=
			\langle R_{\prec_{_{A^*}}}^*(a)x, L_{\cdot}^*(y)b \rangle=
			\langle y\cdot (R_{\prec_{_{A^*}}}^*(a)x) ,b\rangle,	\cr 
			%&&\langle  (\sigma(L_{\cdot}(x)\otimes\id  ))\Delta_{_{\prec}}(y), a\otimes b\rangle=
			%\langle y, (L_{\cdot}^*(x)b) \prec_{_{A^*}} a \rangle=
			%\langle R_{\prec_{_{A^*}}}^*(a)y, L_{\cdot}^*(x)b \rangle=
			%\langle x\cdot (R_{\prec_{_{A^*}}}^*(a)y) ,b\rangle,	\cr 
			&&\langle (\sigma(\id\otimes R_{_{\prec_{_A}}}(y)))\Delta_{_{\prec}}(x), a\otimes b\rangle=
			\langle x, b\prec_{_{A^*}} (R_{_{\prec_{_A}}}^*(y)a)\rangle=
			%\langle x, R_{\prec_{_{A^*}}}(R_{_{\prec_{_A}}}^*(y)a)\rangle=
			\langle R_{\prec_{_{A^*}}}^*(R_{_{\prec_{_A}}}^*(y)a)x,b\rangle,	\cr 
			%&&\langle (\sigma(\id\otimes R_{_{\prec_{_A}}}(x) ) )\Delta_{_{\prec}}(y),  a\otimes b\rangle=
			%\langle y, b\prec_{_{A^*}} (R_{_{\prec_{_A}}}^*(x)a)\rangle=
			%%\langle y, R_{\prec_{_{A^*}}}(R_{_{\prec_{_A}}}^*(x)a)\rangle=
			%\langle R_{\prec_{_{A^*}}}^*(R_{_{\prec_{_A}}}^*(x)a)y,b\rangle,	\cr 
			&&\langle (L_{_{\succ_{_{A}}}}(y)\otimes\id) \Delta_{_{\succ}}(x), a\otimes b\rangle=
			\langle x, (L_{_{\succ_{_{A}}}}^*(y)a)\succ_{_{ A^*}} b \rangle=
			%\langle x, L_{\succ_{_{ A^*}}}(L_{_{\succ_{_{A}}}}^*(y)a)b \rangle=
			\langle L_{\succ_{_{ A^*}}}^*(L_{_{\succ_{_{A}}}}^*(y)a)x,b\rangle,	\cr 
			%&&\langle (L_{_{\succ_{A}}}(x)\otimes\id)\Delta_{_{\succ}}(y), a\otimes b\rangle=
			%\langle y, (L_{_{\succ_{_{A}}}}^*(x)a)\succ_{_{ A^*}} b \rangle=
			%%\langle y, L_{\succ_{_{ A^*}}}(L_{_{\succ_{_{A}}}}^*(x)a)b \rangle=
			%\langle L_{\succ_{_{ A^*}}}^*(L_{_{\succ_{_{A}}}}^*(x)a)y,b\rangle,	\cr 
			&&\langle (\id\otimes R_{{\cdot}}(y))\Delta_{_{\succ}}(x),  a\otimes b\rangle=
			\langle x, a \succ_{_{ A^*}} (R_{{\cdot}}^*(y)b) \rangle=
			\langle L_{\succ_{_{ A^*}}}^*(a)x, R_{{\cdot}}^*(y)b \rangle=
			\langle (L_{\succ_{_{ A^*}}}^*(a)x)\cdot y,b\rangle,%	\cr  
			%&&\langle (\id\otimes R_{\cdot}(x))\Delta_{_{\succ}}(y), a\otimes b\rangle=
			%\langle y, a \succ_{_{ A^*}} (R_{{\cdot}}^*(x)b) \rangle=
			%\langle L_{\succ_{_{ A^*}}}^*(a)y, R_{{\cdot}}^*(x)b \rangle=
			%\langle (L_{\succ_{_{ A^*}}}^*(a)y)\cdot x,b\rangle,		
			%%%
			%%%
			%%%
			%%%	
		\end{eqnarray*}
		we deduce  equivalence between Eqs.~\eqref{eq_pre_matched_3}  and \eqref{eqq3}.

		Similarly, we get equivalence between Eqs.~\eqref{eq_pre_matched_4} and \eqref{eqq4}.
	\end{proof}
	\begin{rmk}\label{rmk:identities}
		Let $x,y\in A$ and $a, b\in A^*$.  Setting by 
		$\Delta=\Delta_{_\prec}+\Delta_{_\succ}$ we have the following
		\begin{eqnarray*}
			\langle \beta_{_\succ}(a\circ b), x\otimes y \rangle=
			\langle a\otimes b,\Delta (x\succ y) \rangle, \quad
			\langle \sigma\beta_{_\prec}(b\circ a), x\otimes y \rangle=
			\langle a\otimes b, \sigma\Delta(y\prec x)\rangle,
		\end{eqnarray*}
		\begin{eqnarray*}
			\langle( R_{\prec{_{A^*}}}(b)\otimes \id )\beta_{_\succ}(a), x\otimes y \rangle=
			\langle a\otimes b, (R_{_{\succ}}(y)\otimes \id)\Delta_{_{\prec}}(x) \rangle,
		\end{eqnarray*}
		\begin{eqnarray*}
			\langle (\id \otimes L_{\circ}(a))\beta_{{_\succ}}(b), x\otimes y\rangle=
			\langle a\otimes b,  (\id\otimes L_{{_\succ}}(x))\Delta(y)\rangle,
		\end{eqnarray*}
		\begin{eqnarray*}
			\langle \sigma(\id \otimes L_{{_\succ}}(b))\beta_{{_\prec}}(a), x\otimes y\rangle=
			\langle  a\otimes b,(L_{_{\prec}}(y)\otimes \id)\sigma\Delta_{_{\succ}}(x) \rangle,
		\end{eqnarray*}
		\begin{eqnarray*}
			\langle \sigma(R_{\circ}(a)\otimes \id )\beta_{{_\prec}}(b), x\otimes y\rangle=
			\langle  a\otimes b, (\id \otimes R_{{_\prec}}(x))\sigma\Delta(y)\rangle, 
		\end{eqnarray*}
		\begin{eqnarray*}
			\langle (L_{{_\succ}}(a)\otimes \id )\beta_{{_\succ}}(b), x\otimes y\rangle=
			\langle  a\otimes b, (\id \otimes  R_{_{\succ}}(y))\Delta_{_{\succ}}(x)\rangle,
		\end{eqnarray*}
		\begin{eqnarray*}
			\langle (\id\otimes R_{\circ}(a))\beta_{_{\succ}}(b), x\otimes y\rangle=
			\langle  a\otimes b, (\id\otimes  L_{{_\succ}}(x))\sigma\Delta(y) \rangle,
		\end{eqnarray*}
		\begin{eqnarray*}
			\langle \sigma(L_{\circ}(b)\otimes \id)\beta_{_{\prec}} (a), x\otimes y\rangle=
			\langle  a\otimes b, (R_{_{\prec}}(x)\otimes \id )\sigma\Delta(y)\rangle,
		\end{eqnarray*}
		\begin{eqnarray*}
			\langle \sigma(\id \otimes R_{_{\prec}})\beta_{_{\prec}}(a), x\otimes y\rangle=
			\langle  a\otimes b, (L_{_{\prec}}(y)\otimes \id )\Delta_{_{\prec}}(x)\rangle.
		\end{eqnarray*}
	\end{rmk}
	Considering above identities, we derive and can prove the following lemma and theorem
	\begin{lem}
		Let $(A, \prec_{_A}, \succ_{_A})$  be a pre-anti-flexible algebra whose 
		products are given by the  linear maps 
		$\beta_{_{\succ}}^*, \beta_{_{\prec}}^*: A\otimes A\rightarrow A$. 
		Suppose in addition that there is a pre-anti-flexible algebra structure 
		"$ \prec_{_{A^*}}, \succ_{_{A^*}}$" on its dual space $A^*$ given by: 
		$\Delta_{_{\succ}}^*, \Delta_{_{\prec}}^*: A^*\otimes A^*\rightarrow A^*$. 
		Let $x,y\in A$. We have
		\begin{subequations}
			\mbox{ Eq.~\eqref{eq_pre_matched_2} is equivalent to}
			\begin{eqnarray}\label{eq_pre_matched_2'}
				\Delta (x\succ y) -(R_{_{\succ}}(y)\otimes \id)\Delta_{_{\prec}}(x) - 
				(\id\otimes L_{{_\succ}}(x))\Delta(y)=\cr
				(L_{_{\prec}}(y)\otimes \id)\sigma\Delta_{_{\succ}}(x) + 
				(\id \otimes R_{{_\prec}}(x))\sigma\Delta(y)
				-\sigma\Delta(y\prec x).
			\end{eqnarray}
			\mbox{ Eq.~\eqref{eq_pre_matched_4} is equivalent to}
			\begin{eqnarray}\label{eq_pre_matched_4'}
				(\id \otimes  R_{_{\succ}}(y))\Delta_{_{\succ}}(x)- 
				(L_{_{\prec}}(y)\otimes \id )\Delta_{_{\prec}}(x)
				+(R_{_{\prec}}(x)\otimes \id-\id\otimes  L_{{_\succ}}(x))\sigma\Delta(y)=\cr
				(R_{_{\succ}}(y) \otimes  \id)\sigma\Delta_{_{\succ}}(x)- 
				(\id\otimes L_{_{\prec}}(y) )\sigma\Delta_{_{\prec}}(x)
				+(\id \otimes R_{_{\prec}}(x)- L_{{_\succ}}(x)\otimes \id)\Delta(y).
			\end{eqnarray}
		\end{subequations}
		
	\end{lem}
	\begin{thm}
		Let $(A, \prec_{_A}, \succ_{_A})$  be a pre-anti-flexible algebra.
		Suppose in addition that there is a pre-anti-flexible algebra structure 
		"$ \prec_{_{A^*}}, \succ_{_{A^*}}$" on $A^*$ given by: 
		$\Delta_{_{\succ}}^*, \Delta_{_{\prec}}^*: A^*\otimes A^*\rightarrow A^*$. 
		Then the following relations are equivalent:
		\begin{enumerate}
			\item  $(A, A^*, R^*_{\prec_{_A}},L^*_{\succ_{_A}}, R^*_{\prec_{_{A^*}}}, L^*_{\succ_{_{A^*}}})$ 
			is  matched pair of anti-flexible algebras $aF(A)$ and $aF(A^*)$.
			\item The two linear maps 
			$\Delta_{_{\succ}}, \Delta_{_{\prec}}:A\rightarrow A\otimes A$ satisfying 
			Eqs.~\eqref{eq_pre_matched_1}, \eqref{eq_pre_matched_3}, \eqref{eq_pre_matched_2'} 
			and \eqref{eq_pre_matched_4'}.
		\end{enumerate}
	\end{thm}
	In addition we have
	\begin{defi}
		A pre-anti-flexible bialgebra structure on a vector space $A$ is given by the four linear maps 
		$\Delta_{_{\prec}}, \Delta_{_{\succ}}: A\rightarrow A\otimes A$
		and $\beta_{_{\prec}}, \beta_{_{\succ}}:A^*\rightarrow A^*\otimes A^*$, 
		where $A^*$ is the dual of $A$, such that:
		\begin{enumerate}
			\item the dual maps	$\Delta_{_{\prec}}^*, \Delta_{_{\succ}}^*: A^*\otimes A^*\rightarrow A^*$ 
			induce a pre-anti-flexible algebra structure $\prec_{_{A^*}}, \succ_{_{A^*}}$ on $A^*$,
			\item the dual maps $\beta_{_{\prec}}^*, \beta_{_{\succ}}^*:A\otimes A \rightarrow A$ 
			induce a pre-anti-flexible algebra structure $\prec_{_{A}}, \succ_{_{A}}$ on $A$,
			\item the linear maps $\Delta_{_{\prec}}, \Delta_{_{\succ}}$%, \beta_{_{\prec}}, \beta_{_{\succ}}$
			satisfy Eqs.~\eqref{eq_pre_matched_1}, \eqref{eq_pre_matched_3}, \eqref{eq_pre_matched_2'} 
			and \eqref{eq_pre_matched_4'}.
		\end{enumerate}
	\end{defi}
	\begin{thm}
		Let $(A, \prec_{_A}, \succ_{_A})$ and $(A^*, \prec_{_{A^*}}, \succ_{_{A^*}})$
		be two  pre-anti-flexible algebras. The following relations are equivalent:
		\begin{enumerate}
			\item there is an anti-flexible algebra structure on  $aF(A)\oplus aF(A^*)$ and a 
			nondegenerate skew-symmetric bilinear form $\omega$
			is given by  Eq.~\eqref{eq_skew_symmetric_form} and satisfying Eq.~\eqref{eq:simplectic_form},
			\item $(-R^*_{\succ_{_A}},-L^*_{\prec_{_A}} ,R^*_{\cdot_{_A}},L^*_{\cdot_{_A}} , 
			-R^*_{\succ_{_{A^*}}},-L^*_{\prec_{_{A^*}}}, 
			R^*_{\circ_{_{A^*}}}, L^*_{\circ_{_{A^*}}},   A, A^*)$
			is a matched pair of the  pre-anti-flexible algebras $(A, \prec_{_A}, \succ_{_A})$ and 
			$(A^*, \prec_{_{A^*}}, \succ_{_{A^*}})$,
			\item $(R^*_{\prec_{_A}},L^*_{\succ_{_A}}, R^*_{\prec_{_{A^*}}},L^*_{\succ_{_{A^*}}},A, A^*)$ 
			is a matched pair of the underlying  anti-flexible algebras $aF(A)$ and $aF(A^*)$,
			\item $(A, A^*)$ is a pre-anti-flexible bialgebra.
		\end{enumerate}	
	\end{thm}
	\begin{defi}
		A homomorphism of  pre-anti-flexible bialgebras 
		$(A, A^*, \Delta_{_{\prec_{A}}}, \Delta_{_{\succ_{A}}},\beta_{_{\prec_{A^*}}}, \beta_{_{\succ_{A^*}}})$ 
		and\\
		$(B,B^*,\Delta_{_{\prec_{B}}},\Delta_{_{\succ_{B}}},\beta_{_{\prec_{B^*}}},\beta_{_{\succ_{B^*}}})$ 
		is a homomorphism of pre-anti-flexible algebras 
		$\psi:A\rightarrow B$ such that its dual $\psi^*:B^*\rightarrow A^*$ 
		is also a homomorphism of pre-anti-flexible algebras i.e. for any $x\in A$, 
		$a\in B^*,$
		\begin{eqnarray*}
			(\psi\otimes \psi)\Delta_{_{\prec_{A}}}(x)=\Delta_{_{\prec_{B}}}(\psi(x)),\;\;
			(\psi\otimes\psi) \Delta_{_{\succ_{A}}}(x)= \Delta_{_{\succ_{B}}}(\psi(x)),
		\end{eqnarray*}
		\begin{eqnarray*}
			(\psi^*\otimes \psi^*)\beta_{_{\prec_{B^*}}}(a)=\beta_{_{\prec_{A^*}}}(\psi^*(a)),\;\;
			(\psi^*\otimes \psi^*)\beta_{_{\succ_{B^*}}}(a)=\beta_{_{\succ_{A^*}}}(\psi^*(a)).
		\end{eqnarray*}
		An invertible  homomorphism of  pre-anti-flexible bialgebras  is an isomorphism of 
		pre-anti-flexible algebras.
	\end{defi}
	\begin{rmk}\label{rmk:dual_preantiflexible}
		According to Remark~\ref{rmk:identities}, if  
		$(A, A^*, \Delta_{\succ}, \Delta_{\prec}, \beta_{{_\succ}}, \beta_{{_\prec}})$ 
		be a pre-anti-flexible 
		bialgebra, then its associated dual 
		$(A^*, A, \beta_{{_\succ}}, \beta_{{_\prec}}, \Delta_{\succ}, \Delta_{\prec})$ is also 
		a pre-anti-flexible bialgebra.
	\end{rmk}
	\section{Special pre-anti-flexible bialgebras and pre-anti-flexible Yang-Baxter equation}\label{section4}
	We deal here  with a special class of pre-anti-flexible bialgebras
	provided by the  linear maps $\Delta_{_{\prec}}, \Delta_{_{\succ}}:A\rightarrow A\otimes A$ 
	defined by for any $x\in A$
	\begin{subequations}\label{eq:coboundary}
		\begin{equation}\label{eq:coboundary_a}
			\Delta_{\succ}(x)=(\id \otimes L_{\cdot}(x))\mathrm{r}_{_\succ}+
			(R_{\prec}(x)\otimes \id)\sigma \mathrm{r}_{_\prec},
		\end{equation}
		\begin{equation}\label{eq:coboundary_b}
			\Delta_{\prec}(x)=(\id\otimes L_{\succ}(x))\mathrm{r}_{_\prec}+
			(R_{\cdot}(x)\otimes \id)\sigma \mathrm{r}_{_\succ},
		\end{equation}
	\end{subequations}
	which generate the an analogue equation of $\mathcal{D}$-equation of dendriform algebras 
	called pre-anti-flexible Yang-Baxter equation.
	In the following  of this paper, we will refer to both Eqs.~\eqref{eq:coboundary_a} and 
	\eqref{eq:coboundary_b} by  Eq.~\eqref{eq:coboundary}. Other similar referring are 
	scattered throughout this paper. 
	\begin{lem}\label{lem:sigma}
		Let $(A, \prec, \succ)$ be a pre-anti-flexible algebra and 
		$\mathrm{r}_{_\prec}, \mathrm{r}_{_\succ}\in A\otimes A$. Consider 
		$\Delta_{\prec}, \Delta_{\succ}:A\rightarrow A\otimes A$
		two linear maps defined by Eq.~\eqref{eq:coboundary}. Then for any $x\in A,$
		\begin{subequations}
			\begin{eqnarray}\label{eq:sigma_succ}
				\sigma\Delta_{\succ}(x)=(L_{\cdot}(x)\otimes \id )\sigma \mathrm{r}_{_\succ}+
				(\id\otimes R_{\prec}(x)) \mathrm{r}_{_\prec}, 
			\end{eqnarray}
			\begin{eqnarray}\label{eq:sigma_prec}
				\sigma\Delta_{\prec}(x)=(L_{\succ}(x)\otimes \id )\sigma \mathrm{r}_{_\prec}+
				(\id\otimes R_{\cdot}(x)) \mathrm{r}_{_\succ}, 
			\end{eqnarray}
			\begin{eqnarray}\label{eq:delta}
				\Delta(x)=(\id \otimes L_{\cdot}(x))\mathrm{r}_{_\succ}+
				(R_{\prec}(x)\otimes \id)\sigma \mathrm{r}_{_\prec}+
				(\id\otimes L_{\succ}(x))\mathrm{r}_{_\prec}+
				(R_{\cdot}(x)\otimes \id)\sigma \mathrm{r}_{_\succ},
			\end{eqnarray}
			\begin{eqnarray}\label{eq:sigma_delta}
				\sigma\Delta(x)=
				(L_{\cdot}(x)\otimes \id )\sigma \mathrm{r}_{_\succ}+
				(\id\otimes R_{\prec}(x)) \mathrm{r}_{_\prec}+
				(L_{\succ}(x)\otimes \id )\sigma \mathrm{r}_{_\prec}+
				(\id\otimes R_{\cdot}(x)) \mathrm{r}_{_\succ}.
			\end{eqnarray}
		\end{subequations}
	\end{lem}
	As consequences following the definition given by
	Eq.~\eqref{eq:coboundary} we have
	\begin{itemize}
		\item  For any  $x,y \in A$. By Eq.~\eqref{eq:sigma_prec} we have 
		\begin{eqnarray*}
			\Delta_{\succ}(x\cdot y)+\sigma\Delta_{\prec}(y\cdot x)= 
			(\id \otimes(L_{\cdot}(x\cdot y)+R_{\cdot}(y\cdot x)))\mathrm{r}_{_\succ}+
			((R_{\prec}(x\cdot y)+L_{\succ}(y\cdot x))\otimes\id)\sigma \mathrm{r}_{_\prec}.
		\end{eqnarray*}
		According to Eq.~\eqref{eq:useful} we have 
		\begin{eqnarray*}
			\Delta_{\succ}(x\cdot y)+\sigma\Delta_{\prec}(y\cdot x)&=&
			(\id \otimes(L_{\cdot}(x)L_{\cdot}(y)+R_{\cdot}(x)R_{\cdot}(y)))\mathrm{r}_{_\succ}\cr&+&
			((R_{\prec}(y)R_{\prec}(x)+L_{\succ}(y)L_{\succ}( x))\otimes\id)\sigma \mathrm{r}_{_\prec}.
		\end{eqnarray*}
		In addition
		\begin{eqnarray*}
			(R_{\prec}(y)\otimes\id)\Delta_{\succ}(x)&=&
			(R_{\prec}(y)\otimes L_{\cdot}(x))\mathrm{r}_{_\succ}+
			(R_{\prec}(y)R_{\prec}(x)\otimes \id)\sigma \mathrm{r}_{_\prec}\cr 
			(\id \otimes L_{\cdot}(x))\Delta_{\succ}(y)&=&
			(\id\otimes L_{\cdot}(x)L_{\cdot}(y))\mathrm{r}_{_\succ}+
			(R_{\prec}(y)\otimes L_{\cdot}(x))\sigma \mathrm{r}_{_\prec}\cr
			\sigma(\id \otimes L_{\succ}(y))\Delta_{\prec}(x)&=&
			(L_{\succ}(y)L_{\succ}(x)\otimes \id )\sigma \mathrm{r}_{_\prec}+
			(L_{\succ}(y)\otimes R_{\cdot}(x) )  \mathrm{r}_{_\succ}\cr
			\sigma(R_{\cdot}(x)\otimes \id )\Delta_{\prec}(y)&=&
			(L_{\succ}(y)\otimes R_{\cdot} (x))\sigma \mathrm{r}_{_\prec}+
			(\id \otimes R_{\cdot}(x)R_{\cdot}(y))\mathrm{r}_{_\succ}.
		\end{eqnarray*}
		Thus
		\begin{eqnarray*}
			&&(R_{\prec}(y)\otimes\id)\Delta_{\succ}(x)+
			(\id \otimes L_{\cdot}(x))\Delta_{\succ}(y)-	
			\Delta_{\succ}(x\cdot y)+
			\sigma(\id \otimes L_{\succ}(y))\Delta_{\prec}(x)-
			\sigma\Delta_{\prec}(y\cdot x)\cr&&+
			\sigma(R_{\cdot}(x)\otimes \id )\Delta_{\prec}(y)=
			(R_{\prec}(y)  \otimes L_{\cdot}(x)  +  
			L_{\succ}(y) \otimes R_{\cdot}(x) )(\mathrm{r}_{_\succ}+\sigma \mathrm{r}_{_\prec}).
		\end{eqnarray*}
		Therefore
		Eq.~\eqref{eq_pre_matched_1} is equivalent to the following
		\begin{eqnarray}{\label{eq:coboundary1}}
			(R_{_\prec}(y)\otimes L_{\cdot}(x)+L_{\succ}(y)\otimes R_{\cdot}(x))(\mathrm{r}_{_\succ}+
			\sigma \mathrm{r}_{_\prec})=0, \;\; \forall x,y\in A.
		\end{eqnarray}
		\item Besides, for any $x, y\in A$ we have
		\begin{eqnarray*}
			(L_{\succ}(x)\otimes \id -\id \otimes R_{\cdot}(x))\Delta_{_\succ}(y)&=& 
			(L_{\succ}(x)\otimes L_{\cdot}(y))\mathrm{r}_{_\succ}+
			(L_{\succ}(x)R_{\prec}(y)\otimes\id )\sigma \mathrm{r}_{_\prec}\cr
			&-&(\id\otimes R_{\cdot}(x)L_{\cdot}(y) )\mathrm{r}_{_\succ}-
			(R_{\prec}(y)\otimes R_{\cdot}(x))\sigma \mathrm{r}_{_\prec}\cr
			\sigma(L_{\cdot}(y)\otimes \id -\id \otimes R_{\prec}(y) )\Delta_{\prec}(x)&=&
			(L_{\succ}(x)\otimes L_{\cdot}(y))\sigma \mathrm{r}_{_\prec}+
			(\id\otimes L_{\cdot}(y)R_{\cdot}(x))\mathrm{r}_{_\succ}\cr 
			&-&(R_{\prec}(y)L_{\succ}(x)\otimes\id)\sigma \mathrm{r}_{_\prec}-
			(R_{\prec}(y)\otimes R_{\cdot}(x))\mathrm{r}_{_\succ}
		\end{eqnarray*}
		Then 
		\begin{eqnarray*}
			&&(L_{\succ}(x)\otimes \id -\id \otimes R_{\cdot}(x))\Delta_{_\succ}(y)+
			\sigma(L_{\cdot}(y)\otimes \id -\id \otimes R_{\prec}(y) )\Delta_{\prec}(x)\cr
			&&-(L_{\succ}(y)\otimes \id -\id \otimes R_{\cdot}(y))\Delta_{_\succ}(x)-
			\sigma(L_{\cdot}(x)\otimes \id -\id \otimes R_{\prec}(x) )\Delta_{\prec}(y)\cr&&=
			(L_{\succ}(x)\otimes L_{\cdot}(y)-R_{\prec}(y)\otimes R_{\cdot}(x)-
			L_{\succ}(y)\otimes L_{\cdot}(x)+
			R_{\prec}(x)\otimes R_{\cdot}(y))(\mathrm{r}_{_\succ}+\sigma \mathrm{r}_{_\prec})\cr 
			&&+ (([L_{\succ}(x), R_{\prec}(y)]-
			[L_{\succ}(y), R_{\prec}(x)])\otimes\id)\sigma \mathrm{r}_{_\prec}+ 
			(\id\otimes ([L_{\cdot}(y), R_{\cdot}(x)]-
			[L_{\cdot}(x), R_{\cdot}(y)]))\mathrm{r}_{_\succ}\cr
			&&=(L_{\succ}(x)\otimes L_{\cdot}(y)-R_{\prec}(y)\otimes R_{\cdot}(x)-
			L_{\succ}(y)\otimes L_{\cdot}(x)+
			R_{\prec}(x)\otimes R_{\cdot}(y))(\mathrm{r}_{_\succ}+\sigma \mathrm{r}_{_\prec})
		\end{eqnarray*}
		Note that  the last equal sign in above equation is due to 
		Eq.~\eqref{eq_bimodule_pre_anti_flexible1} and Eq.~\eqref{eq:useful1}. 
		
		Therefore, Eq.~\eqref{eq_pre_matched_3} is equivalent to the following
		\begin{eqnarray}{\label{eq:coboundary2}}
			(L_{\succ}(x)\otimes L_{\cdot}(y)-R_{\prec}(y)\otimes R_{\cdot}(x)-
			L_{\succ}(y)\otimes L_{\cdot}(x)+R_{\prec}(x)\otimes R_{\cdot}(y))(\mathrm{r}_{_\succ}+
			\sigma \mathrm{r}_{_\prec})=0,\; \forall x,y\in A.
		\end{eqnarray}
		\item Furthermore, by Eqs.~\eqref{eq:delta} and 
		\eqref{eq:sigma_delta} we have for any $x,y\in A$
		\begin{eqnarray*}
			\Delta(x\succ y)+\sigma\Delta(y\prec x)&=&
			(\id \otimes (L_{\cdot}(x\succ y)+R_{\cdot}(y\prec x)))\mathrm{r}_{_\succ}+
			(\id \otimes(R_{_{\prec}}(y\prec x)+L_{_{\succ}}(x\succ y)))\mathrm{r}_{_\prec}\cr&+&
			((R_{_{\prec}}(x\succ y)+L_{_{\succ}}(y\prec x))\otimes \id)\sigma \mathrm{r}_{_\prec}+
			((R_{\cdot}(x\succ y)+L_{\cdot}(y\prec x))\otimes \id)\sigma \mathrm{r}_{_\succ}
		\end{eqnarray*}
		Taking into account to Eqs.~\eqref{eq_bimodule_pre_anti_flexible2} and 
		\eqref{eq_bimodule_pre_anti_flexible5} we have  
		\begin{eqnarray*}
			\Delta(x\succ y)+\sigma\Delta(y\prec x)&=&
			(\id\otimes (L_{_{\succ}}(x\succ y)+
			R_{_{\prec}}(y\prec x)))(\mathrm{r}_{_\succ}+\mathrm{r}_{_\prec})\cr
			&+&
			((L_{_{\succ}}(y\prec x)+R_{_{\prec}}(x\succ y))\otimes \id)(\sigma \mathrm{r}_{_\succ}+
			\sigma \mathrm{r}_{_\prec})\cr 
			&+&
			(\id \otimes (L_{_{\succ}}(x)L_{_{\prec}}(y)+
			R_{_{\prec}}(x)R_{_{\succ}}(y)))\mathrm{r}_{_\succ}+\cr
			&+&
			((L_{_{\prec}}(y)L_{\cdot}(x)+R_{_{\succ}}(y)R_{\cdot}(x))\otimes \id )\sigma \mathrm{r}_{_\succ}.
		\end{eqnarray*}
		Using Eqs.~ \eqref{eq:sigma_succ} , \eqref{eq:delta} and 
		\eqref{eq:sigma_delta} we have for any $x,y\in A$
		\begin{eqnarray*}
			(R_{_{\succ}}(y)\otimes \id)\Delta_{_{\prec}}(x)&=&
			(R_{_{\succ}}(y)\otimes L_{_{\succ}}(x))\mathrm{r}_{_{\prec}}+
			(R_{_{\succ}}(y)R_{\cdot}(x)\otimes\id)\sigma \mathrm{r}_{_{\succ}}\cr 
			(L_{_{\prec}}(y)\otimes \id )\sigma\Delta_{_{\succ}}(x)&=&
			(L_{_{\prec}}(y)L_{\cdot}(x)\otimes\id)\sigma \mathrm{r}_{_{\succ}}+
			(L_{_{\prec}}(y)\otimes R_{_{\prec}}(x))\mathrm{r}_{_{\prec}}\cr 
			(\id \otimes L_{_{\succ}} (x))\Delta(y)&=&
			(\id\otimes L_{_{\succ}}(x)L_{\cdot}(y))\mathrm{r}_{_{\succ}}+
			(R_{_{\prec}}(y)\otimes L_{_{\succ}}(x)) \sigma \mathrm{r}_{_{\prec}}
			\cr&+&
			(\id\otimes L_{_{\succ}}(x)L_{_{\succ}}(y))\mathrm{r}_{_{\prec}}+
			(R_{\cdot}(y)\otimes L_{_{\succ}}(x))\sigma \mathrm{r}_{_{\succ}}\cr 
			(\id \otimes R_{_{\prec}}(x) )\sigma\Delta(y)&=&
			(L_{\cdot}(y)\otimes R{_{\prec}}(x))\sigma \mathrm{r}_{_{\succ}}+
			(\id\otimes R_{_{\prec}}(x)R_{_{\prec}}(y))\mathrm{r}_{_{\prec}}\cr
			&+&
			(L_{_{\succ}}(y)\otimes R_{_{\prec}}(x)) \sigma \mathrm{r}_{_{\prec}}+
			(\id\otimes R_{_{\prec}}(x)R_{\cdot}(y)) \mathrm{r}_{_{\succ}}
		\end{eqnarray*}
		Thus
		\begin{eqnarray*}
			&&(R_{_{\succ}}(y)\otimes \id)\Delta_{_{\prec}}(x)+
			(L_{_{\prec}}(y)\otimes \id )\sigma\Delta_{_{\succ}}(x)+
			(\id \otimes L_{_{\succ}} (x))\Delta(y)+
			(\id \otimes R_{_{\prec}}(x) )\sigma\Delta(y)
			\cr&&=
			(\id\otimes(L_{_{\succ}}(x)L_{_{\succ}}(y)+
			R_{_{\prec}}(x)R_{_{\prec}}(y)) )(\mathrm{r}_{_{\succ}}+\mathrm{r}_{_{\prec}})+
			((R_{_{\succ}}(y)R_{\cdot}(x)+
			L_{_{\prec}}(y)L_{\cdot}(x))\otimes\id  )\sigma \mathrm{r}_{_{\succ}}\cr&&+
			(R_{_{\prec}}(y)\otimes L_{_{\succ}}(x)   +
			L_{_{\succ}}(y)\otimes R_{_{\prec}}(x) )(\sigma \mathrm{r}_{_{\succ}}+
			\sigma \mathrm{r}_{_{\prec}})+
			(\id \otimes(L_{_{\succ}}(x)L_{_{\prec}}(y)+
			R_{_{\prec}}(x)R_{_{\succ}}(y)) )\mathrm{r}_{_{\succ}}\cr&&+
			(R_{_{\succ}}(y)\otimes L_{_{\succ}}(x)  +
			L_{_{\prec}}(y)\otimes R_{_{\prec}}(x) )(\mathrm{r}_{_{\prec}}+
			\sigma \mathrm{r}_{_{\succ}})
		\end{eqnarray*}
		Therefore, using Eq.~\eqref{eq_bimodule_pre_anti_flexible3} we deduce that 
		Eq.~\eqref{eq_pre_matched_2'} is equivalent to the following
		\begin{eqnarray}{\label{eq:coboundary3}}
			\begin{array}{llll}
				&&
				(R_{_{\succ}}(y)\otimes L_{_{\succ}}(x) +
				L_{_{\prec}}(y) \otimes R_{_{\prec}}(x) )(\mathrm{r}_{_{\prec}}+
				\sigma \mathrm{r}_{_{\succ}})\cr&&+
				(R_{_{\prec}}(y)\otimes L_{_{\succ}}(x) +
				L_{_{\succ}}(y)\otimes R_{_{\prec}}(x) )(\sigma \mathrm{r}_{_{\succ}}+
				\sigma \mathrm{r}_{_{\prec}})\cr&&+
				(\id\otimes(L_{_{\succ}}(x\prec y)+R_{_{\prec}}(y\succ x)))(\mathrm{r}_{_{\succ}}+\mathrm{r}_{_{\prec}})\cr&&-
				((L_{_{\succ}}(y\prec x)+
				R_{_{\prec}}(x\succ y))\otimes \id)(\sigma \mathrm{r}_{_{\succ}}+\sigma \mathrm{r}_{_{\prec}})=0.
			\end{array}
		\end{eqnarray}
		\item Finally, in view of Lemma~\ref{lem:sigma} we have  for any $x,y\in A$
		\begin{eqnarray*}
			(\id \otimes R_{_{\succ}}(y))\Delta_{_{\succ}}(x)&=&
			(\id\otimes R_{_{\succ}}(y)L_{\cdot}(x))\mathrm{r}_{_{\succ}}+
			(R_{_{\prec}}(x)\otimes R_{_{\succ}}(y))\sigma \mathrm{r}_{_{\prec}}
			\cr 
			(L_{_{\prec}}(y)\otimes \id)\Delta_{_{\prec}}(x)&=&
			(L_{_{\prec}}(y)\otimes L_{_{\succ}}(x))\mathrm{r}_{_{\prec}}+
			(L_{_{\prec}}(y)R_{\cdot}(x)\otimes \id)\sigma \mathrm{r}_{_{\succ}}
			\cr
			(R_{_{\prec}}(x)\otimes \id)\sigma \Delta(y)&=&
			(R_{_{\prec}}(x)L_{\cdot}(y)\otimes\id )\sigma \mathrm{r}_{_{\succ}}+
			(R_{_{\prec}}(x)\otimes R_{_{\prec}}(y))\mathrm{r}_{_{\prec}}\cr&+&
			(R{_{\prec}}(x)L_{_{\succ}}(y)\otimes \id)\sigma \mathrm{r}_{_{\prec}}+
			(R_{_{\prec}}(x)\otimes R_{\cdot}(y) )\mathrm{r}_{_{\succ}}
			\cr
			(\id \otimes L_{_{\succ}}(x))\sigma \Delta(y)&=&
			(L_{\cdot}(y)\otimes L_{_{\succ}}(x) )\sigma \mathrm{r}_{_{\succ}}+
			(\id\otimes L_{_{\succ}}(x)R_{_{\prec}}(y))\mathrm{r}_{_{\prec}}\cr&+&
			(L_{_{\succ}}(y)\otimes L_{_{\succ}}(x) )\sigma \mathrm{r}_{_{\prec}}+
			(\id\otimes L_{_{\succ}}(x)R_{\cdot}(y) )\mathrm{r}_{_{\succ}}
			\cr
			(R_{_{\succ}}(y)\otimes \id)\sigma\Delta_{_{\succ}}(x)&=&
			(R_{_{\succ}}(y)L_{\cdot}(x)\otimes\id)\sigma \mathrm{r}_{_{\succ}}+
			(R_{_{\succ}}(y)\otimes R_{_{\prec}}(x))\mathrm{r}_{_{\prec}}\cr
			(\id \otimes L_{_{\prec}}(y))\sigma\Delta_{_{\prec}}(x)&=&
			(L_{_{\succ}}(x)\otimes L_{_{\prec}}(y))\sigma \mathrm{r}_{_{\prec}}+
			(\id\otimes L_{_{\prec}}(y)R_{\cdot}(x))\mathrm{r}_{_{\succ}}\cr
			(\id \otimes R_{_{\prec}}(x))\Delta(y)&=&
			(\id\otimes R_{_{\prec}}(x)L_{\cdot}(y))\mathrm{r}_{_{\succ}}+
			(R_{_{\prec}}(y)\otimes R_{_{\prec}}(x))\sigma \mathrm{r}_{_{\prec}}\cr
			&+&
			(\id\otimes R_{_{\prec}}(x)L_{_{\succ}}(y))\mathrm{r}_{_{\prec}}+
			(R_{\cdot}(y)\otimes R_{_{\prec}}(x))\sigma \mathrm{r}_{_{\succ}}\cr
			(L_{_{\succ}}(x)\otimes \id)\Delta(y)&=&
			(L_{_{\succ}}(x)\otimes L_{\cdot}(y))\mathrm{r}_{_{\succ}}+
			(L_{_{\succ}}(x)R_{_{\prec}}(y)\otimes \id)\sigma \mathrm{r}_{_{\prec}}\cr
			&+&
			(L_{_{\succ}}(x)\otimes L_{_{\succ}}(y))\mathrm{r}_{_{\prec}}+
			(L_{_{\succ}}(x)R_{\cdot}(y)\otimes\id )\sigma \mathrm{r}_{_{\succ}}
		\end{eqnarray*}
		According to Eq.~\eqref{eq_bimodule_pre_anti_flexible4} we have for any $x,y\in A$
		\begin{eqnarray*}
			&&(\id \otimes  R_{_{\succ}}(y))\Delta_{_{\succ}}(x)- 
			(L_{_{\prec}}(y)\otimes \id )\Delta_{_{\prec}}(x)
			+(R_{_{\prec}}(x)\otimes \id-\id\otimes  L_{{_\succ}}(x))\sigma\Delta(y)\cr
			&&-(R_{_{\succ}}(y) \otimes  \id)\sigma\Delta_{_{\succ}}(x)+ 
			(\id\otimes L_{_{\prec}}(y) )\sigma\Delta_{_{\prec}}(x)
			-(\id \otimes R_{_{\prec}}(x)- L_{{_\succ}}(x)\otimes \id)\Delta(y)\cr
			&=&((R_{_{\prec}}(x)\otimes R_{_{\prec}}(x)+ L_{_{\succ}}(x)\otimes L_{_{\succ}}(y))
			-(\id\otimes (L_{_{\succ}}(x)R_{_{\prec}}(y)+
			R_{_{\prec}}(x)L_{_{\succ}}(y))))(\mathrm{r}_{_{\succ}}+\mathrm{r}_{_{\prec}})\cr
			&+&
			((R_{_{\prec}}(x)L_{_{\succ}}(y)+L_{_{\succ}}(x)R_{_{\prec}}(y))\otimes \id-
			(L_{_{\succ}}(y)\otimes L_{_{\succ}}(x) +
			R_{_{\prec}}(y)\otimes R_{_{\prec}}(x)) )(\sigma \mathrm{r}_{_{\prec}}+
			\sigma \mathrm{r}_{_{\succ}})\cr
			&+&
			(R_{_{\prec}}(x)\otimes R_{_{\succ}}(y)+L_{_{\succ}}(x)\otimes
			L_{_{\prec}}(y))(\mathrm{r}_{_{\succ}}+\sigma \mathrm{r}_{_{\prec}})-
			(L_{_{\prec}}(y)\otimes L_{_{\succ}}(x)+R_{_{\succ}}(y)\otimes
			R_{_{\prec}}(x))(\mathrm{r}_{_{\prec}}+\sigma \mathrm{r}_{_{\succ}})
		\end{eqnarray*}
		Therefore, Eq.~\eqref{eq_pre_matched_4'} is equivalent to the following
		\begin{eqnarray}{\label{eq:coboundary4}}
			&&0=((R_{_{\prec}}(x)\otimes R_{_{\prec}}(x)+ 
			L_{_{\succ}}(x)\otimes L_{_{\succ}}(y))
			-(\id\otimes(L_{_{\succ}}(x)R_{_{\prec}}(y)+
			R_{_{\prec}}(x)L_{_{\succ}}(y))))(\mathrm{r}_{_{\succ}}+\mathrm{r}_{_{\prec}})\cr
			&&+
			((R_{_{\prec}}(x)L_{_{\succ}}(y)+L_{_{\succ}}(x)R_{_{\prec}}(y))\otimes \id-
			(L_{_{\succ}}(y)\otimes L_{_{\succ}}(x) +
			R_{_{\prec}}(y)\otimes R_{_{\prec}}(x)) )(\sigma \mathrm{r}_{_{\prec}}+\sigma \mathrm{r}_{_{\succ}})\\
			&&+
			(R_{_{\prec}}(x)\otimes R_{_{\succ}}(y)+
			L_{_{\succ}}(x)\otimes L_{_{\prec}}(y))(\mathrm{r}_{_{\succ}}+\sigma \mathrm{r}_{_{\prec}})-
			(L_{_{\prec}}(y)\otimes L_{_{\succ}}(x)+
			R_{_{\succ}}(y)\otimes R_{_{\prec}}(x))(\mathrm{r}_{_{\prec}}+\sigma \mathrm{r}_{_{\succ}})\nonumber
		\end{eqnarray}
	\end{itemize}
	Clearly, we have provided the proof of the following theorem
	\begin{thm}
		Let $(A, \prec, \succ)$ be a pre-anti-flexible algebra and 
		$\mathrm{r}_{_\prec}, \mathrm{r}_{_\succ}\in A\otimes A$. Consider 
		$\Delta_{_\prec}, \Delta_{_\succ}:A\rightarrow A\otimes A$
		two linear maps defined by Eq.~\eqref{eq:coboundary} such that their dual maps
		$\Delta_{\prec}^*, \Delta_{\succ}^*:A^*\otimes A^*\rightarrow A^*$
		define a pre-anti-flexible algebra on $A^*$. 
		Then $(A, A^*)$ is a pre-anti-flexible bialgebra if and only if 
		$\Delta_{\prec}, \Delta_{\succ}$ satisfying  
		Eqs.~\eqref{eq:coboundary1}~-~\eqref{eq:coboundary4}.
	\end{thm}
	\begin{lem}
		Let $A$ be a vector space and let
		$\Delta_{_\prec}, \Delta_{\succ} :A\rightarrow A\otimes A$ be two linear maps. 
		Then $\Delta_{_{\prec}}^*, \Delta_{_{\succ}}^*:A^*\otimes A^*\rightarrow A^*$ 
		define a pre-anti-flexible algebra structure 
		on $A^*$ if and only if the following conditions are satisfied
		\begin{subequations}
			\begin{eqnarray}\label{eq:rmatrix1}
				(\Delta_{_\succ}\otimes\id )\Delta_{_{\prec}}-
				(\id\otimes\Delta_{_\prec})\Delta_{_{\succ}}=
				(\id\otimes\sigma\Delta_{_\succ})\sigma\Delta_{_{\prec}}-
				(\sigma\Delta_{_{\prec}}\otimes\id)\sigma\Delta_{_{\succ}},
			\end{eqnarray}
			\begin{eqnarray}\label{eq:rmatrix2}
				((\Delta_{_{\prec}}+\Delta_{_{\succ}})\otimes\id)\Delta_{\succ}-
				(\id\otimes\Delta_{\succ})\Delta_{\succ}=
				(\id\otimes \sigma\Delta_{_{\prec}})\sigma\Delta_{_\prec}-
				(\sigma(\Delta_{_{\prec}}+\Delta_{_{\succ}})\otimes\id)\sigma\Delta_{_\prec}.
			\end{eqnarray}
		\end{subequations}
	\end{lem}
	\begin{proof}
		Denote by "$\prec_{_{A^*}}, \succ_{_{A^*}}$" the bilinear products on 
		$A^*$ defined respectively by
		$\Delta_{_{\prec}}, \Delta_{_{\succ}}$, i.e.
		for any $x\in A$ and for any $a, b\in A^*$ 
		\begin{equation*}
			\langle a\prec_{_{A^*}} b, x \rangle=\langle \Delta_{_{\prec}}^*(a\otimes b), x \rangle
			=\langle a\otimes b,\Delta_{_{\prec}}(x)\rangle;\;\;
			\langle a\succ_{_{A^*}} b, x\rangle=\langle \Delta_{_{\succ}}^*(a\otimes b), x  \rangle=
			\langle a\otimes b,   \Delta_{_{\succ}}(x)\rangle.
		\end{equation*}
		Furthermore, according to Eq.~\eqref{eq:biasso}  for any $a,b, c\in A^*$ 
		and any $x\in A$, we have 
		\begin{eqnarray*}
			\langle (a,b,c)_{_m}, x \rangle
			&=&\langle (a\succ_{_{A^*}} b)\prec_{_{A^*}} c-a\succ_{_{A^*}} 
			(b\prec_{_{A^*}} c), x\rangle \cr&=&
			\langle(\Delta_{_{\prec}}^*(\Delta_{_\succ}^*\otimes\id )- 
			\Delta_{_{\succ}}^*(\id\otimes\Delta_{_\prec}^*))(a\otimes b\otimes c), x\rangle\cr
			\langle (a,b,c)_{_m}, x \rangle&=&\langle a\otimes b\otimes c, 
			( (\Delta_{_\succ}\otimes\id )\Delta_{_{\prec}}-
			(\id\otimes\Delta_{_\prec})\Delta_{_{\succ}})(x)\rangle,\cr
			\langle (c,b,a)_{_m}, x \rangle&=&\langle (c\succ_{_{A^*}} b)\prec_{_{A^*}} a-
			c\succ_{_{A^*}} (b\prec_{_{A^*}} a), x\rangle \cr&=&
			\langle( \Delta_{_{\prec}}^*\sigma(\id\otimes\Delta_{_\succ}^*\sigma)-
			\Delta_{_{\succ}}^*\sigma(\Delta_{_{\prec}}^*\sigma\otimes\id))
			(a\otimes b\otimes c), x\rangle\cr
			\langle (c,b,a)_{_m}, x \rangle&=&\langle a\otimes b\otimes c, 
			((\id\otimes\sigma\Delta_{_\succ})\sigma\Delta_{_{\prec}}-
			(\sigma\Delta_{_{\prec}}\otimes\id)\sigma\Delta_{_{\succ}})(x)\rangle,\cr
			\langle (a,b,c)_{_l}, x \rangle&=&\langle  (a \prec_{_{A^*}}b+ 
			a\succ_{_{A^*}} b) \succ_{_{A^*}} c- a\succ_{_{A^*}} (b \succ_{_{A^*}} c), x\rangle \cr
			&=&\langle(\Delta_{\succ}^*((\Delta_{_{\prec}}^*+\Delta_{_{\succ}}^*)\otimes\id)-
			\Delta_{\succ}^*(\id\otimes\Delta_{\succ}^*))(a\otimes b\otimes c) ,
			x\rangle\cr&=&\langle a\otimes b\otimes c,
			(((\Delta_{_{\prec}}+\Delta_{_{\succ}})\otimes\id)\Delta_{\succ}-
			(\id\otimes\Delta_{\succ})\Delta_{\succ})(x) \rangle,\cr
			\langle (c,b,a)_{_r}, x \rangle&=&  \langle (c\prec_{_{A^*}} b) \prec_{_{A^*}}a-
			c\prec_{_{A^*}}(b\prec_{_{A^*}} a+b\succ_{_{A^*}} a), x\rangle\cr
			&= &\langle ((\id\otimes \Delta_{_{\prec}}^*\sigma)\Delta_{_\prec}^*\sigma-
			((\Delta_{_{\prec}}^*+\Delta_{_{\succ}}^*)\sigma\otimes\id)
			\Delta_{_\prec}^*\sigma)(a\otimes b\otimes c),
			x\rangle\cr
			&=&\langle  a\otimes b\otimes c, ((\id\otimes \sigma\Delta_{_{\prec}})\sigma\Delta_{_\prec}-
			(\sigma(\Delta_{_{\prec}}+\Delta_{_{\succ}})\otimes\id)\sigma\Delta_{_\prec})(x)\rangle.
		\end{eqnarray*}
		Therefore, $(A^*, \prec_{_{A^*}}, \succ_{_{A^*}})$ is a 
		pre-anti-anti-flexible algebra if and only if 
		Eqs.~\eqref{eq:rmatrix1} and \eqref{eq:rmatrix2} are satisfied.
	\end{proof}
	For a given pre-anti-flexible algebra  $(A,\prec, \succ)$ and 
	two elements $\mathrm{r}_{_\prec}, \mathrm{r}_{_\succ}$ in $A\otimes A$ given by 
	$\displaystyle \mathrm{r}_{_\succ}=\sum_i{a_i\otimes b_i}$ and 
	$\displaystyle \mathrm{r}_{_\prec}=\sum_i{c_i\otimes d_i},$
	for any $a_i, b_i, c_i$ and $d_i$ in $A$, we designate  by
	\begin{equation*}
		\mathrm{r}_{_{\succ, 12}}=\sum_ia_i\otimes b_i\otimes 1,\quad
		\mathrm{r}_{_{\succ, 13}}=\sum_{i}a_i\otimes 1\otimes b_i,
		\quad \mathrm{r}_{_{\succ, 23}}=\sum_i1\otimes a_i\otimes b_i, 
		\mbox{ etc} \cdots, 
	\end{equation*}
	\begin{equation*}
		\mathrm{r}_{_{\prec, 12}}=\sum_i c_i\otimes d_i\otimes 1,\quad
		\mathrm{r}_{_{\prec, 13}}=\sum_{i}c_i\otimes 1\otimes d_i,
		\quad \mathrm{r}_{_{\prec, 23}}=\sum_i1\otimes c_i\otimes d_i,
		\mbox{ etc} \cdots, 
	\end{equation*}
	where $1$ is the unit element if $(A,\prec, \succ)$ unitary, otherwise is a
	symbol playing a similar role as that of the unit element on $A$. Then  operations
	between two $\mathrm{r}_{_{\prec, ..}}, \mathrm{r}_{_{\succ, ..}}$ are in an obvious way. For instance,
	\begin{equation*}
		\mathrm{r}_{_{\succ,12}}\cdot \mathrm{r}_{_{\succ, 13}}=
		\sum_{i,j}a_i\cdot a_j\otimes b_i\otimes b_j,
		\mathrm{r}_{_{\succ, 13}}\prec \mathrm{r}_{_{\succ, 23}}=
		\sum_{i,j}a_i\otimes a_j\otimes b_i\prec b_j,
		\mathrm{r}_{_{\prec, 23}}\succ \mathrm{r}_{_{\prec, 12}}=
		\sum_{i,j}c_j\otimes c_i\succ d_j\otimes d_i, 
		\mbox{ etc}
	\end{equation*}
	and similarly
	\begin{equation*}
		\mathrm{r}_{_{\prec, 12}}\succ \mathrm{r}_{_{\succ, 13}}=
		\sum_{i,j}c_i\succ a_j\otimes d_i\otimes b_j,
		\mathrm{r}_{_{\succ,13}}\prec \mathrm{r}_{_{\prec, 23}}=
		\sum_{i,j}a_i\otimes c_j\otimes b_i\prec d_j,
		\mathrm{r}_{_{\prec, 23}}\cdot \mathrm{r}_{_{\succ,12}}=
		\sum_{i,j}a_j\otimes c_i\cdot b_j\otimes d_i,
	\end{equation*}
	and so on.
	\begin{pro}
		Let $(A, \prec, \succ)$ be a pre-anti-flexible algebra and 
		$\mathrm{r}_{_\prec}, \mathrm{r}_{_\succ} \in A\otimes A$. Define 
		$\Delta_{_{\prec}}, \Delta_{_{\succ}}:A\rightarrow A\otimes A$ by 
		Eqs.~\eqref{eq:coboundary1} and \eqref{eq:coboundary2}.
		Then $\Delta_{_\prec}^*, \Delta_{_{\succ}}^*: A^*\otimes A^*\rightarrow A^*$ 
		define a pre-anti-flexible algebra structure on $A^*$
		if and only if the following equations are satisfied for any $x\in A$
		\begin{subequations}
			\begin{eqnarray}\label{eq:ybe1}
				&&(\id\otimes\id\otimes L_{\succ}(x))M(\mathrm{r})-
				(R_{\prec}(x)\otimes \id\otimes\id)N(\mathrm{r})\cr
				&&+(\id\otimes \id \otimes R_{\prec}(x))P(\mathrm{r})-
				(L_{\succ}(x)\otimes\id \otimes \id)Q(\mathrm{r})=0,
			\end{eqnarray}
			\begin{eqnarray}\label{eq:ybe2}
				\begin{array}{llll}
					&&(\id \otimes\id\otimes L_{\cdot}(x))M'(\mathrm{r})
					+(\id \otimes \id \otimes R_{\cdot}(x))N'(\mathrm{r})+R'(x)
					\cr&&-(R_{\prec}(x)\otimes\id \otimes \id)P'(\mathrm{r})
					-(L_{\succ}(x)\otimes \id \otimes \id)Q'(\mathrm{r})=0,
				\end{array}
			\end{eqnarray}
		\end{subequations}
		where
		\begin{eqnarray*}
			M(\mathrm{r})=\mathrm{r}_{_{\prec, 23}}\cdot \mathrm{r}_{_{\succ, 12}}+
			\mathrm{r}_{_{\prec, 21}}\prec \mathrm{r}_{_{\prec, 13}}-
			\mathrm{r}_{_{\succ, 13}} \succ \mathrm{r}_{_{\prec, 23}},\;%\cr
			P(\mathrm{r})=\mathrm{r}_{_{\succ, 12}}\cdot \mathrm{r}_{_{\prec, 23}}+
			\mathrm{r}_{_{\prec, 13}}\succ \mathrm{r}_{_{\prec, 21}}-
			\mathrm{r}_{_{\prec, 23}} \prec \mathrm{r}_{_{\succ, 13}},\cr
			N(\mathrm{r})=
			\mathrm{r}_{_{\succ, 32}}\cdot \mathrm{r}_{_{\prec, 21}}+
			\mathrm{r}_{_{\prec, 31}}\succ \mathrm{r}_{_{\prec, 23}}-
			\mathrm{r}_{_{\prec, 21}} \prec \mathrm{r}_{_{\succ,31}},\;%\cr
			Q(\mathrm{r})=
			\mathrm{r}_{_{\prec , 21}}\cdot \mathrm{r}_{_{\succ, 32}}+ 
			\mathrm{r}_{_{\prec, 23}}\prec \mathrm{r}_{_{\prec, 31}}-
			\mathrm{r}_{_{\succ, 31}} \succ \mathrm{r}_{_{\prec, 21}},
		\end{eqnarray*}
		\begin{eqnarray*}
			M'(\mathrm{r})&=&
			\mathrm{r}_{_{\succ, 23}}\prec \mathrm{r}_{_{\succ, 12}}+
			\mathrm{r}_{_{\succ, 21}}\succ \mathrm{r}_{_{\succ, 13}}-
			\mathrm{r}_{_{\succ, 13}}\cdot \mathrm{r}_{_{\succ,23}}+%\cr&+&
			\mathrm{r}_{_{\succ, 23}}\succ (\mathrm{r}_{_{\prec, 12}}+\mathrm{r}_{_{\succ, 12}})+
			(\mathrm{r}_{_{\prec, 21}}+\mathrm{r}_{_{\succ, 21}})\prec \mathrm{r}_{_{\succ, 13}},\cr
			N'(\mathrm{r})&=&
			\mathrm{r}_{_{\succ, 12}}\succ \mathrm{r}_{_{\succ, 23}}+
			\mathrm{r}_{_{\succ, 13}}\prec \mathrm{r}_{_{\succ, 21}}-
			\mathrm{r}_{_{\succ,23}}\cdot \mathrm{r}_{_{\succ, 13}}+
			(\mathrm{r}_{_{\prec, 12}}+\mathrm{r}_{_{\succ, 12}})\prec \mathrm{r}_{_{\succ, 23}}+
			\mathrm{r}_{_{\succ, 13}}\succ (\mathrm{r}_{_{\prec, 21}}+\mathrm{r}_{_{\succ, 21}}),\cr
			P'(\mathrm{r})&=&
			\mathrm{r}_{_{\prec, 32}}\prec \mathrm{r}_{_{\prec, 21}}+
			\mathrm{r}_{_{\prec, 31}}\cdot \mathrm{r}_{_{\succ, 23}}-
			\mathrm{r}_{_{\succ, 21}}\succ \mathrm{r}_{_{\prec , 31}}-
			(\mathrm{r}_{_{\prec, 21}}+\mathrm{r}_{_{\succ, 21}})\prec \mathrm{r}_{_{\prec, 31}},\cr
			Q'(\mathrm{r})&=&
			\mathrm{r}_{_{\prec, 21}}\succ \mathrm{r}_{_{\prec, 32}}+ 
			\mathrm{r}_{_{\succ, 23}}\cdot \mathrm{r}_{_{\prec, 31}}
			-\mathrm{r}_{_{\prec,31}}\prec \mathrm{r}_{_{\succ, 21}}-
			\mathrm{r}_{_{\prec,31}}\succ (\mathrm{r}_{_{\prec, 21}}+\mathrm{r}_{_{\succ, 21}}),
		\end{eqnarray*}
		and
		\begin{eqnarray*}
			R'(x)&=&[(\id\otimes (R_{_{\prec}}(x)+L_{_{\succ}}(x))\otimes \id)\mathrm{r}_{_{\prec, 32}}]
			\succ(\mathrm{r}_{_{\prec, 12}}+\mathrm{r}_{_{\succ, 12}})\cr
			&-&[((L_{_{\prec}}(x)+R_{_{\succ}}(x))\otimes\id \otimes \id)\mathrm{r}_{_{\prec, 31}}]
			\succ(\mathrm{r}_{_{\prec, 21}}+\mathrm{r}_{_{\succ, 21}}).
		\end{eqnarray*}
	\end{pro}
	\begin{proof}
		Let $x\in A$.
		Setting $\displaystyle \mathrm{r}_{_\succ}=\sum_i a_i\otimes b_i, \;\; 
		\mathrm{r}_{_\prec}=\sum_i c_i\otimes d_i$	we have
		\begin{eqnarray*}
			(\Delta_{_\succ}\otimes\id )\Delta_{_{\prec}}(x)&=&
			\sum_i\{a_{j}\otimes(c_i\cdot b_j)\otimes (x\succ d_{i})+
			(d_j\prec c_i)\otimes c_j\otimes (x\succ d_i)\cr&+&
			a_j\otimes ((b_i\cdot x)\cdot b_j)\otimes a_i+
			(d_j\prec (b_i\cdot x))\otimes c_j\otimes a_i\}\cr
			&=&
			(\id\otimes\id\otimes L_{\succ}(x))(\mathrm{r}_{_{\prec, 23}}\cdot \mathrm{r}_{_{\succ, 12}}+
			\mathrm{r}_{_{\prec, 21}}\prec \mathrm{r}_{_{\prec, 13}})\cr
			&+&\sum_i\{a_j\otimes ((b_i\cdot x)\cdot b_j)\otimes a_i+
			(d_j\prec (b_i\cdot x))\otimes c_j\otimes a_i\}
			\cr
			(\id\otimes\Delta_{_\prec})\Delta_{_{\succ}}(x)&=&
			\sum_i\{a_i\otimes c_j\otimes ((x\cdot b_i)\prec d_j)+
			a_i\otimes (b_j\cdot (x\cdot b_i))\otimes a_j\cr&+&
			(d_i\prec x)\otimes c_j\otimes (c_i\succ d_j)+
			(d_i\prec x)\otimes (b_j\cdot c_i)\otimes a_j\}\cr
			&=&  
			(R_{\prec}(x)\otimes \id\otimes\id)(\mathrm{r}_{_{\prec, 31}}\succ \mathrm{r}_{_{\prec, 23}}+
			\mathrm{r}_{_{\succ, 32}}\cdot \mathrm{r}_{_{\prec, 21}})\cr
			&+&\sum_i \{a_i\otimes c_j\otimes ((x\cdot b_i)\prec d_j)+
			a_i\otimes (b_j\cdot (x\cdot b_i))\otimes a_j\}\cr
			(\id\otimes\sigma\Delta_{_\succ})\sigma\Delta_{_{\prec}}(x)&=&
			\sum_i\{(x\succ d_i)\otimes (c_i\cdot b_j)\otimes a_j+
			(x\succ d_i)\otimes c_j\otimes (d_j\prec c_i)\cr&+&
			a_i\otimes ((b_i\cdot x)\cdot b_j)\otimes a_j+
			a_i\otimes c_j\otimes (d_j\prec (b_i\cdot x))\}\cr&=&
			(L_{\succ}(x)\otimes\id \otimes \id)(\mathrm{r}_{_{\prec , 21}}\cdot \mathrm{r}_{_{\succ, 32}}+ 
			\mathrm{r}_{_{\prec, 23}}\prec \mathrm{r}_{_{\prec, 31}})\cr 
			&+&\sum_i\{a_i\otimes ((b_i\cdot x)\cdot b_j)\otimes a_j+
			a_i\otimes c_j\otimes (d_j\prec (b_i\cdot x))\}\cr
			(\sigma\Delta_{_{\prec}}\otimes\id)\sigma\Delta_{_{\succ}}(x)&=&
			\sum_i\{((x\cdot b_i)\succ d_j)\otimes c_j\otimes a_i+
			a_j\otimes (b_j\cdot (x\cdot b_i))\otimes a_i\cr 
			&+&(c_i\succ d_j)\otimes c_j\otimes (d_i \prec x)+
			a_j\otimes (b_j\cdot c_i)\otimes (b_i\prec x)\}\cr
			&=&
			(\id\otimes \id \otimes R_{\prec}(x))(\mathrm{r}_{_{\prec, 13}}\succ \mathrm{r}_{_{\prec, 21}}+
			\mathrm{r}_{_{\succ, 12}}\cdot \mathrm{r}_{_{\prec, 23}})\cr
			&+&\sum_i\{((x\cdot b_i)\succ d_j)\otimes c_j\otimes a_i+
			a_j\otimes (b_j\cdot (x\cdot b_i))\otimes a_i\}
		\end{eqnarray*}
		\begin{eqnarray*}
			((\Delta_{_{\prec}}+\Delta_{_{\succ}})\otimes\id)\Delta_{\succ}(x)&=&
			\sum_i\{c_j\otimes (a_i\succ d_j)\otimes (x\cdot b_i)+
			(b_j\cdot a_i)\otimes a_j\otimes (x\cdot b_i)\cr&+&
			a_j\otimes (a_i\cdot b_j)\otimes (x\cdot b_i)
			+(d_j\prec a_i)\otimes c_j\otimes(x\cdot b_i)\cr&+&
			c_j\otimes ((d_i\prec x)\succ d_j)\otimes c_i+
			(b_j\cdot(d_i\prec x))\otimes a_j\otimes c_i\cr&+&
			a_j\otimes ((d_i\prec x)\cdot b_j)\otimes c_i+
			(d_j\prec (d_i\prec x))\otimes c_j\otimes c_i\}\cr 
			&=&
			(\id \otimes\id\otimes L_{\cdot}(x))(\mathrm{r}_{_{\succ, 23}}\succ \mathrm{r}_{_{\prec, 12}}+
			\mathrm{r}_{_{\succ, 21}}\cdot \mathrm{r}_{_{\succ, 13}}\cr&+&
			\mathrm{r}_{_{\succ, 23}}\cdot \mathrm{r}_{_{\succ, 12}}+
			\mathrm{r}_{_{\prec, 21}}\prec \mathrm{r}_{_{\succ, 13}})\cr 
			&+&\sum_i\{c_j\otimes ((d_i\prec x)\succ d_j)\otimes c_i+
			(b_j\cdot(d_i\prec x))\otimes a_j\otimes c_i\cr&+&
			a_j\otimes ((d_i\prec x)\cdot b_j)\otimes c_i+
			(d_j\prec (d_i\prec x))\otimes c_j\otimes c_i\}\cr 
			(\id\otimes\Delta_{\succ})\Delta_{\succ}(x)&=&
			\sum_i\{a_i\otimes a_j\otimes ((x\cdot b_i)\cdot b_j)
			+a_i\otimes (d_j\prec (x\cdot b_i))\otimes c_j\cr
			&+&(d_i\prec x)\otimes a_j\otimes (c_i\cdot b_j)+
			(d_i\prec x)\otimes (d_j\prec c_i)\otimes c_j\}\cr 
			&=&
			(R_{\prec}(x)\otimes\id \otimes \id)(\mathrm{r}_{_{\prec, 31}}\cdot \mathrm{r}_{_{\succ, 23}}+
			\mathrm{r}_{_{\prec, 32}}\prec \mathrm{r}_{_{\prec, 21}})\cr
			&+&\sum_i\{a_i\otimes a_j\otimes ((x\cdot b_i)\cdot b_j)
			+a_i\otimes (d_j\prec (x\cdot b_i))\otimes c_j\}\cr
			(\id\otimes \sigma\Delta_{_{\prec}})\sigma\Delta_{_\prec}(x)&=& 
			\sum_i\{(x\succ d_i)\otimes (c_i\succ d_j)\otimes c_j+
			(x\succ d_i)\otimes a_j\otimes (b_j\cdot c_i)\cr 
			&+&a_i\otimes ((b_i\cdot x)\succ d_j)\otimes c_j+
			a_i\otimes a_j\otimes (b_j\cdot(b_i\cdot x))\}\cr
			&=&(L_{\succ}(x)\otimes \id \otimes \id)(\mathrm{r}_{_{\prec, 21}}\succ \mathrm{r}_{_{\prec, 32}}+
			\mathrm{r}_{_{\succ, 23}}\cdot \mathrm{r}_{_{\prec, 31}})\cr 
			&+& \sum_i\{a_i\otimes ((b_i\cdot x)\succ d_j)\otimes c_j+
			a_i\otimes a_j\otimes (b_j\cdot(b_i\cdot x))\}\cr
			(\sigma(\Delta_{_{\prec}}+\Delta_{_{\succ}})\otimes\id)\sigma\Delta_{_\prec}(x)&=&
			\sum_i\{((x\succ d_i)\succ d_j)\otimes c_j\otimes c_i+
			a_j\otimes (b_j\cdot(x\succ d_i))\otimes c_i\cr 
			&+&
			((x\succ d_i)\cdot b_j)\otimes a_j\otimes c_i+
			c_j\otimes (d_j\prec(x\succ d_i))\otimes c_i\cr
			&+&
			(a_i\succ d_j)\otimes c_j\otimes (b_i\cdot x)+
			a_j\otimes (b_j\cdot a_i)\otimes (b_i\cdot x)
			\cr&+&
			(a_i\cdot b_j)\otimes a_j\otimes (b_i\cdot x)+
			c_j\otimes (d_j\prec a_i)\otimes (b_i\cdot x)\}\cr
			&=&
			(\id \otimes \id \otimes R_{\cdot}(x))(\mathrm{r}_{_{\succ, 13}}\succ \mathrm{r}_{_{\prec, 21}}+
			\mathrm{r}_{_{\succ, 12}}\cdot \mathrm{r}_{_{\succ, 23}}\cr&+&
			\mathrm{r}_{_{\succ, 13}}\cdot \mathrm{r}_{_{\succ, 21}}+
			\mathrm{r}_{_{\prec, 12}}\prec \mathrm{r}_{_{\succ, 23}})\cr
			&+&
			\sum_i\{((x\succ d_i)\succ d_j)\otimes c_j\otimes c_i+
			a_j\otimes (b_j\cdot(x\succ d_i))\otimes c_i\cr 
			&+&
			((x\succ d_i)\cdot b_j)\otimes a_j\otimes c_i+
			c_j\otimes (d_j\prec(x\succ d_i))\otimes c_i\}
		\end{eqnarray*}
		Then
		\begin{eqnarray*}
			&&(\Delta_{_\succ}\otimes\id )\Delta_{_{\prec}}(x)-
			(\id\otimes\Delta_{_\prec})\Delta_{_{\succ}}(x)-
			(\id\otimes\sigma\Delta_{_\succ})\sigma\Delta_{_{\prec}}(x)+
			(\sigma\Delta_{_{\prec}}\otimes\id)\sigma\Delta_{_{\succ}}(x)=A1(x)+A2(x)\cr&&+
			(\id\otimes\id\otimes L_{\succ}(x))(\mathrm{r}_{_{\prec, 23}}\cdot \mathrm{r}_{_{\succ, 12}}+
			\mathrm{r}_{_{\prec, 21}}\prec \mathrm{r}_{_{\prec, 13}})
			-(R_{\prec}(x)\otimes \id\otimes\id)(\mathrm{r}_{_{\prec, 31}}\succ \mathrm{r}_{_{\prec, 23}}+
			\mathrm{r}_{_{\succ, 32}}\cdot \mathrm{r}_{_{\prec, 21}})\cr &&
			+(\id\otimes \id \otimes R_{\prec}(x))(\mathrm{r}_{_{\prec, 13}}\succ \mathrm{r}_{_{\prec, 21}}+
			\mathrm{r}_{_{\succ, 12}}\cdot \mathrm{r}_{_{\prec, 23}})
			-(L_{\succ}(x)\otimes\id \otimes \id)(\mathrm{r}_{_{\prec , 21}}\cdot \mathrm{r}_{_{\succ, 32}}+
			\mathrm{r}_{_{\prec, 23}}\prec \mathrm{r}_{_{\prec, 31}})
		\end{eqnarray*}
		where
		\begin{eqnarray*}
			A1(x)&=&\sum_i\{ a_j\otimes ((b_i\cdot x)\cdot b_j+b_j\cdot (x\cdot b_i))\otimes a_i-
			a_i\otimes (b_j\cdot (x\cdot b_i)+(b_i\cdot x)\cdot b_j)\otimes a_j\}\\
			A2(x)&=&\sum_i\{(d_j\prec (b_i\cdot x)+(x\cdot b_i)\succ d_j)\otimes c_j\otimes a_i -
			a_i\otimes c_j\otimes ((x\cdot b_i)\prec d_j+d_j\prec (b_i\cdot x)) \}\\
		\end{eqnarray*}
		By exchanging $i$ and $j$, and using  Remark~\ref{rmk_1}~\eqref{rmk_flex}, 
		we have $A1(x)=0$.
		
		Using Eqs.~\eqref{eq:pre-antiflexible2} we have  
		\begin{eqnarray*}
			A2(x)&=&(L_{\succ}(x)\otimes\id \otimes\id) 
			(\mathrm{r}_{_{\succ, 31}} \succ \mathrm{r}_{_{\prec, 21}})
			-(\id \otimes \id \otimes L_{\succ}(x))
			(\mathrm{r}_{_{\succ, 13}} \succ \mathrm{r}_{_{\prec, 23}})
			\cr&+&
			(R_{\prec}(x)\otimes\id \otimes\id) (\mathrm{r}_{_{\prec, 21}} \prec \mathrm{r}_{_{\succ,31}})-
			(\id \otimes \id \otimes R_{\prec}(x))(\mathrm{r}_{_{\prec, 23}} \prec \mathrm{r}_{_{\succ, 13}}).
		\end{eqnarray*}
		Besides, 
		\begin{eqnarray*}
			&&((\Delta_{_{\prec}}+\Delta_{_{\succ}})\otimes\id)\Delta_{\succ}(x)-
			(\id\otimes\Delta_{\succ})\Delta_{\succ}(x)
			-(\id\otimes \sigma\Delta_{_{\prec}})\sigma\Delta_{_\prec}(x)
			+(\sigma(\Delta_{_{\prec}}+
			\Delta_{_{\succ}})\otimes\id)\sigma\Delta_{_\prec}(x)\cr
			&=&(\id \otimes\id\otimes L_{\cdot}(x))(\mathrm{r}_{_{\succ, 23}}\succ \mathrm{r}_{_{\prec, 12}}
			+\mathrm{r}_{_{\succ, 21}}\cdot \mathrm{r}_{_{\succ, 13}}+
			\mathrm{r}_{_{\succ, 23}}\cdot \mathrm{r}_{_{\succ, 12}}+
			\mathrm{r}_{_{\prec, 21}}\prec \mathrm{r}_{_{\succ, 13}})+B(x)
			\cr &+&(\id \otimes \id \otimes R_{\cdot}(x))(\mathrm{r}_{_{\succ, 13}}\succ \mathrm{r}_{_{\prec, 21}
			}+\mathrm{r}_{_{\succ, 12}}\cdot \mathrm{r}_{_{\succ, 23}}+
			\mathrm{r}_{_{\succ, 13}}\cdot \mathrm{r}_{_{\succ, 21}}+
			\mathrm{r}_{_{\prec, 12}}\prec \mathrm{r}_{_{\succ, 23}})+
			\cr&-&(R_{\prec}(x)\otimes\id \otimes \id)(\mathrm{r}_{_{\prec, 31}}\cdot \mathrm{r}_{_{\succ, 23}
			}+\mathrm{r}_{_{\prec, 32}}\prec \mathrm{r}_{_{\prec, 21}})
			-(L_{\succ}(x)\otimes \id \otimes \id)(\mathrm{r}_{_{\prec, 21}}\succ \mathrm{r}_{_{\prec, 32}}+ 
			\mathrm{r}_{_{\succ, 23}}\cdot \mathrm{r}_{_{\prec, 31}}),
		\end{eqnarray*}
		where $B(x)=B1(x)+B2(x)+B3(x)+B4(x)+B5(x)$ with
		$$
		B1(x)=-\sum_i\{a_i\otimes a_j\otimes ((x\cdot b_i)\cdot b_j+
		b_j\cdot(b_i\cdot x))\},\;\;
		B2(x)=\sum_i\{c_j\otimes ((d_i\prec x)\succ d_j+
		d_j\prec(x\succ d_i))\otimes c_i\},
		$$
		$$
		B3(x)=\sum_i\{((x\succ d_i)\succ d_j+
		d_j\prec(d_i\prec x))\otimes c_j\otimes c_i\},\;\;
		B4(x)=\sum_i\{ (b_j\cdot(d_i\prec x)+(x\succ d_i)\cdot b_j)\otimes a_j\otimes c_i\}
		$$
		\begin{eqnarray*}
			B5(x)=\sum_i\{a_j\otimes (b_j\cdot(x\succ d_i)+(d_i\prec x)\cdot b_j)\otimes c_i-
			a_i\otimes ((b_i\cdot x)\succ d_j+d_j\prec (x\cdot b_i))\otimes c_j\}.
		\end{eqnarray*}
		Considering Remark~\ref{rmk_1}~\eqref{rmk_flex} we have 
		\begin{eqnarray*}
			B1(x)=-(\id \otimes \id \otimes L_{\cdot}(x))
			(\mathrm{r}_{_{\succ, 13}}\cdot \mathrm{r}_{_{\succ,23}})-
			(\id \otimes \id \otimes R_{\cdot}(x))(\mathrm{r}_{_{\succ,23}}\cdot \mathrm{r}_{_{\succ, 13}}).
		\end{eqnarray*}
		Using Eq.~\eqref{eq_bimodule_pre_anti_flexible5} we have 
		\begin{eqnarray*}
			B3(x)&=&(x\succ (d_i\succ d_j)+(d_j\prec d_i)\prec x)\otimes c_j\otimes c_i-
			((x\prec d_i)\succ d_j+d_j\prec (d_i\succ x))\otimes c_j\otimes c_i\cr
			&=&(L_{_{\succ}}(x)\otimes\id \otimes\id)(\mathrm{r}_{_{\prec,31}}\succ \mathrm{r}_{_{\prec, 21}})+
			(R_{_{\prec}}(x)\otimes\id \otimes\id)(\mathrm{r}_{_{\prec, 21}}\prec \mathrm{r}_{_{\prec, 31}})
			\cr&-&((x\prec d_i)\succ d_j+d_j\prec (d_i\succ x))\otimes c_j\otimes c_i.
		\end{eqnarray*}
		By Eqs.~\eqref{eq_bimodule_pre_anti_flexible2} and \eqref{eq_bimodule_pre_anti_flexible5} we have
		\begin{eqnarray*}
			B4(x)&=& (b_j\succ(d_i\prec x)+(x\succ d_i)\prec b_j   )  \otimes a_j\otimes c_i+
			(b_j\prec(d_i\prec x)+(x\succ d_i)\succ b_j   )  \otimes a_j\otimes c_i\cr
			&=&(L_{_{\succ}}(x)\otimes\id \otimes\id)(\mathrm{r}_{_{\prec, 31}}\prec \mathrm{r}_{_{\succ, 21}}+
			\mathrm{r}_{_{\prec, 31}}\succ \mathrm{r}_{_{\succ, 21}})\cr
			&+&
			(R_{_{\prec}}(x)\otimes\id \otimes\id)(\mathrm{r}_{_{\succ, 21}}\succ \mathrm{r}_{_{\prec , 31}}+ 
			\mathrm{r}_{_{\succ, 21}}\prec \mathrm{r}_{_{\prec, 31}})\cr
			&-&(b_j\prec(d_i\succ x)+(x\prec d_i)\succ b_j   )  \otimes a_j\otimes c_i.
		\end{eqnarray*}
		Furthermore, we have
		\begin{eqnarray*}
			B5(x)&=&\sum_i\{a_j\otimes (b_j\cdot(x\succ d_i)+(d_i\prec x)\cdot b_j)\otimes c_i-
			a_i\otimes ((b_i\cdot x)\succ d_j+d_j\prec (x\cdot b_i))\otimes c_j\}\cr
			&=&\sum_i\{a_j\otimes (b_j\prec (x\succ d_i)+b_j\succ (x\succ d_i))\otimes c_i+
			a_j\otimes ((d_i\prec x)\prec b_j+(d_i\prec x)\succ b_j) \otimes c_i\cr 
			&-&a_i\otimes ((b_i\cdot x)\succ d_j+d_j\prec (x\cdot b_i))\otimes c_j\}
			=\sum_i\{a_j\otimes (b_j\succ (x\succ d_i)+(d_i\prec x)\prec b_j)\otimes c_i\cr
			&-&a_i\otimes ((b_i\cdot x)\succ d_j+d_j\prec (x\cdot b_i))\otimes c_j\}+
			\sum_i\{a_j\otimes (b_j\prec (x\succ d_i)+(d_i\prec x)\succ b_j)\otimes c_i\}\cr
			B5(x)&=&\sum_i\{a_j\otimes (b_j\prec (x\succ d_i)+(d_i\prec x)\succ b_j)\otimes c_i\}
		\end{eqnarray*}
		The last equal sign in above equation is due to 
		Eq.~\eqref{eq:pre-antiflexible2} and changing indices 
		$i$ to $j$ in the last term of the  first summation.
		Finally, we have 
		\begin{eqnarray*}
			R'(x)&=&\sum_i\{c_j\otimes ((d_i\prec x)\succ d_j+d_j\prec(x\succ d_i))\otimes c_i-
			((x\prec d_i)\succ d_j+d_j\prec (d_i\succ x))\otimes c_j\otimes c_i\cr
			&-&(b_j\prec(d_i\succ x)+(x\prec d_i)\succ b_j   )  \otimes a_j\otimes c_i+
			a_j\otimes (b_j\prec (x\succ d_i)+(d_i\prec x)\succ b_j)\otimes c_i\}\cr
			R'(x)&=&
			[(\id\otimes R_{_{\prec}}(x)\otimes \id)\mathrm{r}_{_{\prec, 32}}]
			\succ(\mathrm{r}_{_{\prec, 12}}+\mathrm{r}_{_{\succ, 12}})
			+
			[(\id\otimes L_{_{\succ}}(x)\otimes \id)\mathrm{r}_{_{\prec, 32}}]
			\prec(\mathrm{r}_{_{\prec, 12}}+\mathrm{r}_{_{\succ, 12}})\cr
			&-&
			[(L_{_{\prec}}(x)\otimes\id \otimes \id)\mathrm{r}_{_{\prec, 31}}]
			\succ(\mathrm{r}_{_{\prec, 21}}+\mathrm{r}_{_{\succ, 21}})
			-
			[(R_{_{\succ}}(x)\otimes\id \otimes \id)\mathrm{r}_{_{\prec, 31}}]
			\prec (\mathrm{r}_{_{\prec, 21}}+\mathrm{r}_{_{\succ, 21}}).
		\end{eqnarray*}
		Therefore, hold the equivalences.
	\end{proof}
	\begin{rmk}
		Considering the flipping map "$flp$" defined on $A\otimes A$ such that for any elements  
		$\mathrm{r}, \mathrm{r}'\in A\otimes A$, 
		$flp(\mathrm{r}\prec \mathrm{r}' )=
		\mathrm{r}\succ  \mathrm{r}'$, 
		$flp( \mathrm{r}\succ  \mathrm{r}'=
		\mathrm{r}\prec  \mathrm{r}'$, we establish finally
		the following relations
		$
		P(\mathrm{r})=flp(M(\mathrm{r})),\;
		N(\mathrm{r})=\sigma_{13}(flp(M(\mathrm{r}))),\;
		Q(\mathrm{r})=flp( \sigma_{13} (flp(M(\mathrm{r})))),
		$
		with $\sigma_{13}(x\otimes y\otimes z)=z\otimes y\otimes x$ 
		for any $x,y,z \in A$. Besides, 
		we also have  $N'(\mathrm{r})=flp(M'(\mathrm{r}))$ and
		$Q'(\mathrm{r})=flp(P'(\mathrm{r}))$.
	\end{rmk}
	\begin{pro}
		Let $(A, \prec, \succ)$ be a pre-anti-flexible algebra and 
		$\mathrm{r}_{_\prec}, \mathrm{r}_{_\succ} \in A\otimes A$. Define 
		the linear maps
		$\Delta_{_{\prec}}, \Delta_{_{\succ}}:A\rightarrow A\otimes A$ by 
		Eq.~\eqref{eq:coboundary}.
		Then $\Delta_{_\prec}^*, \Delta_{_{\succ}}^*: A^*\otimes A^*\rightarrow A^*$ 
		define a pre-anti-flexible algebra structure on $A^*$
		if and only if the following equations are satisfied for any $x\in A$
		\begin{subequations}
			\begin{eqnarray}\label{eq:ybe1'}
				&&((\id\otimes\id\otimes L_{\succ}(x))-
				(R_{\prec}(x)\otimes \id\otimes\id)\sigma_{13}\circ flp
				+(\id\otimes \id \otimes R_{\prec}(x))flp
				\cr&&-(L_{\succ}(x)\otimes\id \otimes \id)flp\circ  
				\sigma_{13} \circ flp)M(\mathrm{r})=0,
			\end{eqnarray}
			\begin{eqnarray}\label{eq:ybe2'}
				\begin{array}{llll}
					&&((\id \otimes\id\otimes L_{\cdot}(x))
					+(\id \otimes \id \otimes R_{\cdot}(x))flp)M'(\mathrm{r})
					\cr&&-((R_{\prec}(x)\otimes\id \otimes \id)
					+(L_{\succ}(x)\otimes \id \otimes \id)flp)P'(\mathrm{r})+R'(x)=0,
				\end{array}
			\end{eqnarray}
		\end{subequations}
		where
		$M(\mathrm{r})=\mathrm{r}_{_{\prec, 23}}\cdot \mathrm{r}_{_{\succ, 12}}+
		\mathrm{r}_{_{\prec, 21}}\prec \mathrm{r}_{_{\prec, 13}}-
		\mathrm{r}_{_{\succ, 13}} \succ \mathrm{r}_{_{\prec, 23}}$,
		\begin{eqnarray*}
			M'(\mathrm{r})&=&
			\mathrm{r}_{_{\succ, 23}}\prec \mathrm{r}_{_{\succ, 12}}+
			\mathrm{r}_{_{\succ, 21}}\succ \mathrm{r}_{_{\succ, 13}}-
			\mathrm{r}_{_{\succ, 13}}\cdot \mathrm{r}_{_{\succ,23}}+
			\mathrm{r}_{_{\succ, 23}}\succ (\mathrm{r}_{_{\prec, 12}}+
			\mathrm{r}_{_{\succ, 12}})+(\mathrm{r}_{_{\prec, 21}}+
			\mathrm{r}_{_{\succ, 21}})\prec \mathrm{r}_{_{\succ, 13}},\cr
			P'(\mathrm{r})&=&
			\mathrm{r}_{_{\prec, 32}}\prec \mathrm{r}_{_{\prec, 21}}+
			\mathrm{r}_{_{\prec, 31}}\cdot \mathrm{r}_{_{\succ, 23}}-
			\mathrm{r}_{_{\succ, 21}}\succ \mathrm{r}_{_{\prec , 31}}-
			(\mathrm{r}_{_{\prec, 21}}+\mathrm{r}_{_{\succ, 21}})\prec \mathrm{r}_{_{\prec, 31}},\cr
			R'(x)&=&
			[(\id\otimes (R_{_{\prec}}(x)+L_{_{\succ}}(x))\otimes \id)
			\mathrm{r}_{_{\prec, 32}}]\succ(\mathrm{r}_{_{\prec, 12}}+\mathrm{r}_{_{\succ, 12}})\cr
			&-&
			[((L_{_{\prec}}(x)+R_{_{\succ}}(x))\otimes\id \otimes \id)
			\mathrm{r}_{_{\prec, 31}}]\succ(\mathrm{r}_{_{\prec, 21}}+\mathrm{r}_{_{\succ, 21}}).
		\end{eqnarray*}
	\end{pro}
	\begin{thm}\label{thm_coboundary}
		Let $(A, \prec, \succ)$ be a pre-anti-flexible algebra and 
		$\mathrm{r}_{_\prec}, \mathrm{r}_{_\succ} \in A\otimes A$. Define the linear maps
		$\Delta_{_{\prec}}, \Delta_{_{\succ}}:A\rightarrow A\otimes A$ by 
		Eq.~\eqref{eq:coboundary}.
		Then  $(A, A^*)$ is a pre-anti-flexible bialgebra 
		if and only if $\mathrm{r}_{_\prec}, \mathrm{r}_{_\succ}$ satisfy 
		Eqs.~\eqref{eq:coboundary1}~-~\eqref{eq:coboundary4}, 
		Eqs.~\eqref{eq:ybe1'} and \eqref{eq:ybe2'}.
	\end{thm}
	In view of symmetries brought out by equations characterizing 
	pre-anti-flexible bialgebras provided by the linear maps given 
	by Eq.~\eqref{eq:coboundary}, i.e. 
	Eqs.~\eqref{eq:coboundary1}~-~\eqref{eq:coboundary4}, 
	Eqs.~\eqref{eq:ybe1'} and \eqref{eq:ybe2'}, we will consider pre-anti-flexible balgebras 
	generated by $\mathrm{r}\in A\otimes A$ in the following cases. 
	\begin{itemize}
		\item[Case1]
		\begin{eqnarray}\label{eq:particular1-r}
			\mathrm{r}_{_{\prec}}=\mathrm{r},\;\; \mathrm{r}_{_{\succ}}=
			-\sigma \mathrm{r}, \quad \mathrm{r}\in A\otimes A.
		\end{eqnarray}
		\begin{cor}
			Let $(A, \prec, \succ)$ be a pre-anti-flexible algebra and 
			$\mathrm{r}\in A\otimes A$. Then the maps $\Delta_{_{\prec}}, \Delta_{_{\succ}}$
			defined by Eq.~\eqref{eq:coboundary} with $\mathrm{r}_{_{\succ}}, \mathrm{r}_{_{\prec}}$ given by
			Eq.~\eqref{eq:particular1-r} induce a pre-anti-flexible algebra structure on $A^*$
			such that  $(A, A^*)$ is a pre-anti-flexible bialgebra if and only if 
			$\mathrm{r}$ satisfies the following equations 
			\begin{subequations}
				\begin{eqnarray}\label{eq:A}
					\begin{array}{lll}
						&&(\id\otimes(L_{_{\succ}}(x\prec y)+R_{_{\prec}}(y\succ x))+
						((L_{_{\succ}}(y\prec x)+R_{_{\prec}}(x\succ y))\otimes \id) 
						(\mathrm{r}-\sigma \mathrm{r})\cr
						&&-
						(R_{_{\prec}}(y)\otimes L_{_{\succ}}(x) +
						L_{_{\succ}}(y)\otimes R_{_{\prec}}(x) )
						(\mathrm{r}-\sigma \mathrm{r})=0.
					\end{array}
				\end{eqnarray}
				\begin{eqnarray}\label{eq:B}
					&&0=((R_{_{\prec}}(x)\otimes R_{_{\prec}}(x)+ 
					L_{_{\succ}}(x)\otimes L_{_{\succ}}(y))+(L_{_{\succ}}(y)\otimes L_{_{\succ}}(x) +
					R_{_{\prec}}(y)\otimes R_{_{\prec}}(x))
					)(\mathrm{r}-\sigma \mathrm{r})\cr
					&&+
					((R_{_{\prec}}(x)L_{_{\succ}}(y)+L_{_{\succ}}(x)R_{_{\prec}}(y))\otimes \id+
					(\id\otimes(L_{_{\succ}}(x)R_{_{\prec}}(y)+
					R_{_{\prec}}(x)L_{_{\succ}}(y))))(\sigma \mathrm{r}- \mathrm{r})
				\end{eqnarray}
				\begin{eqnarray}\label{eq:C}
					&&((\id\otimes\id\otimes L_{\succ}(x))-
					(R_{\prec}(x)\otimes \id\otimes\id)\sigma_{13}\circ flp
					+(\id\otimes \id \otimes R_{\prec}(x))flp\cr&&-
					(L_{\succ}(x)\otimes\id \otimes \id)flp\circ  \sigma_{13} \circ flp)M_1(\mathrm{r})=0,
				\end{eqnarray}
				\begin{eqnarray}\label{eq:D}
					\begin{array}{llll}
						&&((\id \otimes\id\otimes L_{\cdot}(x))
						+(\id \otimes \id \otimes R_{\cdot}(x))flp)M_1'(\mathrm{r})+R_1'(x)
						\cr&&-((R_{\prec}(x)\otimes\id \otimes \id)
						+(L_{\succ}(x)\otimes \id \otimes \id)flp)P_1'(\mathrm{r})=0,
					\end{array}
				\end{eqnarray}
			\end{subequations}
			where $x,y\in A$, 
			\begin{eqnarray*}
				M_1(\mathrm{r})&=&-\mathrm{r}_{_{23}}\cdot \mathrm{r}_{_{21}}+
				\mathrm{r}_{_{21}}\prec \mathrm{r}_{_{13}}+
				\mathrm{r}_{_{31}} \succ \mathrm{r}_{_{23}}\cr 
				M_1'(\mathrm{r})&=&
				\mathrm{r}_{_{32}}\prec \mathrm{r}_{_{21}}+
				\mathrm{r}_{_{12}}\succ  \mathrm{r}_{_{31}}-
				\mathrm{r}_{_{31}}\cdot \mathrm{r}_{_{32}}-
				\mathrm{r}_{_{32}}\succ (\mathrm{r}_{_{12}}- \mathrm{r}_{_{21}})-
				(\mathrm{r}_{_{21}}- \mathrm{r}_{_{12}})\prec \mathrm{r}_{_{31}},\cr
				P_1'(\mathrm{r})&=&\mathrm{r}_{_{32}}\prec \mathrm{r}_{_{21}}-
				\mathrm{r}_{_{31}}\cdot  \mathrm{r}_{_{32}}+
				\mathrm{r}_{_{12}}\succ \mathrm{r}_{_{31}}-
				(\mathrm{r}_{_{21}}- \mathrm{r}_{_{12}})\prec \mathrm{r}_{_{31}}, \; \; \mbox{ and }\cr
				R_1'(x)&=&[(\id\otimes (R_{_{\prec}}(x)+
				L_{_{\succ}}(x))\otimes \id)\mathrm{r}_{_{\prec, 32}}]
				\succ(\mathrm{r}_{_{12}}-\mathrm{r}_{_{21}})\cr&-&
				[((L_{_{\prec}}(x)+R_{_{\succ}}(x))\otimes\id \otimes \id)\mathrm{r}_{_{31}}]
				\succ(\mathrm{r}_{_{21}}-\mathrm{r}_{_{12}}).
			\end{eqnarray*}
		\end{cor}
		\begin{rmk}
			It is straight to identify $M_1'(\mathrm{r})=
			P_1'(\mathrm{r})-\mathrm{r}_{_{32}}\succ (\mathrm{r}_{_{12}}- \mathrm{r}_{_{21}})$ and 
			setting in addition  
			$\sigma_{_{123}}: A\otimes A\otimes A \rightarrow A\otimes A\otimes A$, by 
			$\sigma_{_{123}}(x\otimes y\otimes z)=z\otimes y\otimes x$, we have
			\begin{eqnarray*}
				M_1'(\mathrm{r})=\sigma_{_{123}}(M_1(\mathrm{r}))-
				\mathrm{r}_{_{32}}\succ (\mathrm{r}_{_{12}}- \mathrm{r}_{_{21}})-
				(\mathrm{r}_{_{21}}- \mathrm{r}_{_{12}})\prec \mathrm{r}_{_{31}}.
			\end{eqnarray*} 
			Besides, if in addition $\mathrm{r}$ commutes then $R_1'(x)=0$ and 
			Eqs.~\eqref{eq:A} and \eqref{eq:B} 
			are satisfied and finally $\mathrm{r}$ satisfies the following equation
			\begin{eqnarray}\label{eq:AFPYBE}
				\mathrm{r}_{_{23}}\cdot \mathrm{r}_{_{12}}=
				\mathrm{r}_{_{12}}\prec \mathrm{r}_{_{13}}+
				\mathrm{r}_{_{13}} \succ \mathrm{r}_{_{23}}.
			\end{eqnarray}
		\end{rmk}
		\item[Case2]	
		\begin{eqnarray}\label{eq:particular2-r}
			\mathrm{r}_{_{\prec}}+\mathrm{r}_{_{\succ}}=0,\; 
			\mathrm{r}_{_{\succ}}=\mathrm{r}, \quad \mathrm{r}\in A\otimes A.
		\end{eqnarray}
		\begin{cor}
			Let $(A, \prec, \succ)$ be a pre-anti-flexible algebra and 
			$\mathrm{r}\in A\otimes A$. Then the maps $\Delta_{_{\prec}}, \Delta_{_{\succ}}$
			defined by Eq.~\eqref{eq:coboundary} with 
			$\mathrm{r}_{_{\succ}}, \mathrm{r}_{_{\prec}}$ given by
			Eq.~\eqref{eq:particular2-r} induce a pre-anti-flexible algebra structure on $A^*$
			such that  $(A, A^*)$ is a pre-anti-flexible bialgebra if and only if 
			$\mathrm{r}$ satisfies the following equations for any $x,y\in A$
			\begin{subequations}
				\begin{eqnarray}\label{eq:A'}
					(R_{_\prec}(y)\otimes L_{\cdot}(x)+L_{\succ}(y)\otimes R_{\cdot}(x))(\mathrm{r}-
					\sigma \mathrm{r})=0,
				\end{eqnarray}
				\begin{eqnarray}\label{eq:B'}
					(L_{\succ}(x)\otimes L_{\cdot}(y)-R_{\prec}(y)\otimes R_{\cdot}(x)-
					L_{\succ}(y)\otimes L_{\cdot}(x)+R_{\prec}(x)\otimes R_{\cdot}(y))(\mathrm{r}-
					\sigma \mathrm{r})=0,
				\end{eqnarray}
				\begin{eqnarray}\label{eq:C'}
					(R_{_{\succ}}(y)\otimes L_{_{\succ}}(x) +
					L_{_{\prec}}(y) \otimes R_{_{\prec}}(x) )(\mathrm{r}-\sigma \mathrm{r}),
				\end{eqnarray}
				\begin{eqnarray}\label{eq:D'}
					(R_{_{\prec}}(x)\otimes R_{_{\succ}}(y)+
					L_{_{\succ}}(x)\otimes L_{_{\prec}}(y)+
					L_{_{\prec}}(y)\otimes L_{_{\succ}}(x)+
					R_{_{\succ}}(y)\otimes R_{_{\prec}}(x))(\mathrm{r}-\sigma \mathrm{r})=0,
				\end{eqnarray}
				\begin{eqnarray}\label{eq:E'}
					&&((\id\otimes\id\otimes L_{\succ}(x))-
					(R_{\prec}(x)\otimes \id\otimes\id)\sigma_{13}\circ flp
					+(\id\otimes \id \otimes R_{\prec}(x))flp
					\cr&&-(L_{\succ}(x)\otimes\id \otimes \id)flp\circ  
					\sigma_{13} \circ flp)M_2(\mathrm{r})=0,
				\end{eqnarray}
				\begin{eqnarray}\label{eq:F'}
					\begin{array}{llll}
						&&((\id \otimes\id\otimes L_{\cdot}(x))
						+(\id \otimes \id \otimes R_{\cdot}(x))flp)M_2'(\mathrm{r})
						\cr&&-((R_{\prec}(x)\otimes\id \otimes \id)
						+(L_{\succ}(x)\otimes \id \otimes \id)flp)P_2'(\mathrm{r})=0,
					\end{array}
				\end{eqnarray}
			\end{subequations}
			where
			\begin{eqnarray*}
				M_2(\mathrm{r})&=&-\mathrm{r}_{_{23}}\cdot \mathrm{r}_{_{12}}+
				\mathrm{r}_{_{21}}\prec \mathrm{r}_{_{13}}+
				\mathrm{r}_{_{13}} \succ \mathrm{r}_{_{23}},\cr
				M_2'(\mathrm{r})&=&
				-\mathrm{r}_{_{13}}\cdot \mathrm{r}_{_{23}}+
				\mathrm{r}_{_{23}}\prec \mathrm{r}_{_{12}}+
				\mathrm{r}_{_{21}}\succ \mathrm{r}_{_{13}},\cr 
				P_2'(\mathrm{r})&=&
				-\mathrm{r}_{_{31}}\cdot \mathrm{r}_{_{23}}+
				\mathrm{r}_{_{32}}\prec \mathrm{r}_{_{21}}+
				\mathrm{r}_{_{21}}\succ \mathrm{r}_{_{31}}.
			\end{eqnarray*}
		\end{cor}
		\begin{rmk}
			Clearly, we have $P_2'(\mathrm{r})=\sigma_{_{123}}(M_2(\mathrm{r}))$ 
			and if $\mathrm{r}$  commutes then 
			$P_2'(\mathrm{r})=M_2'(\mathrm{r})$ and Eqs.~\eqref{eq:A'}~-~\eqref{eq:D'} 
			are satisfied and finally
			$\mathrm{r}$ satisfied Eq.~\eqref{eq:AFPYBE}.
		\end{rmk}
	\end{itemize}
	We finally deduct the following pre-anti-flexible bialgebras 
	provided by a given $\mathrm{r}\in A\otimes A$ possessing 
	some internal symmetries while browsing $A\otimes A$.
	\begin{cor}
		Let $(A, \prec, \succ)$ be a pre-anti-flexible algebra and 
		consider symmetric element  $\mathrm{r}\in A\otimes A$
		satisfying Eq.~\eqref{eq:AFPYBE}. Then the linear maps 
		$\Delta_{_{\prec}}, \Delta_{_{\succ}}$
		defined by Eq.~\eqref{eq:coboundary} with $\mathrm{r}_{_{\succ}}=\mathrm{r}$ 
		and $\mathrm{r}_{_{\prec}}=-\mathrm{r}$ induce a 
		pre-anti-flexible algebra structure on $A^*$ such that 
		$(A, A^*)$ is a pre-anti-flexible bialgebra.  
	\end{cor}
	\begin{defi}
		Let $(A, \prec, \succ)$ be a  pre-anti-flexible algebra and $\mathrm{r} \in A \otimes A$. 
		The Eq.~\eqref{eq:AFPYBE} is called the
		\textbf{pre-anti-flexible Yang-Baxter equation} (PAFYBE) in $(A, \prec, \succ)$.
	\end{defi}
	\begin{rmk}
		We due the notion of pre-anti-flexible Yang-Baxter equation in pre-anti-flexible algebras
		as  an analogue of the anti-flexible Yang-Baxter
		equation in  anti-flexible algebras (\cite{DBH3}) or
		classical Yang-Baxter equation in   Lie algebras (\cite{Drinfeld}) or
		the associative Yang-Baxter equation in   associative algebras (\cite{Aguiar, Bai_Double})  
		and $\mathcal{D}$-equation in dendriform algebras (\cite{Bai_Double}).
		
		For no other specific reason than which showing that 
		both dendriform   and pre-anti-flexible algebras possessing the 
		same shape of  dual bimodules (see Remark~\ref{rmk_useful}~\eqref{dual-bimodule}), 
		to our amazement, $\mathcal{D}$-equation 
		in dendriform algebras and PAFYBE in  pre-anti-flexible algebras 
		own the same form translated by Eq.~\eqref{eq:AFPYBE}.
		This could making parallel with associative Yang-Baxter equation in associative algebras 
		and anti-flexible Yang-Baxter equation in anti-flexible algebras (\cite{DBH3}).
	\end{rmk}
	\section{Solutions of the pre-anti-flexible Yang-Baxter equation}\label{section5}
	Let $A$ be a vector space. For any $\mathrm{r} \in A \otimes  A$, $\mathrm{r}$ 
	can be regarded as a linear map $\mathrm{r}:A^*\rightarrow A$
	in the following way:
	\begin{equation*}%\label{eq:r}
		\langle \mathrm{r}, u^*\otimes v^*\rangle=
		\langle \mathrm{r}(u^*), v^* \rangle, \;\; \forall u^*, v^*\in A^*.
	\end{equation*}
	As PAFYBE in  pre-anti-flexible algebras have the same form of the
	$\mathcal{D}$-equation in dendriform algebra, we omitted proofs (too similar 
	to the case of dendriform algebra and related $\mathcal{D}$-equation) of the following 
	in which $(A, \prec, \succ)$ is a pre-anti-flexible algebra.
	\begin{pro}\label{pro_pre_anti_flexible dual}
		For a given  $\mathrm{r}\in A\otimes A$, $r$  is a symmetric solution of 
		the PAFYBE in $A$ if and only if for any $x\in A$ and any $a, b\in A^*$
		\begin{eqnarray}\label{eq:pre-anti-flexible-dual}
			\begin{array}{llllllllll}
				a\prec b&=&-R_{_{\succ}}^*(\mathrm{r}(a))b+L_{\cdot}^*(\mathrm{r}(b))a,\; 
				a\succ b= R^*_{\cdot}(\mathrm{r}(a))b-L_{_{\prec}}^*(\mathrm{r}(b))a,\;\cr
				a\cdot  b&=& a\prec b+a\succ b = 
				R^*_{_{\prec}}(\mathrm{r}(a))b+L_{_{\succ}}^*(\mathrm{r}(b))a, \;\cr
				x\prec a&=& x\prec \mathrm{r}(a)+\mathrm{r}(R^*_{_{\succ}}(x)a)-R^*_{_{\succ}}(x)a,\;
				x\succ a=x\succ \mathrm{r}(a)-\mathrm{r}(R^*_{\cdot}(x)a)+R_{\cdot}^*(x)a,  \cr
				x\cdot a&=&x\cdot \mathrm{r}(a)-R^*_{_{\prec}}(x)a+R^*_{_{\prec}}(x)a,\; 
				a\cdot x= \mathrm{r}(a)\cdot x-\mathrm{r}(L^*_{_{\succ}}(x)a)+L^*_{_{\succ}}(x)a,  \cr
				a\prec x&=&\mathrm{r}(a)\prec x- \mathrm{r}(L^*_{\cdot}(x)a)+L^*_{\cdot}(x)a, \;
				a\succ x=\mathrm{r}(a)\succ x+ \mathrm{r}(L^*_{_{\prec}}(x)a)-L^*_{_{\prec}}(x)a.
			\end{array}
		\end{eqnarray}
	\end{pro} 
	\begin{thm}
		Consider a symmetric and non-degenerate element $\mathrm{r}\in A\otimes A$.
		Then $\mathrm{r}$ is a solution of the PAFYBE in $A$ if and only if 
		the inverse homomorhpism $A^*\rightarrow A$ induced by $\mathrm{r}$ 
		regarded as a bilinear form $\mathfrak{B}$ on $A$
		(i.e. $\mathfrak{B}(x, y)= \langle \mathrm{r}^{-1}(x), y\rangle$, 
		for any $x,y\in A$) and satisfies 
		\begin{equation}\label{eq:2-cocycle}
			\mathfrak{B}(x\cdot y, z)=\mathfrak{B}(y, z\prec x)+
			\mathfrak{B}(x, y\succ z), \mbox{ for any } x,y,z\in A.
		\end{equation}
	\end{thm}
	\begin{cor}
		Let $\mathrm{r}\in A\otimes A$ be a symmetric solution of PAFYBE in $A$. 
		Suppose in addition 
		by "$\prec_{_{  A^*}}, \succ_{_{ A^*}}$" the pre-anti-flexible algebra structure on $A^*$
		induced by $r$ via Proposition~\ref{pro_pre_anti_flexible dual}. Then we have for 
		any $a,b\in A^*$
		\begin{equation*}
			a\prec_{_{  A^*}} b=\mathrm{r}^{-1}(\mathrm{r}(a)\prec_{_  A} \mathrm{r}(b)), \;
			a\succ_{_{ A^*}} b=\mathrm{r}^{-1}(\mathrm{r}(a)\succ_{_  A} \mathrm{r}(b)).
		\end{equation*}
		Therefore, $\mathrm{r}:A^*\rightarrow A$ is an isomorphism of pre-anti-flexible algebras.
	\end{cor}
	\begin{thm}
		Let $(A, \prec, \succ)$ be a pre-anti-flexible algebra and
		$\mathrm{r}\in A\otimes A$ symmetric.
		Then, $\mathrm{r}$ is a solution of {PAFYBE} if and only if its satisfies 
		\begin{equation*}
			\mathrm{r}(a)\cdot\mathrm{r}(b)=\mathrm{r}(R^*_{_{\prec}}(\mathrm{r}(a))b+
			L^*_{_{\succ}}(\mathrm{r}(b))a ), \;\forall a,b \in A^*.
		\end{equation*}
	\end{thm}
	Recall that a  $\mathcal{O}$-operator related to the bimodule $(l, r, V )$ 
	of an anti-flexible algebra $(A, \cdot)$ is a linear map $T :V\rightarrow A$ satisfies
	\begin{equation*}
		T (u) \cdot T (v) = T (l(T (u))v + r(T (v))u), \; \forall u, v \in  V.
	\end{equation*}
	In addition, for a given pre-anti-flexible algebra $(A, \prec, \succ)$, according 
	to Proposition~\ref{prop_operation_bimodule_pre_anti_flexible}~\eqref{eq:one},
	$(L_{_{\succ}},R_{_{\prec}},  A)$ is a bimodule of 
	its underlying anti-flexible algebra $aF(A)$. Furthermore, for any $x,y\in A$
	\begin{eqnarray}\label{eq:o-operator}
		\id(x)\cdot \id(y)= \id (L_{_{\succ}}(\id(x))y+R_{_{\prec}}(\id(y))x),
	\end{eqnarray}
	then $\id:A\rightarrow A$ is an $\mathcal{O}$-operator of $aF(A)$ associated to
	the bimodule $(L_{_{\succ}}, R_{_{\prec}}, A)$.
	\begin{cor}
		Consider a symmetric element $\mathrm{r}\in A\otimes A$.
		Then $\mathrm{r}$ is a solution
		of PAFYBE in $A$ if and only if it is an $\mathcal{O}$-operator
		of the underlying anti-flexible $aF(A)$ associated to the bimodule
		$(R^*_{_{\prec}}, L^*_{_{\succ}}, A^*)$. Furthermore, there is a pre-anti-flexible 
		algebra structure on $A^*$ given by
		\begin{eqnarray*}
			a\prec b=L^*_{_{\succ}}(\mathrm{r}(b))a;\;\;
			a\succ b=R^*_{_{\prec}}(\mathrm{r}(a))b;\; \forall a,b\in A^*,
		\end{eqnarray*}
		which is the same of that associated to the pre-anti-flexible bialgebra derived on $A^*$
		by Eq~\eqref{eq:pre-anti-flexible-dual}. If in addition $\mathrm{r}$ is 
		non degenerate, then there is a new compatible pre-anti-flexible algebraic
		structure given on $A$ by 
		\begin{eqnarray*}
			x\prec'y=\mathrm{r}(L^*_{_{\succ}}(y)\mathrm{r}^{-1}(x)),\;
			x\succ'y=\mathrm{r}(R^*_{_{\prec}}(x)\mathrm{r}^{-1}(y)), \;\forall x,y\in A,
		\end{eqnarray*}
		which is the pre-anti-flexible algebra structure given by
		\begin{eqnarray*}
			\mathfrak{B}(x\prec'y, z)=\mathfrak{B}(x, y\ast z), \;
			\mathfrak{B}(x\succ'y, z)=\mathfrak{B}(y, z\cdot x), \;\forall x,y,z\in A,
		\end{eqnarray*}
		where  $\mathfrak{B}$ is given by $
		\mathfrak{B}(x,y)=\langle \mathrm{r}^{-1}(x), y\rangle$ for
		any $x,y\in A$ and satisfies Eq.~\eqref{eq:2-cocycle}.
	\end{cor}
	Taking into account \cite[Proposition 2.7.]{DBH3} we have
	\begin{thm}
		Let $(A, \cdot)$ be an anti-flexible algebra, $(l,r, V)$  a bimodule
		of $(A, \cdot)$ and $T:V\rightarrow A$ an $\mathcal{O}$-operator associated 
		to $(l,r, V)$. Then $\mathrm{r}=T+\sigma T$ is a symmetric solution of 
		the PAFYBE in $T(V)\ltimes_{r^*, 0,0, l^*} V^*$, where
		$T(V)\subset A$ is endowed with  a pre-anti-flexible given by 
		for any $u,v\in V, $
		\begin{eqnarray*}
			T(v)\prec T(v)=T(r(T(v))u), \; 
			T(u)\succ T(v)=T(l(T(u))v),
		\end{eqnarray*} 
		such that $(r^*, 0,0, l^*, T(V)^*)$ is its associated bimodule 
		and underlying anti-flexible algebra is a sub-algebra of $A$, and finally
		$T$ can be identified with an element in 
		$T(V)\otimes V^*\subset(T(V)\ltimes_{r^*, 0,0, l^*} V^*)\otimes 
		T(V)\ltimes_{r^*, 0,0, l^*} V^*$.
	\end{thm}
	Considering the above theorem,
	Proposition~\ref{prop_operation_bimodule_pre_anti_flexible}~\eqref{eq:one} and 
	Eq.~\eqref{eq:o-operator}, we have
	\begin{cor}
		Let $(A, \prec, \succ)$ be a $n$-dimensional pre-anti-flexible algebra. Then the element 
		\begin{eqnarray*}
			\mathrm{r}=\sum_{i}^{n} (e_i\otimes e_i^*+e_i^*\otimes e_i)
		\end{eqnarray*}
		is a symmetric solution of PAFYBE in 
		$A\ltimes_{R_{_{\prec}}^*, 0,0, L_{_{\succ}}^*} A^*$, 
		where $\{e_1, \cdots , e_n\}$ it a basis of $A$  and 
		$\{e^*_1, \cdots , e^*_n\}$  its associated dual basis.
		Furthermore, $\mathrm{r}$ is non degenerate and it 
		induced  bilinear form $\mathfrak{B}$ on
		$A\ltimes_{R_{_{\prec}}^*, 0,0, L_{_{\succ}}^*} A^*$
		is given by 
		\begin{eqnarray*}
			\mathfrak{B}(x+a, y+b)=\langle x, b\rangle+\langle y, a\rangle, \;
			\forall x,y\in A, a, b\in A^*.
		\end{eqnarray*} 
	\end{cor}
	\bigskip
	\noindent
	{\bf Acknowledgments.}   
	The author thanks Professor C. Bai for helpful discussions and 
	his encouragement, and Nankai ZhiDe Foundation.

\end{document}